\documentclass[a4paper]{amsart}
\usepackage[utf8]{inputenc}
\usepackage[english]{babel}
\usepackage{fancyhdr}
\usepackage{amsmath}
\usepackage{amsfonts,amstext}
\usepackage{amssymb,amsthm,amscd,amsxtra,wasysym,graphicx}
\usepackage{mathrsfs}
\usepackage{xcolor}

\usepackage{hyperref}
\hypersetup{
	unicode=false,          
	pdftoolbar=true,        
	pdfmenubar=true,        
	pdffitwindow=false,     % window fit to page when opened
	pdfstartview={FitH},    % fits the width of the page to the window
	pdftitle={CMAF in big cohomology classes},    % title
	pdfauthor={Quang-Tuan Dang},     % author
	colorlinks=true,       % false: boxed links; true: colored links
	linkcolor=blue,          % color of internal links
	citecolor=blue,        % color of links to bibliography
	filecolor=black,      % color of file links
	urlcolor=blue}           % color of external links

\usepackage{cleveref}
\usepackage{indentfirst}
\newtheorem{theorem}{Theorem}
\newtheorem{lemma}[theorem]{Lemma}

\newtheorem{proposition}[theorem]{Proposition}
\newtheorem{definition}[theorem]{Definition}
\newtheorem{bigthm}{Theorem}

\newtheorem{defprop}[theorem]{Definition/Proposition}

\theoremstyle{definition}
\newtheorem{remark}[theorem]{Remark}
\numberwithin{theorem}{section}
\numberwithin{equation}{section}
 
\DeclareMathOperator \PSH {{\rm PSH}}
\DeclareMathOperator \loc {{\rm loc}}

\DeclareMathOperator \MA {{\rm MA}}

\DeclareMathOperator \Vol {{\rm Vol}}

\def\C{\mathbb{C}}
\def\R{\mathbb{R}}

\def\e{\varepsilon}
\def\f{\varphi}
\def\dc{dd^c}

\begin{document}
	
	\title[Pluripotential Monge-Amp\`ere flows]{Pluripotential Monge-Amp\`ere Flows\\ in Big Cohomology Classes}
	\author{Quang-Tuan Dang}
	\address{Laboratoire de Math\'ematiques D'Orsay, Universit\'e Paris-Saclay, CNRS, 91405 Orsay, France}
	\email{quang-tuan.dang@universite-paris-saclay.fr}
	\curraddr{Institut de Mathematiques de Toulouse,  Universit\'e de Toulouse; CNRS,
		118 route de Narbonne, 31400 Toulouse, France}
	\email{quang-tuan.dang@math.univ-toulouse.fr}
	\date{\today}
	\thanks{This work is partially supported by the ANR projects GRACK and PARAPLUI}

	\begin{abstract}
		We study pluripotential complex Monge-Amp\`ere flows in big cohomology classes on compact K\"ahler manifolds. We use the Perron method, considering pluripotential subsolutions to the Cauchy problem.  We prove that, under natural assumptions on the data, the upper envelope of all subsolutions  is continuous in space and semi-concave in time, and provides a unique pluripotential solution with such regularity. We apply this theory to study pluripotential K\"ahler-Ricci flows on compact K\"ahler manifolds of general type as well as on stable varieties %with semi-log canonical singularities.   %As an application, we prove that, on a compact K\"ahler manifold of general type the  K\"ahler-Ricci flow, starting from any positive closed current, can be uniquely continued after the finite time singularity as a weak solution and eventually converges to the unique singular K\"ahler-Einstein metric.
	\end{abstract}
	
	\maketitle

	\tableofcontents

	\section*{Introduction} 
	
	The primary goal of this paper is to study pluripotential complex Monge-Amp\`ere flows
	motivated by the Minimal Model Program (MMP for brevity) in algebraic geometry, whose aim is the (birational) classification of projective
	manifolds. In a recent celebrated work, Birkar-Cascini-Hacon-Mckernan~\cite{birkar2010existence} showed the existence of
	minimal models for a large class of varieties which are called varieties of general type. 
	J. Song and G. Tian \cite{song2012canonical,song2017kahler} have recently proposed an analytic analogue making use of (twisted) K\"ahler-Ricci flows on compact K\"ahler manifolds. %The latter requires to develop a good theory of weak K\"ahler-Ricci flows. 
	%The first steps of a parabolic pluripotential theory have been developed by Guedj-Lu-Zeriahi \cite{GLZ1,GLZ2}, allowing us to define a unique weak K\"ahler-Ricci flow on compact K\"ahler varieties with Kawamata log terminal singularities. In this paper we aim at developing this theory further, extending it to the most general singularities encountered in the MMP, and studying the geometric convergence of  Monge-Amp\`ere flows.
	
	Let $X$ be a compact K\"ahler manifold of dimension $n$ equipped with a K\"ahler form $\hat{\omega}$. The (normalized) K\"ahler-Ricci flow on $X$ starting at $\hat{\omega}$  is the solution to the following evolution equation
	\begin{align}\label{nkrf0}
	\dfrac{\partial\theta_t}{\partial t}=-{\rm Ric}(\theta_t) -\lambda \theta_t ,\quad\theta|_{t=0}=\hat{\omega},
	\end{align}
	where the sign of $\lambda\in \R$ depends on that of the first Chern class $c_1(K_X)$. Solving the normalized K\"ahler-Ricci flow~\eqref{nkrf0} turns out to be equivalent to solving the scalar complex Monge-Amp\`ere flow
	\begin{align*}
	\begin{cases}
	(\omega_t+\dc\f_t)^n=e^{\partial_t\f_t+\lambda\f_t+h(t,x)}dV\\
	\omega_{t}+\dc\f_t>0,
	\end{cases}
	\end{align*}
	where $h$ is a smooth density, and $\omega_t\in\{\theta_t\}\in H^{1,1}(X,\R)$ is fixed.

	Since the MMP requires one to work on singular varieties, it is necessary to develop a fine theory dealing with weak solutions. One has indeed to deal with similar complex Monge-Amp\`ere flows with various degeneracies: the reference forms $\omega_t$ are no longer K\"ahler and the densities $h$ is no longer smooth, with integrability properties that depend on the type of singularities.	
	A parabolic viscosity approach has been developed recently in~\cite{eyssidieux2016weak}, which requires the densities to be continuous hence has a limited scope of applications. The first elements of a parabolic pluripotential theory  has been laid down in~\cite{guedj2018pluripotential,guedj2020pluripotential} which are the parabolic analogues of the pioneering work of Bedford and Taylor in the local case \cite{bedford1976dirichlet,bedford1982new}. We extend here this theory so as to be able to deal  with big cohomology classes.  %This theory allows to deal with $L^p$-densities, $p>1$.   
	
	%\vskip0.5cm
	\subsection*{Assumptions and Notations.} Before going further and stating the
	main results of the paper, let us fix some notations. Let $X$ be a compact K\"ahler manifold of dimension $n$.
	We let $X_T:=(0,T)\times X$ denote the real $(2n+1)$-dimensional manifold with $T\in(0,+\infty]$. We focus mostly on finite time intervals i.e. $T<+\infty$. The parabolic boundary of $X_T$ is denoted by
	$$\partial X_T:=\{0\}\times X.$$
	We fix $\theta$ a smooth closed $(1,1)$-form representing a big cohomology class. We let $\Omega$ denote the ample locus of $\theta$,
	$$\Omega:={\rm Amp}(\theta)$$
	which is a non empty Zariski open subset of $X$. We also set $\Omega_T:=(0,T)\times \Omega$.
	
	We assume that $(\omega_t)_{t\in[0,T)}$ is a smooth family of closed $(1,1)$-forms on $X$ such that
	$$g(t)\theta\leq\omega_t,\quad
	\forall\, t\in [0,T),$$ 
	where $g(t)$ is an increasing smooth positive function on $[0,T]$.
	
	Throughout the article we assume  that there exists a K\"ahler form $\Theta$ such that 
	\begin{align}\label{estimate_ome}
	-\Theta\leq \omega_t,\dot{\omega_t},\ddot{\omega_t}\leq \Theta.\end{align}
	We let $dV$ denote a smooth volume form on $X$. We shall always assume that
	\begin{itemize}
		\item $0\leq f\in L^p(X,dV)$ for some $p>1$, and $f$ is strictly positive almost everywhere;
		\item $F:[0,T]\times X\times \R\rightarrow\R$ is  continuous on $[0,T]\times X\times \mathbb{R}$;   
		\item the function $r\mapsto F(\cdot,\cdot,r)$ is  increasing in $r$;
		\item the function $F$ is  uniformly Lipschitz in $(t,x)\in[0,T]\times \mathbb{R}$, i.e. there exists a constant $\kappa_F>0$  such that for all $t,t'\in [0,T]$, $x\in X$, $r,r'\in\mathbb{R}$,
		\begin{align*}
		|F(t,x,r)-F(t',x,r')|\leq \kappa_F(|t-t'|+|r-r'|);
		\end{align*}
		\item the function $(t,r)\mapsto F(t,\cdot,r)$ is convex.
	\end{itemize}

%	The purpose of this paper is to extend the results of~\cite{} to big cohomology classes.  
With the assumptions above, we consider the complex Monge-Ampère flow:
	\begin{align*}\label{cmaf} \tag{CMAF}
	dt\wedge(\omega_t+\dc\f_t)^n=e^{\dot{\f}_t+F(t,\cdot,\f_t)}fdV\wedge dt
	\end{align*}
	on $X_T$. 
	Note that the equation~\eqref{cmaf} should be understood in the weak sense of measures in $(0,T)\times \Omega$ (see Section~\ref{subsol}).
	
	%	For complex Monge-Amp\`ere flows, both theories that we mentioned above (\cite{EGZ16}, \cite{GLZ2}) have been studied in the case where the cohomology class $\{\omega_t\}$ is big and semipositive. For further applications, we need to extend these theories in the general case where the class $\{\omega_t\}$  is merely big. 
	%	
	%	In two papers \cite{GLZ1,GLZ2}, the authors have introduced two methods to construct solution of complex Mong-Amp\`ere flows \eqref{cmaf}. One is approximate \eqref{cmaf} here by smooth complex Monge-Amp\`ere flows and establish a priori estimates, the smooth solutions of these flows exist (cf. \cite{Tos18}). The other one is Perron method for Cauchy-Dirichlet problem, the solution is the upper envelope of all subsolutions. 
	%	
	%	In our case where the class $\{\omega_t\}$ is merely big, we can not apply approximation method. Therefore, we shall study a Cauchy problem for complex Monge-Amp\`ere flows \eqref{cmaf} with initial data $\f_0$ (see Definition \ref{defcp}).
	
	The existence of the weak K\"ahler-Ricci flow is often proved by using approximation arguments and a priori estimates (cf.~\cite{song2017kahler,guedj2020pluripotential}). Big cohomology classes can not be approximated by K\"ahler ones so this approach  breaks down in our case. We shall instead use the Perron method, inspired by~\cite{guedj2018pluripotential}, considering the upper envelope $U$ of all pluripotential subsolutions to the Cauchy problem.  We prove  that this upper envelope is locally uniformly semi-concave in time:

	\begin{bigthm}\label{A}
		Let $\f_0$ be a $\omega_{0}$-psh function with minimal singularities.
		Then the upper envelope $U$ of all subsolutions to~\eqref{cmaf} with initial data $\f_0$ is a pluripotential solution to~\eqref{cmaf} which is locally uniformly Lipschitz and locally uniformly semi-concave in $t\in(0,T)$.     
	\end{bigthm} 
	
	We prove Theorem~\ref{A} by following the arguments of~\cite{guedj2018pluripotential} in the local context:
	\begin{itemize}
		\item we first show that the upper envelope of all subsolutions  is locally uniformly Lipschitz in $t$  (Theorem~\ref{thmlips})
		and that it is itself a pluripotential subsolution; 
		\item  we then show that the envelope is locally uniformly semi-concave (Theorem \ref{thmscc}); 
		\item we finally apply a balayage process and use the analogue result in the local context~\cite{guedj2018pluripotential} to conclude the proof.
	\end{itemize}

	We prove in Theorem~\ref{thm: min sing} that the envelope $U$ in Theorem~\ref{A} has minimal singularities and is continuous in $(0,T)\times \Omega$ under an extra assumption:
	\begin{equation}
	\label{eq: min sing cond}
	\dot{\omega}_t \leq A \omega_t, \; t\in [0,T),
	\end{equation}
	for some positive constant $A$. 
	We also show that  $U$ is the unique pluripotential solution with such regularity by establishing the following comparison principle:

	\begin{bigthm}\label{B} 
		%Assume that \eqref{eq: min sing cond} holds. 
		Let $\f$ (resp. $\psi$) be a pluripotential subsolution (resp. supersolution) to \eqref{cmaf} with initial data $\f_0$ (resp. $\psi_0$).  We assume that $\psi$ is locally uniformly semi-concave in $t\in (0,T)$ and $\psi$ is continuous in $(0,T)\times \Omega$. We assume moreover that for each $t$, $\psi_t$ has  minimal singularities.
		Then $\f\leq \psi$ on $[0,T)\times X$ if $\f_0\leq\psi_0$.	
	\end{bigthm}

	The assumption that $\psi_t$ has minimal singularities means that for each $t\in (0,T)$, there exists a constant $C_t$ such that $|\psi_t-V_{\omega_t}|$ is bounded by $C_t$, where $V_{\omega_t}$ is the largest negative $\omega_t$-psh function. 
	The proof of Theorem~\ref{B} is provided in Section~\ref{sect: compa_princ}, generalizing some ideas from~\cite{guedj2020pluripotential}.

	\medskip
	
	Starting from a K\"ahler form $\omega_0$, it follows from \cite{cao1985deformation,tsuji1988existence,tian2006kahler} that the (smooth) normalized K\"ahler-Ricci flow exists in $[0,T)$ where 
	\begin{align*}
	T:=\sup\{t>0: e^{-t}\{\omega_{0}\}+(1-e^{-t})c_1(K_X) \,\, \textrm{ is K\"ahler}\}.
	\end{align*}
	The maximal existence time $T$ is finite unless $K_X$ is nef (numerically effective).
	
	It is an interesting question to know how to define the flow for $t>T$.  
	This  was formulated in~\cite[Section 10, Question 8]{feldman2003rotationally} and a precise conjecture was made in~\cite{boucksom2012semipositivity}. Note that, if $X$ is of general type, i.e.  $K_X$ is big, then for any $t>T$ the cohomology class $e^{-t}\{\omega_{0}\}+(1-e^{-t})c_1(K_X)$ remains big but are no longer nef, thus one can not hope to make sense of the flow in the classical one. It was proved in~\cite{to2021convergence} that the flow can be continued through  $T$ in the viscosity sense and it eventually converges to the unique singular K\"ahler-Einstein metric on $\textrm{Amp}(K_X)$.  
	Using the tools developed above, we establish the pluripotential analogue of the main result of~\cite{to2021convergence}:

	\begin{bigthm}\label{C}
		Let $X$ be a compact n-dimensional K\"ahler manifold of general type. Then the normalized pluripotential K\"ahler-Ricci flow emanating from a K\"ahler metric $\omega_0$,
		\begin{align*}
		\dfrac{\partial \theta_t}{\partial t}=-\textrm{Ric}(\theta_t)-\theta_t,  
		\end{align*}  
		exists for all time. It coincides with the smooth flow on $[0,T)$   and deforms $\omega_0$ towards the unique singular K\"ahler-Einstein metric $\omega_{KE}$ on $\textrm{Amp}(K_X)$, as $t\rightarrow+\infty$.
	\end{bigthm}

	We actually establish a more general result allowing to run the flow from an arbitrary closed positive current with bounded potential;
	(see Theorem~\ref{thm_conv}). We can also continue the pluripotential K\"ahler-Ricci flow for all time when $K_X$ is pseudoeffective (see Section~\ref{sect: singularity}).
	
	\smallskip
	
	In the last part of the paper we study pluripotential K\"ahler-Ricci flows on  K\"ahler varieties $X$ with semi-log canonical singularities (the most general class of singularities appearing in the log MMP) and ample canonical line bundle.
	
	It has been shown by R. Berman and H. Guenancia \cite{berman2014kahler} that $X$ admits  a unique K\"ahler-Einstein current $\omega_{KE}$ in the class $c_1(K_X)$ which is smooth in the regular locus $X_{\rm reg}$. We apply our theory to run the pluripotential (normalized) K\"ahler-Ricci flow on $X$ and recover the canonical metric $\omega_{KE}$ as the long time limit of the flow. 
	More precisely, we have the following:

	\begin{bigthm}\label{D}
		Let  $X$ be a projective complex algebraic variety with semi-log canonical singularities such that $K_X$ is ample. Then the normalized pluripotential K\"ahler-Ricci flow emanating from a K\"ahler metric $\omega_0$,
		\begin{align*}
		\dfrac{\partial \theta_t}{\partial t}=-\textrm{Ric}(\theta_t)-\theta_t,  
		\end{align*}  
		exists for all time. It  deforms $\omega_0$ towards the unique singular K\"ahler-Einstein metric $\omega_{KE}$ on $X_{\rm reg}$,
		as $t\rightarrow+\infty$.
	\end{bigthm}

	Again we actually show that the flow can be run from an arbitrary positive closed current with bounded potentials (see Theorem~\ref{thm_conv2}).
	
	For varieties of  general type with log terminal singularities 
	%(which in our notation means that $a_i>-1$) 
	the pluripotential K\"ahler-Ricci flow (with non continuous data) was constructed in \cite[Section 5.1]{guedj2020pluripotential}.
	A similar result has been obtained 
	in the recent work~\cite{chau2019kahler}, where the authors have extended the approach of Song-Tian~\cite{song2017kahler} to the case 
	of $\mathbb{Q}$-factorial projective varieties with log canonical singularities: establishing higher order a priori estimates,
	they obtain a good notion of weak K\"ahler-Ricci flow which is smooth in the regular locus of variety.
	\subsection*{Organization of the paper}
	In Section~\ref{prel} we provide some backgrounds on  pluripotential theory in big cohomology classes.
	%	 In section \ref{subsol} we define pluripotential sub/super solutions for complex Monge-Amp\`ere flows.
	In Section~\ref{sec:envelope} we study the regularity properties of the envelope of pluripotential subsolutions.
	In Section~\ref{exist} we shall prove Theorem~\hyperref[A]{A}	and Theorem~\ref{B}.
	We study in Section~\ref{applies} the pluripotential normalized K\"ahler-Ricci flow on compact K\"ahler manifolds of general type (resp. stable varieties)
	and prove Theorem~\ref{C}  (resp. Theorem~\ref{D}). 
	\subsection*{Acknowledgements.} The author would like to thank his advisors Vincent Guedj and Hoang-Chinh Lu for constant help and encouragement. We are truly grateful to Henri Guenancia and Ahmed Zeriahi for several interesting discussions. We thank T\^at-Dat T\^o for useful conversations on his results in~\cite{to2021convergence}.
	We also thank
	the referee for giving numerous valuable comments which really improved the presentation of the paper.	
	
	\section{Preliminaries}\label{prel}
	In this section we recall  necessary definitions and backgrounds.	
	Let $X$ be a compact K\"ahler manifold of complex dimension $n$, and $\Theta$ be a K\"ahler metric on $X$. We let $H^{1,1}(X,\R)$ denote the Bott-Chern cohomology of $d$-closed real $(1,1)$-forms (or currents) modulo $\partial\bar{\partial}$-exact ones. 
	
	\subsection{Monge-Amp\`ere operators in big cohomology classes}

	\subsubsection{Big cohomology classes} 	
	
	Let $\theta$ be a smooth real closed $(1,1)-$form on $X$. An upper semi-continuous function $\f:X\rightarrow[-\infty,+\infty)$ is called $\theta$-\textit{plurisubharmonic} ($\theta$-psh for short) if in any local holomorphic coordinates $\f$ can be written as the sum of a psh and a smooth function, and 
	$$\theta+\dc \f\geq 0,$$
	in the weak sense of currents, where $d=\partial+\bar{\partial}$ and $d^c=\frac{i}{2\pi}(\bar{\partial}-\partial)$. 
	
	We let $\PSH(X,\theta)$ denote the set of all $\theta$-psh functions on $X$ which are not identically $-\infty$. This set is endowed with the weak topology which coincides with the $L^1$-topology. By Hartogs' lemma $\f\mapsto\sup_X\f$ is continuous in the $L^1$-topology.
	
	By the $\dc$-lemma any closed positive $(1,1)$-current $T$ cohomologous to $\theta$ can be written as $T=\theta+\dc\f$ for some $\theta$-psh function $\varphi$ which is moreover unique up to an additive constant.

	If $T$ and $T'$ are two closed positive $(1,1)$-currents on $X$ which are cohomologous, then $T$ is said to be \textit{less singular} than $T'$ if their global potentials satisfy $\f'\leq \f+O(1)$ (then we also say that $\varphi$ is less singular than $\varphi'$).  A positive current $T$ is now said to have \emph{minimal singularities} if it is less singular than any other positive current in its cohomology class.	
	\begin{definition}
		A $\theta$-psh function $\f$ is said to have \emph{minimal singularities} if it is less singular than any other $\theta$-psh function on $X$. 
	\end{definition}	
	Such $\theta$-psh functions always exists, one can consider, following Demailly, the upper envelope
	$$V_{\theta}:=\sup\{\f:\f\in\PSH(X,\theta), \text{and}\,\, \f\leq 0 \}.$$
	Observe that $V_{\theta}^*$ is a $\theta$-psh function satisfying $V_{\theta}^*\leq V_{\theta}$, hence $V_{\theta}=V_{\theta}^*$ is a $\theta$-psh function with minimal singularities.
	
	The cohomology class $\alpha=\{\theta\}\in H^{1,1}(X,\R)$ is said to be \textit{big} if there exists a closed $(1,1)$-current 
	$$T_+=\theta+\dc\f_+,$$
	cohomologous to $\theta$ such that $T_+$ is \textit{strictly positive} i.e $T_+\geq \e_0\Theta$ for some constant $\e_0>0$. %We also call $T_+$ a K\"ahler current. 
	
	A function $u$ has {\it analytic singularities}  if it can locally be written as 
	$$u=\dfrac{c}{2}\log \sum_{j=1}^{N}|f_j^2|+h,$$
	where the functions $f_j$ are holomorphic, $h$ is smooth  and $c$ is a positive constant.

	In the sequel we always assume that the class $\alpha=\{\theta\}$ is big. 
	By Demailly's regularization theorem \cite{demailly1992regularization}, any $\theta$-psh function $u$ can be approximated from above by a sequence of $(\theta+\e_j\omega)-$psh functions $(u_j)$ with analytic singularities. Applying this to the potential $\f_+$ of a K\"ahler current $T_+=\theta+\dc\f_+$, one can moreover assume that the function $\f_+$ has analytic singularities. Such a current $T_+$ is then smooth on a Zariski open subset, 
	this motivates the following:
	
	\begin{definition}
		The \emph{ample locus} ${\rm Amp}(\alpha)$ of $\alpha$ is  the set of $x \in X$ such that there exists
		a K\"ahler current with analytic singularities which is smooth around $x$.
		
	\end{definition}
	
	It follows from  the Noetherian property of closed analytic subsets that  ${\rm Amp}(\alpha)$ is a Zariski open set.	
	Note that any $\theta$-psh function $\f$ with minimal singularities is locally bounded on the ample locus ${\rm Amp}(\alpha)$ since it has to satisfy $\f_+\leq \f+O(1)$. Moreover, $\f_+$ does not have minimal singularities unless $\alpha$ is a K\"ahler class (cf. \cite[Proposition 2.5]{boucksom2004divisorial}).
	
	By the above analysis, there exists a $\theta$-psh function $\chi$ on $X$ with analytic singularities such that, for some $\delta_0>0$,
	\begin{align}\label{chi}
	\theta+\dc\chi\geq 2\delta_0\Theta.
	\end{align}
	Subtracting a large constant, we can always assume that $\chi\leq 0$, thus $\chi\leq V_{\theta}$. Moreover, $\chi$ is smooth in the ample locus $\text{Amp}(\alpha)$, and $\chi(x)\rightarrow -\infty$ as $x\rightarrow\partial\Omega$ (cf.~\cite[Theorem 3.17]{boucksom2004divisorial}).

	\subsubsection{Full Monge-Amp\`ere mass}
	In~\cite{boucksom2010monge}, the authors defined the non-pluripolar product $T\mapsto\langle T^n \rangle$ of any closed positive $(1,1)$-current $T\in\alpha$, which
	is shown to be well-defined  as a positive measure on $X$ putting no mass on  pluripolar sets. In particular given a $\theta$-psh function $\f$, one can define its non-pluripolar Monge-Amp\`ere product by 
	$$
	\MA_\theta(\f):=\langle (\theta+\dc\f)^n\rangle.
	$$
	From now we denote the non-pluripolar Monge-Ampère product $(\theta+\dc\f)^n$ instead of $\langle (\theta+\dc\f)^n\rangle$.
	By definition the total mass of $\MA(\f)$ is less than or equal to the volume $\Vol(\alpha)$ of the class $\alpha$:
	$$\int_X\MA(\f)\leq \Vol(\alpha):= \int_X \MA(V_{\theta}).$$
	A particular class of $\theta$-psh functions that appears naturally is the one for which
	the last inequality is an equality. We will say that such functions (or the associated
	currents) have \emph{full Monge–Amp\`ere mass}. For example, $\theta$-psh functions with minimal
	singularities have full Monge–Amp\`ere mass (cf.~\cite[Theorem 1.16]{boucksom2010monge}), but the converse is not true.
	%	From now on we  use  $\MA_{\theta}(\f)=(\theta+\dc\f)^n$ instead of $\langle (\theta+\dc\f)^n\rangle$ to denote the non-pluripolar product.	
	
	%	When $f\in L^p(X,dV)$ for some $p>1$ the function $\rho$ can be found by solving a complex Monge-Amp\`ere equation as we now explain. 
	We let $c_1$ be the normalizing constant such that $2^ne^{c_1}fdV$ has total mass equal to $\Vol(\alpha)$.   We have the following:
	
	\begin{theorem}[{\cite[Theorem 4.1]{boucksom2010monge}}]
		There exists a unique $\theta$-psh function $\rho$ with full Monge-Amp\`ere mass such that
		\begin{align}\label{rho}
		(\theta+\dc\rho)^n=2^ne^{c_1}fdV,
		\end{align} and normalized by $\sup_X\rho=0$. Moreover, there exists a constant $M>0$ only depending on $\theta$, $dV$, and $p>1$ such that 
		\begin{align*}
		\rho\geq V_\theta-M\|f\|_p^{1/n}.
		\end{align*}
	\end{theorem}

	\subsection{Parabolic potentials}
	In this section we define the parabolic pluripotential objects in big cohomology classes necessary for our study.  These are mainly taken from \cite{guedj2018pluripotential,guedj2020pluripotential} but  we need to be more precise when dealing with unbounded functions. Let $\omega = (\omega_t)_{t\in [0,T)}$ be a smooth family of closed real $(1,1)$-forms satisfying the assumptions in Introduction.

	\begin{definition}\label{def_pp}
		We let $\mathcal{P}(X_T,\omega)$ denote the set of functions $\f: X_T\rightarrow[-\infty,+\infty)$ such that
		\begin{itemize}
			\item $\f$ is upper semi-continuous on $X_T$ and $\f\in L^1_{\rm loc}(X_T)$;
			\item for each $t\in (0,T)$ fixed, the slice $\f_t:x\mapsto\f(t,x)$ is $\omega_t$-psh on $X$;
			\item  
			for any compact subinterval $J\subset(0,T)$, there exists a positive constant $\kappa=\kappa_{J}(\f)$ such that 
			\begin{align}\label{famlip}
			\partial_t \f \leq \kappa-\kappa(\rho+\chi),
			\end{align}
			in the sense of distributions on $J\times \Omega$, where $\rho,\chi$ are defined in \eqref{rho}, \eqref{chi}.
		\end{itemize}
	\end{definition}

	%We say that a family $\mathcal{A}\subset\mathcal{P}(X_T,\omega)$ is locally uniformly Lipschitz in $(0,T)$ if the inequality \eqref{famlip} is satisfied for any $\f\in\mathcal{A}$ with a uniform constant $\kappa=\kappa_{J}(\mathcal{A})>0$ which depends only on $J$ and $\mathcal{A}$.
	
	We would like to have an interpretation of the last condition. For any compact subset $K\Subset\Omega$, there exists a constant $C=C(K)>0$ such that 
	$$\partial_{\tau}(t,x)\leq C,\; \forall\, (t,x)\in J\times K.$$ 
	Hence for every $x\in K$, the function $t\mapsto\f(t,x)-Ct$ is decreasing in $J$, so the partial derivative $\partial_{t}\f$ exists for almost everywhere $t\in J$ (see e.g.~\cite[Theorem 2.1.8]{kannan1996advanced}).
	
	\begin{lemma}\label{12}	
		Let $\f_0$ be an $\omega_{0}$-psh function and $\f\in \mathcal{P}(X_T,\omega)$. %Assume that $\f$ is upper semi-continuous in $(0,T)\times \Omega$. 
		If  $\f_t\rightarrow \f_0$ in $L^1(X)$ as $t\rightarrow 0$, then the extension $\f:[0,T)\times X\rightarrow[-\infty,+\infty)$ is upper semi-continuous in $[0,T)\times X$.
	\end{lemma} 
	
	\begin{proof} %Recall that $\Omega_T=(0,T)\times \Omega$.  The upper semicontinuity of $\varphi$ inside $\Omega_T$ follows from the semicontinuity in space and Lipschitz regularity in time. 
		It suffices to prove that the extension $\f$ is upper semi-continuous at $(0,x_0)$ for any $x_0\in X$. Let $(t_j,x_j)\in X_T$ be a sequence which converges to $(0,x_0)$. We will show that
		\begin{align*}
		\limsup_{j\rightarrow+\infty}\f(t_j,x_j)\leq \f_0(x_0).
		\end{align*}
		Since $\f$ is bounded from above we can assume the functions $\f_t$ are negative.  Let $h_t$ be a smooth local potential for $\omega_t$ in an open neighborhood $B$ of $x_0$  i.e. $\dc h_t=\omega_t$. Up to replacing $\f_t$ by $\f_t+h_t$, we may assume that the functions  $\f_t$  are psh and negative on $B$. 
		Fix $r$ so small that $B(x,2r)\Subset B$. For any $\delta\in[0,r)$, there exists $j_0$ such that $x_j\in B(x_0,\delta)$ for all $j\geq j_0$ hence $B(x_0,r)\subset B(x_j,r+\delta)$. We have 		
		\begin{align*}
		\f(t_j,x_j)&\leq \dfrac{1}{\Vol(B(x_j,r+\delta))}\int_{B(x_j,r+\delta)}\f(t_j,x)dV\\
		&\leq\dfrac{1}{\Vol(B(x_j,r+\delta))}\int_{B(x_0,r)}\f(t_j,x)dV.
		\end{align*}
		Since $\limsup_j\f_{t_j}(x)\leq \f_0(x)$ for all $x\in X$, Fatou's lemma implies that
		\begin{align*}
		\limsup_{j\rightarrow+\infty}\f(t_j,x_j)\leq \dfrac{\Vol(B(x_0,r))}{\Vol(B(x_0,r+\delta))}\dfrac{1}{\Vol(B(x_0,r))}\int_{B(x_0,r)}\f_0(x)dV(x).
		\end{align*}
		Now we first let $\delta\rightarrow 0$ and then $r\rightarrow 0$ to conclude the proof. 
	\end{proof}
	
	\begin{definition}
		We say that $\varphi \in \mathcal{P}(X_T,\omega)$ has \emph{minimal singularities} if $\varphi_t-V_{\omega_t}$ is bounded for each $t\in (0,T)$ fixed.
	\end{definition}
	
	If $\f\in\mathcal{P}(X_T,\omega)\cap L^{\infty}_{\loc}(\Omega_T)$ the product
	$$(\omega_t+\dc\f_t)^n$$
	is well defined as a positive measure in $\Omega$ as follows from the works of Bedford-Taylor~\cite{bedford1976dirichlet,bedford1982new}.
	This Monge-Amp\`ere measure extends trivially over $X$ since $\Omega$ is a Zariski open subset in $X$. Since $\omega_t\leq \Theta$ for  $0\leq t\leq T$, the positive Borel measures $(\omega_t+\dc\f_t)^n$ have uniformly bounded masses on X:
	\begin{align}\label{bdmass}
	\int_X(\omega_t+\dc\f_t)^n\leq\int_X (\Theta+\dc\f_t)^n\leq \int_X\Theta^n.
	\end{align}
	These can be considered as a family of currents of degree $2n$ in the real $(2n+1)$-dimensional manifold $X_T=(0,T)\times X$. We have the following: %We now show that this family depends continuously on $t$:
	\begin{lemma}\label{lem: 17}
		Let $\f\in\mathcal{P}(X_T,\omega)\cap L^{\infty}_{\rm loc}(\Omega_T)$ and $\gamma$ be a continuous test function in $\Omega_T$. Then  $t\mapsto \int_\Omega\gamma(t,\cdot)(\omega_t+dd^c \varphi_t)^n$ is a Borel bounded measurable function in $(0,T)$, and
		\begin{align*}
		\sup_{0<t<T}\left|\int_\Omega\gamma(t,\cdot)(\omega_t+\dc\f_t)^n \right|\leq (\max_{\Omega_T}\gamma)\int_X\Theta^n.
		\end{align*}
		%In particular,  
	\end{lemma}
	\begin{proof}
		%Fix  compact sets $J\Subset(0,T)$, $K\Subset\Omega$ such that $ \textrm{Supp}(\gamma) \Subset J\times K$. The continuity on $(0,T)\backslash J$ is clear. Fix now $t_0\in J$. By definition of $\mathcal{P}(X_T,\omega)$ there exists a constant $C>0$ depending on $J$, $K$ such that, for each $x\in K$ fixed, $t\mapsto \f(t,x)-Ct$ is decreasing in $J$. Thus the continuity of $t\mapsto\omega_t$ and Bedford-Taylor's monotone convergence theorem (see e.g. \cite[Theorem~3.18, 3.23]{GZ}) ensure that $(\omega_t+\dc\f_t)^n\rightarrow(\omega_{t_0}+\dc\f_{t_0})^n$ weakly as $t\to t_0$. On the other hand $\gamma(t,\cdot)$ uniformly converges on $\Omega$ to $\gamma(t_0,\cdot)$, the first statement follows. 
		For the first statement, the proof is identical to the corresponding one in the local context; see~\cite[Lemma 2.2]{guedj2021viscosity}. 
		The second one follows from the inequality \eqref{bdmass} above. 
	\end{proof}
	This shows that $dt\wedge (\omega_t+\dc\f_t)^n$ is well-defined as a positive Borel measure in $X_T$.
	\begin{definition}\label{def_lhs}
		Fix $\f\in\mathcal{P}(X_T,\omega)\cap L^\infty_{\rm loc}(\Omega_T)$. The map
		\begin{align*}
		\gamma\mapsto\int_{X_T}\gamma dt\wedge(\omega_t+\dc\f_t)^n:=\int_0^Tdt\left(\int_\Omega\gamma(t,\cdot)(\omega_t+\dc\f_t)^n \right) 
		\end{align*}
		defines a positive $(2n+1)$-current on $\Omega_T$, hence on $X_T$, denoted by $dt\wedge(\omega_t+\dc\f_t)^n$, which can be identified with a  positive Radon measure on $X_T$.
	\end{definition}
	The following is a parabolic analogue of the convergence result of Bedford-Taylor \cite{bedford1976dirichlet,bedford1982new}.
	\begin{lemma}\label{lem: convergence}
		Assume that $(\f^j)$ is a monotone sequence of functions in $\mathcal{P}(X_T,\omega)$   which converges almost everywhere to a function $\f\in \mathcal{P}(X_T,\omega) \cap L^{\infty}_{\loc}(\Omega_T)$ on $X_T$. Then   
		$$dt\wedge(\omega_{t}+\dc\f^j_t)^n\rightarrow dt\wedge(\omega_{t}+\dc\f_t)^n$$
		in the sense of measures on $\Omega_T$.
	\end{lemma}
	\begin{proof}
		The proof is similar to that of \cite[Proposition~1.12]{guedj2020pluripotential} but, because of its crucial role in the sequel, we give the details here. 
		
		Let $\gamma(t,x)$ be a continuous test function in $\Omega_T$. By definition we have, for any $j$,
		\begin{align*}
		\int_{\Omega_T}\gamma(t,\cdot)dt\wedge(\omega_t+\dc\f^j_t)^n=\int_0^Tdt\left(\int_\Omega\gamma(t,\cdot)(\omega_t+\dc\f^j_t)^n\right).
		\end{align*}
		We now apply Bedford-Taylor's convergence theorem (see e.g.~\cite[Theorem 3.23]{guedj2017degenerate}) to infer that, for any $t\in(0,T)$,
		\begin{align*}
		\int_\Omega\gamma(t,\cdot)(\omega_t+\dc\f^j_t)^n\rightarrow\int_\Omega\gamma(t,\cdot)(\omega_t+\dc\f_t)^n.
		\end{align*}
		On the other hand, Lemma~\ref{lem: 17} yields, for all $t\in(0,T)$,
		\begin{align*}
		\left|\int_{\Omega}\gamma(t,\cdot)(\omega_t+\dc\f^j_t)^n\right|\leq\max_\Omega|\gamma(t,\cdot)| \Vol(\omega_t)\leq C(\gamma) \int_X \Theta^n.
		\end{align*}
		The result follows from Lebesgue dominated convergence theorem. 
	\end{proof}

	We say that $\f:X_T\rightarrow\R$ is \textit{locally uniformly semi-concave} (resp. \textit{semi-convex}) in $\Omega_T=(0,T)\times \Omega$ if for any compact subset $J\times K\subset (0,T)\times \Omega$ there exists $\kappa=\kappa(\f,J,K)>0$ (resp. $\kappa<0$) such that for all $x\in K$, the function $t\rightarrow\f(t,x)-\kappa t^2$ is concave (resp. convex) in $t\in J$. For any $x\in \Omega$ fixed, the left and right derivatives, 
	\begin{align*}
	\partial_t^+\f(t,x)=\lim_{s\rightarrow 0^+}\dfrac{\f(t+s,x)-\f(t,x)}{s},
	\end{align*}
	and
	\begin{align*}
	\partial_t^-\f(t,x)=\lim_{s\rightarrow 0^-}\dfrac{\f(t+s,x)-\f(t,x)}{s}
	\end{align*}
	exist for all $t\in(0,T)$, and they coincide when $\partial_t\f(t,x)$ exists.
	
	%	If $\f \in \mathcal{P}(X_T,\omega)$ is  continuous in $\Omega_T$  and uniformly semi-concave in $t$, then $(t,x)\rightarrow\partial_{t}^{+}\f(t,x)$ is lower semi-continuous while $(t,x)\rightarrow\partial_t^-\f(t,x)$ is upper semi-continuous in $\Omega_T$. Moreover, $\partial_t^+\f$ and $\partial_{t}^-\f$ coincide and are continuous almost everywhere in $\Omega_T$.
	Let $\ell$ denote the Lebesgue measure on $\R$ and $\mu$ denote a positive Borel measure on $X$. We have the following result whose proof is identical to that of \cite[Lemma 1.12]{guedj2018pluripotential}.
	\begin{proposition}\label{t_der}
		Let $\f:\Omega_T\rightarrow\R$ be a 
		continuous function which is locally uniformly semi-concave in $(0,T)$. Then $(t,x)\rightarrow\partial^-_t\f(t,x)$ is upper semi-continuous while $(t,x)\rightarrow\partial_{t}^+\f(t,x)$ is lower semi-continuous in $\Omega_T$. In particular, $\partial_{t}^+\f$ and $\partial_{t}^-\f$  coincide and are continuous in $\Omega_T\backslash E$, where $E$ is a Borel set with $\ell\otimes\mu$ measure zero. 
	\end{proposition}
	%By replacing $\f$ with $-\f$ one can obtain a similar conclusion for uniformly semi-convex functions.
	
	The following convergence results  play a key role in the sequel. We omit their proofs and refer the reader to~\cite[Section 2]{guedj2018pluripotential}.
	
	\begin{proposition}\label{p26}
		Let $D$ be a bounded open subset in $\mathbb{R}^m$,  
		$J\subset\R$ be a bounded open interval, and $0\leq f\in L^p(D)$ with $p>1$. Let $(v_j)$ be a sequence of Borel functions in $J\times D$ such that $(e^{v_j}f)$ is uniformly bounded in $L^1(J\times D,dt\wedge dV)$. Assume that for any $x\in D$, $v_j(.,x)$ converges to a bounded Borel function $v(.,x)$ in the sense of distributions on $J$ and for all $\eta\in \mathcal{C}^{\infty}_0(J\times D)$
		\begin{align}\label{p16}
		\sup_{j\in \mathbb{N},x\in D}\left|\int_J\eta(t,x)v_j(t,x)dt\right|<+\infty.
		\end{align}
		Then for any positive smooth test function $\eta\in\mathcal{C}_0^{\infty}(J\times D)$,
		\begin{align*}
		\liminf_{j\rightarrow+\infty}\int_{J\times D}\eta(t,x)e^{v_j(t,x)}f(x)dt\wedge dV \geq \int_{J\times D}\eta(t,x)e^{v(t,x)}f(x)dt\wedge dV.
		\end{align*}
	\end{proposition}
	\begin{proof}
		See~\cite[Proposition 2.6]{guedj2018pluripotential}.
	\end{proof}

	\begin{proposition}\label{115}
		Let $(f_j)$ be a sequence of positive functions converging to $f$ in $L^1(X_T,\ell\otimes\mu)$. Let $(\f_j)$ be a sequence of functions in $\mathcal{P}(X_T,\omega)\cap L^{\infty}_{\loc} (\Omega_T)$ which
		\begin{itemize}
			\item converges $\ell\otimes\mu$-almost everywhere in $X_T$ to a function $\f\in\mathcal{P}(X_T,\omega)\cap L^{\infty}_{\loc} (\Omega_T)$;
			\item is locally uniformly semi-concave in $(0,T)\times \Omega$. 
		\end{itemize} 
		Then the limit $\lim\limits_{j\rightarrow+\infty}\dot{\f}^j(t,x)$ exists and is equal to $\dot{\f}(t,x)$ for $\ell\otimes\mu$-almost every $(t,x)\in \Omega_T$, and
		\begin{align*}
		h(\dot{\f}^j)f_j\ell\otimes\mu\rightarrow h(\dot{\f})f\ell\otimes\mu,
		\end{align*}
		in the weak sense of measures on $\Omega_T$, for all $h\in \mathcal{C}^0(\R,\R)$. 
	\end{proposition}
	
	\begin{proof}
		We refer the reader to~\cite[Proposition 2.9]{guedj2018pluripotential} (see also~\cite[Theorem 1.14]{guedj2020pluripotential}).
	\end{proof}

	\subsection{Pluripotential subsolutions/supersolutions}\label{subsol}
	We assume that $T<+\infty$. As explained in the introduction, we assume here that $g(t)\theta\leq \omega_t\leq \Theta$, where $\theta$ is a big $(1,1)$-form, $g$ is a smooth increasing positive function in $t\in[0,T]$, and $\Theta$ is a K\"ahler form.

	Let us emphasize here that by comparison with~\cite{guedj2018pluripotential,guedj2020pluripotential}, for an element $\f\in\mathcal{P}(X_T,\omega)$ the weak derivative $\partial_\tau\f(t,\cdot)$ is merely locally bounded from above in $\Omega$ but $\f$ is not locally uniformly Lipschitz in $(0,T)$. This is natural as we are dealing with quasi-psh functions which are bounded from above but not from below. 
	
	Before defining pluripotential subsolutions (supersolutions), we need to make sense of the quantity $\dot{\f}_t=\partial_\tau\f(t,\cdot)$, in order to define the right-hand side of \eqref{cmaf}. 
	By the definition of $\mathcal{P}(X_T,\omega)$, for any compact subset $K\subset \Omega$, $J\Subset (0,T)$, there exists a constant $C=C_{K,J}>0$ such that $Ct-\f(t,x)$ is increasing  in $t\in J$ for every $x\in K$. Thus, for every $x\in K$, $\partial_\tau\f_t(x)$ is well defined for almost every $t\in J$ (see e.g. \cite[Lemma 1.2.8]{kannan1996advanced}). This implies that the right-hand side of \eqref{cmaf} is well-defined almost everywhere in $\Omega_T$ (using Fubini's theorem). This analysis motivates the following:

	\begin{definition}
		We say that a parabolic potential $\f\in \mathcal{P}(X_T,\omega)$ is a {\em pluripotential subsolution} to~\eqref{cmaf} on $X_T$ if
		\begin{itemize}
			\item for each $t\in (0,T)$ fixed, the $\omega_t$-psh function $\f(t,\cdot)$ is locally bounded in $\Omega$
			\item the inequality 
			\begin{align*}
			(\omega_t+\dc\f_t)^n\wedge dt\geq e^{\dot{\f}_t+F(t,\cdot,\f_t)}fdV\wedge dt
			\end{align*}
			holds in the sense of measures in $(0,T)\times \Omega$.
		\end{itemize}
	\end{definition}
	\begin{definition}
		We say that a parabolic potential $\f\in \mathcal{P}(X_T,\omega)$ is a {\em pluripotential supersolution} to~\eqref{cmaf} on $X_T$ if
		\begin{itemize}
			\item for each $t\in (0,T)$ fixed, the $\omega_t$-psh function $\f(t,\cdot)$ is locally bounded in $\Omega$,
			\item the inequality 
			\begin{align*}
			(\omega_t+\dc\f_t)^n\wedge dt\leq e^{\dot{\f}_t+F(t,\cdot,\f_t)}fdV\wedge dt
			\end{align*}
			holds in the sense of measures in $(0,T)\times \Omega$.
		\end{itemize}
	\end{definition}
	\begin{remark}
		In these definitions the left-hand side is well-defined by using Bedford-Taylor's theory (see Definition~\ref{def_lhs}).
	\end{remark}
	\begin{lemma}\label{lem}
		Let $\f\in\mathcal{P}(X_T,\omega)$ be a parabolic potential  such that the restriction of $\f$ to $\{t\}\times \Omega$ is an $\omega_t$-psh function which is locally bounded on $\Omega$. 
		Then 
		
		1) $\f$ is a pluripotential subsolution to \eqref{cmaf} if and only if for a.e. $t\in(0,T)$, 
		\begin{align}\label{lemsub}
		(\omega_t+\dc\f_t)^n\geq e^{\partial_{\tau}\f(t,\cdot)+F(t,\cdot,\f_t)}fdV,
		\end{align}
		in the sense of measures in $\Omega$.
		
		2) $\f$ is a pluripotential supersolution to \eqref{cmaf} if and only if for a.e. $t\in(0,T)$, 
		\begin{align}\label{lemsuper}
		(\omega_t+\dc\f_t)^n\leq e^{\partial_{\tau}\f(t,\cdot)+F(t,\cdot,.\f_t)}fdV,
		\end{align}
		in the sense of measures in $\Omega$.
	\end{lemma}
	\begin{proof}
		We shall prove the result for subsolutions. The proof for supersolutions is similar.
		
		We first assume that \eqref{lemsub} holds for almost every $t$. Let $\eta$ be a positive continuous test function in $(0,T)\times \Omega$. We thus obtain
		\begin{align*}
		\int_\Omega
		\eta(t,\cdot)(\omega_t+\dc\f_t)^n\geq \int_\Omega\eta(t,\cdot)e^{\partial_{t}\f_t+F(t,\cdot,\f_t)}fdV.
		\end{align*}
		Integrating with respect to $t$, we get
		\begin{align*}
		\int_0^T\int_\Omega\eta (\omega_t+\dc\f_t)^n\wedge dt\geq \int_0^T\int_\Omega\eta e^{\partial_{t}\f_t+F(t,\cdot,\f_t)}fdVdt,
		\end{align*}
		hence $\f$ is a pluripotential subsolution to \eqref{cmaf}.
		
		Conversely, assume that $\f$ is a pluripotential subsolution to \eqref{cmaf}. We consider positive test functions that can be decomposed as
		\begin{align*}
		\eta(t,x)=\lambda(t)\xi_j(x),
		\end{align*} 
		where $(\xi_j)$ is a sequence of positive test functions on $\Omega$ which generates a dense subspace of the space $\mathcal{C}^0_c(\Omega)$ in $\mathcal{C}^0$-topology. 
		It follows from Fubini's theorem that 
		\begin{align*}
		\int_0^T\left(\int_\Omega\xi_j(x)(\omega_t+\dc\f_t)^n\right)\lambda(t)dt\geq\int_0^T\left(\int_\Omega\xi_j(x)e^{\partial_t\f_t+F(t,.,\f_t)}fdV\right)\lambda(t)dt.
		\end{align*}
		Hence for any $j$ there exists a subset $E_j$ which has full measure in $(0,T)$ so that for all $t\in(0,T)$
		\begin{align}\label{preine}
		\int_\Omega\xi_j(x)(\omega_t+\dc\f_t)^n\geq \int_\Omega\xi_j(x)e^{\partial_t\f_t+F(t,\cdot,\f_t)}f(x)dV(x).
		\end{align}
		If we set $E:=\cap_jE_j$, then $E$ has full measure in $(0,T)$. Moreover, the inequality \eqref{preine} holds for all $t\in E$ and for all $j$. Let $\xi$ be an arbitrary positive continuous function in $\Omega$. We can approximate this function by convex combinations of the $\xi_j$, we infer that for all $t\in E$,
		\begin{align*}
		\int_\Omega\xi(x)(\omega_t+\dc\f_t)^n\geq \int_\Omega\xi(x)e^{\partial_t\f_t+F(t,\cdot,\f_t)}f(x)dV(x), 
		\end{align*}
		from which \eqref{lemsub} follows.
	\end{proof}
	
	\begin{lemma}\label{maxi}
		Let $\f,\psi\in \mathcal{P}(X_T,\omega)$ be two pluripotential subsolutions to \eqref{cmaf}. Then
		\begin{align*}
		1_{\{\f\geq\psi\}}\partial_t\max(\f,\psi)=1_{\{\f\geq \psi\} }\partial_t\f,
		\end{align*}
		almost everywhere in $\Omega_T$, and 
		\begin{align*}
		(\omega+\dc\max(\f,\psi))^n\wedge dt\geq 1_{\{\f\geq \psi\}} (\omega+\dc\f)^n\wedge dt
		\end{align*} 
		in the sense of measures in $\Omega_T$.
	\end{lemma}

	\begin{proof} 
		Fix $K\Subset\Omega$  and $J\Subset (0,T)$. Then there exists a constant $C=C_{K,J}>0$ such that, for any $x\in K$ fixed, $Ct-\f(t,x)$ and $Ct-\psi(t,x)$ are increasing in $t\in J$.  These functions thus have derivatives in $t$ almost everywhere on $J$ (see e.g.~\cite[Theorem 1.2.8]{kannan1996advanced}). Hence the first equality follows from~\cite[Lemma 7.6]{gilbarg1983elliptic}.
		
		The second inequality is a simple consequence of  the elliptic  maximum principle in the local context (see e.g. \cite[Theorem 3.27]{guedj2017degenerate}).
	\end{proof}
	
	\begin{lemma}\label{lem_initial}
		For every $\lambda\in [(1+\delta_0g(0))^{-1},1]$,
		the function $\lambda g(0)(\rho+\chi)/2$ is $\omega_0$-psh. In particular, there exists a uniform constant $C_0>0$ such that
		\begin{align*}
		\lambda g(0)\frac{\rho+\chi}{2}-C_0\leq \f_0.
		\end{align*}
	\end{lemma}
	Recall here that $\chi$ is a fixed $\theta$-psh function with analytic singularities such that $\theta+\dc\chi\geq 2\delta_0\Theta$, for some $\delta_0>0$.
	\begin{proof}
		By hypothesis \eqref{estimate_ome}, we first observe that
		\begin{align*}
		\omega_0+\dc\lambda g(0)\frac{\rho+\chi}{2}&= \frac{\lambda}{2}(\omega_0+g(0)\dc\rho)+\frac{\lambda}{2}(\omega_0+g(0)\dc\chi)+(1-\lambda)\omega_0\\
		&\geq \frac{\lambda}{2}g(0)(\theta+\dc\rho)+\frac{\lambda}{2}g(0)(\theta+\dc\chi)+(1-\lambda)\omega_0\\
		&\geq \lambda g(0)\delta_0\Theta-(1-\lambda)\Theta\geq 0
		\end{align*}
		where the last inequality follows from the choice of $\lambda$. Thus 
		the function $\lambda g(0)(\rho+\chi)/2$ is $\omega_0$-psh. Since $\f_0$ is $\omega_0$-psh with minimal singularities,  there exists a constant $C_0>0$ such that $\lambda g(0)(\rho(x)+\chi(x))/2-C_0\leq \f_0(x)$ for all $x\in X$. 
	\end{proof}

	\section{The envelope of subsolutions} \label{sec:envelope}
	\subsection{Definition}
	\begin{definition}\label{defcp}
		A Cauchy datum for \eqref{cmaf} is a $\omega_{0}$-psh function $\f_0:X\rightarrow\R$ with minimal singularities.
		We say $\f \in \mathcal{P}(X_T,\omega)$ is a subsolution to the Cauchy problem:
		\begin{align*}
		(\omega_t+\dc u_t)^n=e^{\partial_{t}u_t+F(t,\cdot,u_t)}fdV,\quad  u\big|_{\{0\}\times X}=\f_0
		\end{align*}
		if $\f$ is a pluripotential subsolution  to \eqref{cmaf} such that
		$\limsup_{t\rightarrow 0}\f(t,x)\leq\f_0(x)$ for all $x\in X$.
		
		We let $\mathcal{S}_{\f_0,f,F}(X_T)$ denote the set of pluripotential subsolutions to the Cauchy problem above.		
	\end{definition}
	
	\begin{lemma}\label{exist_sub}
		The set $\mathcal{S}_{\f_0,f,F}(X_T)$ is non-empty, uniformly bounded from above on $X_T$, and stable under finite maxima.
	\end{lemma}
	\begin{proof} Fix $\lambda\in [(1+\delta_0g(0))^{-1},1]$.
		Consider, for any $(t,x)\in X_T$,
		\begin{equation}\label{eq: subsol underline u}
		\underline{u}(t,x):=\lambda g(t)\frac{\rho(x)+\chi(x)}{2}-C_1(t+1),
		\end{equation}
		where $\rho$ and $\chi$ are defined in \eqref{rho} and \eqref{chi}, the uniform constant $C_1>0$ will be chosen later.
		By hypothesis \eqref{estimate_ome} on $\omega_t$ we have
		\begin{align*}
		\frac{\lambda}{2}(\omega_t+g(t)\dc\chi)+(1-\lambda)\omega_t &\geq \frac{\lambda}{2}g(t)(\theta+\dc\chi)+(1-\lambda)\omega_t\\
		&\geq \lambda g(t)\delta_0\Theta-(1-\lambda)\Theta\\
		&=[\lambda(1+\delta_0g(0))-1]\Theta\geq 0.
		\end{align*}
		since $g(t)$ is increasing. Therefore, we obtain
		\begin{equation}\label{eq: lem22}	\begin{aligned}[t]
		(\omega_t+\dc\underline{u}_t)^n&=\left(\frac{\lambda}{2}(\omega_t+g(t)\dc\rho)+\frac{\lambda}{2}(\omega_t+g(t)\dc\chi)+(1-\lambda)\omega_t\right)^n\\
		&\geq \left(\frac{\lambda}{2}g(t)(\theta+\dc\rho)\right)^n=(\lambda g(t))^ne^{c_1}fdV.
		\end{aligned}
		\end{equation} 
		We set
		\begin{align}\label{const_c1}
		C_1=C_0+M_F+|n\log (g(T))| +|c_1|,
		\end{align}
		it thus follows from~\eqref{eq: lem22} that 
		\begin{align*}
		\exp(\partial_t\underline{u}_t+F(t,\cdot,\underline{u}_t(\cdot)))fdV&=\exp(\lambda g'(t)(\rho+\chi)/2-C_1+F(t,\cdot,\underline{u}_t(\cdot)))fdV\\
		&\leq \exp( n\log (\lambda g(t))+c_1)fdV\\
		&\leq (\omega_t+\dc\underline{u}_t)^n
		\end{align*}
		using in the first inequality that $g$ is increasing in $t\in[0,T]$, $\sup_X\rho=\sup_X\chi=0$.
		It follows moreover from the choice of $C_1$ and Lemma~\ref{lem_initial} that $\underline{u}(0,\cdot)\leq \f_0$ on $X$, hence $\underline{u}\in\mathcal{S}_{\f_0,f,F}(X_T)$. 
		
		Let now $\f\in\mathcal{S}_{\f_0,f,F}(X_T)$ such that $\f\geq \underline{u}$. Set with $\mu:=fdV$.
		Consider the set
		$$G:=\{x\in X: \underline{u}(T,x)>-M \},$$
		for $M>0$ so large that $\mu(G)>\frac{\mu(X)}{2}$. We observe that for every $t\in (0,T)$, $$\f_t(x)\geq \underline{u}(t,x)\geq \underline{u}(T,x)>-M,\,\forall\, x\in G.$$
		Set $-m_F=\inf_{[0,T)\times X}F(t,x,-M)>-\infty$.
		Since $F(.,.,r)$ is non-decreasing in $r$ we obtain 
		\begin{align*}
		\int_Ge^{\dot{\f_t}-m_F}d\mu\leq \int_Ge^{\dot{\f}_t+F(t,\cdot,\f_t)}d\mu\leq \int_G(\omega_t+\dc\f_t)^n\leq\int_G(\Theta+\dc\f_t)^n\leq \int_X\Theta^n.
		\end{align*}
		On the other hand, it follows from Jensen's inequality that
		\begin{align*}
		\exp\left(\int_G\dot{\f_t}\dfrac{d\mu}{\mu(G)}\right)\leq \int_G e^{\dot{\f_t}}\dfrac{d\mu}{\mu(G)}.
		\end{align*} 
		Combining these two estimates we get
		\begin{align*}
		\int_G\dot{\f_t}d\mu&\leq \mu(G)\log\left(\frac{e^{m_F}\int_X\Theta^n}{\mu(G)}\right)\leq \mu(X)\log\left(\frac{2e^{m_F}\int_X\Theta^n}{\mu(X)}\right)=:C.
		\end{align*}
		We then infer that the function 
		$
		t\mapsto\int_G\f_td\mu-Ct
		$
		is non-increasing in $(0,T)$, hence 
		\begin{align}\label{eq: bound integral}
		\int_G\f_t d\mu \leq\int_G\f_0d\mu + Ct \leq \int_G\f_0d\mu +CT. 
		\end{align}
		On the other hand, it follows from~\cite[Proposition 8.5]{guedj2017degenerate} that there exists a uniform constant $C'$ (only depending on $\mu$) such that 
		\[
		\int_X (\psi -\sup_X \psi) d\mu \geq -C', \; \text{for all}\; \psi \in \PSH(X,\Theta). 
		\]
		Thus for each $t\in (0,T)$,
		\begin{align*}
		-C' \leq \int_X (\varphi_t -\sup_X \varphi_t) d\mu \leq \int_G  (\varphi_t -\sup_X \varphi_t) d\mu  \leq C'' -\mu(G) \sup_X \f_t.
		\end{align*}	
		We deduce that $\sup_X \varphi_t$ is uniformly bounded from above. 		
		
		\smallskip
		The stability under finite maxima follows immediately from Lemma~\ref{maxi}.
	\end{proof}
	From now on, we let $M_0>0$ denote  a uniform upper bound of  all pluripotential subsolutions $\f$ to \eqref{cmaf} such that $\f\geq \underline{u}$ on $X_T$,
	and set
	\begin{align}
	M_F:=\sup_{X_T}F(\cdot,\cdot,M_0).
	\end{align}		
	Lemma~\ref{exist_sub} allows us to define the upper envelope of  subsolutions:
	
	\begin{definition}
		We let 
		\begin{align*}
		U=U_{\f_0,f,F,X_T}:=\sup\{ \f\in\mathcal{S}_{\f_0,f,F}(X_T): \underline{u}\leq \f\leq M_0 \}
		\end{align*}
		denote the upper envelope of all subsolutions. 
	\end{definition}
	
	\begin{lemma}\label{lemexist} 
		There exists $\underline{\f}\in\mathcal{S}_{\f_0,f,F}(X_T)$ such that for any $x\in X$,
		\begin{align*}
		\lim_{t\rightarrow 0}\underline{\f}(t,x)=\f_0(x).
		\end{align*}
	\end{lemma}
	\begin{proof}

		We set $\delta=\delta_0g(0)$. For any $(t,x)\in [0,\delta)\times X$, consider
		$$\underline{\f}(t,x)=(1-\alpha_t )\f_0(x)+\alpha_tg(0)\frac{\rho(x)+\chi(x)}{2}
		+nt(\log(\delta_0^{-1} t)-1)-Ct,$$
		where $\alpha_t=\delta^{-1}t$, the functions $\rho,\chi$ are defined in \eqref{rho}, \eqref{chi}, and $$C:=M_F+\delta^{-1}
		\sup_X\left(g(0)\frac{\rho+\chi}{2}-\f_{0}\right)-\min(c_1,0).$$ 
		Lemma \ref{lem_initial} with $\lambda=1$ ensures that $C<+\infty$. We compute 
		\begin{align*}
		\omega_t+\dc \underline{\f}_t&=(1-\alpha_t)(\omega_{0}+\dc\f_{0})+\frac{\alpha_t}{2}(\omega_{0}+g(0)\dc\rho)\\
		&+ \frac{\alpha_t}{2}(\omega_{0}+g(0)\dc\chi)+\omega_t-\omega_{0}.
		\end{align*}
		Since $\omega_t+t\Theta$ is increasing, we have  $\omega_t-\omega_{0}\geq-t\Theta$, hence
		\begin{align*}
		\frac{\alpha_t}{2}(\omega_{0}+g(0)\dc\chi)+\omega_t-\omega_{0}&\geq \frac{\alpha_t g(0)}{2}(\theta+\dc\chi)+\omega_t-\omega_{0}\\
		&\geq \alpha_t g(0)\delta_0\Theta-t\Theta=0.
		\end{align*}
		Computing the time derivative we obtain
		\begin{align*}
		\partial_t\underline{\f}(t,\cdot)=(\delta_0g(0))^{-1}\left(g(0)\frac{\rho+\chi}{2}-\f_0\right)+n\log(\delta_0^{-1}t)-C.
		\end{align*}
		Hence, for all $t\in (0,\delta)$ we have, by the choice of the constant $C$,
		\begin{align*}
		(\omega_t+\dc \underline{\varphi}_t)^n&\geq \left ( \frac{\alpha_t g(0)}{2}(\theta+\dc\rho)\right )^n=(\delta_0^{-1}t)^ne^{c_1}fdV\\
		&\geq e^{\partial_t\underline{\varphi}_t+F(t,\cdot,\underline{\varphi}_t)}fdV.
		\end{align*}
		We divide $[0,T]$ into $N$ small intervals of the same length $[T_k,T_{k+1}]$,  $k=0,...,N-1$ so that $|T_{k+1}-T_k| \leq \delta:= \delta_0 g(0)$, $T_0=0$ and $T_{N}=T$. For $t\in [T_k,T_{k+1}]$ we define 
		\begin{align*}
		\underline{\varphi}^{(k)}(t,\cdot):&= (1-\alpha^{(k)}_t) \varphi_{T_k} + \alpha^{(k)}_t g(T_k) \frac{\rho +\chi}{2}-C^{(k)}(t-T_k)\\
		&\quad + n(t-T_k) (\log (\delta_0^{-1}(t-T_k)) -1),
		\end{align*}
		where $\alpha_t^{(k)}=\delta ^{-1}(t-T_k)$, and 		
		\[
		C^{(k)}=M_F+\delta^{-1}
		\sup_X\left(g(T_k)\frac{\rho+\chi}{2}-\f_{T_k}\right)-\min(c_1,0).
		\]
		The subsolution constructed in the proof of Lemma~\ref{exist_sub} (see \eqref{eq: subsol underline u}) ensures that  $C^{(k)}<+\infty$ is  a uniform positive constant. 
		The same arguments as above ensure that $\underline{\f}^{(k)}$ is a pluripotential subsolution to~\eqref{cmaf} in $[T_k,T_{k+1}]\times X$. 
		Gluing these functions, we get our desired pluripotential subsolution defined on $[0,T)\times X$.
		It is also clear from the definition that $\underline{\f}(t,.)$ converges to $\f_0$ in $L^1(X,dV)$ as $t\rightarrow 0^+$.		
	\end{proof}

	\subsection{Lipschitz regularity in time} In this section, we study the regularity in time $t$ of the Perron upper envelope by adapting some arguments in~\cite[Section 4]{guedj2018pluripotential}. 
	\begin{proposition}\label{pro1}
		For all $0<S<T$ we have $U_{\f_0,f,F,X_S}=U_{\f_0,f,F,X_T}$ in $X_S$.
	\end{proposition}
	\begin{proof}
		Set $U_T:=U_{\f_0,f,F,X_T}$ and $U_S:=U_{\f_0,f,F,X_S}$. We can assume that $|T-S|\leq \dfrac{\delta_0 g(0)}{2}$, since if we can show that $U_T=U_S$ for such $S$ we can restart the process to prove that $U_S=U_{S'}$ for $S-\dfrac{\delta_0 g(0)}{2}<S'<S$. 
		
		It suffices to show that $U_S\leq U_T$ because the reverse inequality is clear. Fix $\f\in\mathcal{S}_{\f_0,f,F}(X_S)$. Fix $0<t_0<S$ such that $T-t_0< \delta_0 g(0)$. 
		Set, for $(t,x)\in (t_0,T)\times X$, 
		\begin{align*}
		\psi(t,x)=(1-\alpha_t)\f(t_0,x)+\alpha_t g(t_0)\dfrac{\rho(x)+\chi(x)}{2}-C(t-t_0)+n(t-t_0)(\log[\delta_0^{-1}(t-t_0)]-1),
		\end{align*}
		where $\alpha_t=(\delta_0g(t_0))^{-1}(t-t_0)<1$, the functions $\rho,\chi$ are defined in~\eqref{rho}, \eqref{chi}, and $$C:=M_F+(\delta_0g(t_0))^{-1}
		\sup_X\left(g(t_0)\frac{\rho+\chi}{2}-\f_{t_0}\right)+|c_1|.$$ 
		From~\eqref{eq: subsol underline u}, with  $\lambda=1$ and $t=t_0$, we see that $C<+\infty$ is a uniform constant. From \eqref{eq: subsol underline u} again we see that $\partial_t \psi(t,x)$ satisfies  \eqref{12}. 
		We compute
		\begin{align*}
		\omega_t+\dc\psi_t&=(1-\alpha_t)(\omega_{t_0}+\dc\f_{t_0})+\dfrac{\alpha_t}{2}(\omega_{t_0}+g(t_0)\dc\rho)\\
		&+\dfrac{\alpha_t}{2}(\omega_{t_0}+g(t_0)\dc\chi)+\omega_t-\omega_{t_0}.
		\end{align*}
		Since $\omega_t+t\Theta$ is increasing, we have $\omega_t-\omega_{t_0}\geq-(t-t_0)\Theta$ for all $t\geq t_0$. It thus follows that 
		\begin{align*}
		\dfrac{\alpha_t}{2}(\omega_{t_0}+g(t_0)\dc\chi)+\omega_t-\omega_{t_0}&\geq \dfrac{\alpha_tg(t_0)}{2}(\theta+\dc\chi)+\omega_t-\omega_{t_0}\\
		&\geq \alpha_t g(t_0)\delta_0\Theta-(t-t_0)\Theta=0.
		\end{align*}
		Hence, for all $t\in[t_0,T)$,
		\begin{align*}
		(\omega_t+\dc\psi_t)^n&\geq \left ( \frac{\alpha_t g(t_0)}{2}(\theta+\dc\rho)\right )^n=(\delta_0^{-1}(t-t_0))^ne^{c_1}fdV\\
		&\geq e^{\partial_t\psi_t+F(t,\cdot,\psi_t)}fdV,
		\end{align*}
		by the choice of the constant $C$. Therefore the function \begin{align*}
		(t,x)\mapsto u(t,x)	:=\begin{cases}
		\f(t,x), \quad \text{if}\, t\in[0,t_0]\\
		\psi(t,x),\quad\text{if}\, t\in [t_0,T)
		\end{cases}
		\end{align*} 
		is a pluripotential subsolution to~\eqref{cmaf} in $[0,T)\times X$ by using Lemma~\ref{lem}. We thus have $u\in\mathcal{S}_{\f_0,f,F}(X_T)$ since $u(0,\cdot)=\f_0$. This yields $u\leq U_T$ in $[0,T)\times X$. In particular $\f\leq U_T$ in $[0,t_0]\times X$, and it follows that $U_S\leq U_T$ on $[0,t_0]\times X$ by taking supremum over all subsolutions. We now let $t_0\rightarrow S$ to obtain $U_s\leq U_T$ in $X_S$.
	\end{proof}
	Next we introduce the mixed type inequality:
	
	\begin{lemma}\label{logconcave}
		Let $\theta_1,\theta_2$ be two closed smooth $(1,1)$-forms on $X$ with big cohomology classes. Let $\f_1$ ($\f_2$ resp.) be a bounded $\theta_1$-psh ($\theta_2$-psh resp.) function such that
		\begin{align*}
		(\theta_1+\dc\f_1)^n\geq e^{f_1}\mu\quad \text{and}\quad (\theta_2+\dc\f_2)^n\geq e^{f_2}\mu
		\end{align*}
		where $f_1,f_2$ are bounded measurable functions and $\mu$ is a positive Radon measure with $L^1$ density with respect to Lebesgue measure. Then, for any $\lambda\in(0,1)$,
		\begin{align*}
		(\lambda(\theta_1+\dc\f_1)+(1-\lambda)(\theta_2+\dc\f_2))^n\geq e^{\lambda f_1+(1-\lambda)f_2}\mu.
		\end{align*}
	\end{lemma}
	\begin{proof}
		The proof is the same as that of \cite[Lemma 2.10]{guedj2018pluripotential} using the convexity of the exponential together with the mixed Monge-Amp\`ere inequalities;
		see e.g.~\cite{dinew2009inequality}.
	\end{proof}	
	%We now prove that the upper envelope is locally uniformly Lipschitz. To  enlarge the scope of applications we will weaken our assumptions in this section. We merely assume that $0\leq f\in L^{1-\delta}(X)$ for some $0<\delta<1$ and we assume that there exist $\rho \in L^{\infty}_{\loc} (\Omega) \cap \PSH(X,\theta)$ and $\gamma>0$ such that, for all $t\in (0,T)$,
	%	\begin{equation}
	%		\label{eq: new rho}
	%		(\theta +dd^c \rho)^n \geq  e^{F(t,\cdot,\gamma \rho) } fdV. 
	%	\end{equation}
	\begin{theorem}\label{thmlips}
		There exists a uniform constant $L_U>0$ such that for all $(t,x)\in X_T$,
		\begin{align}\label{lips}
		t|\partial_tU(t,x)|\leq L_U-L_U(\rho(x)+\chi(x)).
		\end{align}
	\end{theorem}
	\begin{proof} %According Lemma \ref{lem_increasing}, we can reduce to the case when $r\mapsto F(.,.,r)$ is increasing in $r$. Then there exists a smooth increasing function $g:[0,T)\rightarrow \mathbb{R}^+$ such that $\omega_t\geq e^{g(t)}\theta$ for all $t$.

		Let $\f\in\mathcal{S}_{\f_0,f,F}(X_T)$ such that $\f\geq \underline{u}$ on $X_T$, where $\underline{u}$ is defined in \eqref{eq: subsol underline u}. Fix $0<T'<T$ and $\e_0>0$ so small that $(1+\e_0)T'<T$. Set, for all $(t,x)\in X_{T'}$, $s\in (1-\e_0,1+\e_0)$,
		\begin{align*}
		u^s(t,x):=\dfrac{\alpha_s}{s}\f(st,x)+(1-\alpha_s)g(t)\dfrac{\rho(x)+\chi(x)}{2}-C|s-1|(t+1),
		\end{align*}
		where $\rho,\chi$ are defined in \eqref{rho}, \eqref{chi}, $\alpha_s=1-A|s-1|,$ and
		\begin{align*}
		C=C_0(A+2)+\kappa_FT+AM_F+(A+2)C_1(T+1)+(A+2)M_0. 
		\end{align*}
		The constant $C_1$ is defined in \eqref{const_c1}, and the constant $A$ will be chosen later that depends only on $T$.
		We will show that $u^s\in\mathcal{S}_{\f_0,f,F}(X_T)$. We compute
		\begin{align*}
		\omega_t+\dc u^s(t,\cdot)&=\dfrac{\alpha_s}{s}(\omega_{st}
		+\dc\f_{st})\\
		&\quad+\alpha_s\omega_t-\dfrac{\alpha_s}{s}\omega_{st}
		+(1-\alpha_s)\left(\omega_t+g(t)\dc \dfrac{\rho+\chi}{2}\right).	
		\end{align*}
		Since $\f$ is a subsolution to \eqref{cmaf}, we have for almost every $t\in(0,T')$,
		\begin{align*}
		(s^{-1}(\omega_{st}+\dc\f_{st}))^n\geq e^{-n\log s+\partial_{t}\f(st,\cdot)+F(t,\cdot,\f(st,\cdot))}fdV.
		\end{align*} 
		Recalling the definition of $\rho$ and $\omega_t \geq g(t) \theta$ we also have 
		\begin{align}\label{eq: weaker assumption on rho}
		\left(\dfrac{1}{2}(\omega_t+g(t)\dc\rho)\right)^n\geq\left(\dfrac{g(t)}{2}(\theta+\dc\rho)\right)^n=e^{n\log g(t)+c_1}fdV.
		\end{align}
		On the other hand, since $\dot{\omega}_t\geq -\Theta$, we have
		\begin{align*}
		\alpha_s\omega_t-\dfrac{\alpha_s}{s}\omega_{st}&=\dfrac{\alpha_s}{s}(\omega_t-\omega_{st})+\alpha_s(1-s^{-1})\omega_t\\
		&\geq -\alpha_st \left|s^{-1}-1\right|\Theta-\alpha_s\left |s^{-1}-1\right|\Theta\\
		&\geq -(t+1)s^{-1}|s-1|\Theta\\
		&\geq -(2T+2)|s-1|\Theta,
		\end{align*}     
		where the last line follows from $s\geq 1/2$.
		Recall that $\theta+\dc\chi\geq 2\delta_0\Theta$ for some $\delta_0>0$. If we choose  $A\geq 2(T+1)(\delta_0g(0))^{-1}$, then  
		\begin{align*}
		\alpha_s\omega_t-\dfrac{\alpha_s}{s}\omega_{st}+(1-\alpha_s)\frac{\omega_t+g(t)\dc \chi}{2}\geq -(2T+2)|s-1|\Theta+A|s-1|g(t)\delta_0\Theta\geq 0,
		\end{align*}
		since $g$ is increasing in $t$.
		Combining these estimates with the mixed Monge-Amp\`ere inequality (Lemma \ref{logconcave}) we obtain
		\begin{equation}\label{eq: 28}
		\begin{split}
		(\omega_t+\dc u^s(t,\cdot))^n&\geq (\alpha_s(s^{-1}\omega_{st}+\dc\f_{st})+(1-\alpha_s)(2^{-1}(\omega_t+\dc\rho)))^n\\
		&\geq \exp(\alpha_s\partial_{t}\f(st,\cdot)+\alpha_sF(st,\cdot,\f(st,\cdot))\\
		&\quad -\alpha_s n\log s
		+(1-\alpha_s)(n\log g(t)+c_1))f dV \\
		&\geq \exp(\partial_{t}u^s(t,\cdot)+F(t,\cdot,u^s(t,\cdot))fdV. 
		\end{split}   
		\end{equation}
		where the last line follows from the choice of $C$ as we now explain. Indeed, observe that \begin{equation}
		\begin{split}
		\alpha_s\partial_{t}\f(st,\cdot)&=\partial_{t}u^s(t,\cdot)+C|s-1|-(1-\alpha_s)g'(t) \frac{\rho+\chi}{2} \\
		&\geq \partial_{t}u^s(t,\cdot)+C|s-1|. \label{eq: Lipschitz 1}
		\end{split}   
		\end{equation}
		Since $g$ is non-decreasing,  we also have 
		\begin{align*}
		g(ts)-g(t)\leq \kappa_gt|s-1|\leq \kappa_gT\e_0g(0)^{-1} g(t). 
		\end{align*}
		It thus follows that  $g(ts)\leq \gamma g(t)$ where $\gamma=1+\e_0\kappa_gTg(0)^{-1}$. Choosing $\varepsilon_0$ small enough at the beginning we can ensure that $\gamma(1+\delta_0g(0))^{-1}<1$ . Up to increasing $A$ so large that $\frac{A}{A+2}\geq \gamma(1+\delta_0g(0))^{-1}$, hence
		\begin{align*}
		\left( 1-\frac{\alpha_s}{s}\right)\f_{st}
		&=\left( 1-\frac{\alpha_s}{s}\right)(\f_{st}-M_0)+\left( 1-\frac{\alpha_s}{s}\right)M_0\\
		&\geq (A+2)|s-1|(\f_{ts}-M_0)\\
		&\geq (A+2)|s-1|\left(\frac{A}{(A+2)\gamma}g(ts)\frac{\rho+\chi}{2}-C_1(T+1)-M_0\right)\\
		&\geq A|s-1|g(t)\frac{\rho+\chi}{2}-((A+2)C_1(T+1)+(A+2)M_0)|s-1|,
		\end{align*} where the first inequality follows from the elementary one $1-\frac{\alpha_s}{s}\leq (A+2)|s-1|$, while the second inequality follows from  $\f\geq \underline{u}$.
		We infer 
		\begin{align*}
		\f_{st}\geq \frac{\alpha_s}{s}\f_{st}+(1-\alpha_s)g(t)\frac{\rho+\chi}{2}- C|s-1| 
		\geq u^s(t,\cdot). \label{eq: Lipschitz 2}
		\end{align*}
		Using this and the assumption that $F$ is non-decreasing in $r$  
		and uniformly Lipschitz in $t$, we get 
		\begin{equation}
		\begin{split}\label{eq: Lipschitz 3}
		\alpha_sF(st,\cdot,\f(st,\cdot))&= F(st,\cdot,\f(st,\cdot))-(1-\alpha_s) F(st,\cdot,\f(st,\cdot)) \\
		&\geq F(t,\cdot,\f_{ts}(\cdot))-\kappa_Ft|s-1|-|s-1|AM_F  \\
		&\geq F(t,\cdot,u^s(t,\cdot))-(\kappa_FT+AM_F)|s-1|. 
		\end{split}
		\end{equation}
		Combining \eqref{eq: Lipschitz 1}, \eqref{eq: Lipschitz 3}, and the definition of $C$, we obtain the last inequality in~\eqref{eq: 28}.
		Hence $u^s$ is a pluripotential subsolution to~\eqref{cmaf} by Lemma \ref{lem}.
		We now take care of the initial values. For any $x\in X$ we have
		\begin{align*}
		u^s(0,x)&=\f_0(x) -C|s-1|+(1-\alpha_s)g(0)\frac{\rho(x)+\chi(x)}{2}-\left( 1-\frac{\alpha_s}{s}\right)\f_0(x)\\
		&\leq \f_0(x)-C|s-1|+A|s-1|g(0)\frac{\rho(x)+\chi(x)}{2}-(A+2)|s-1|\f_0(x)\\
		&\leq \f_0(x)-C|s-1|+(A+2)|s-1|\left( \frac{A}{A+2}g(0)\frac{\rho(x)+\chi(x)}{2}-\f_0(x)\right)\\
		&\leq \f_0(x)
		\end{align*}
		where the last line follows again from the choice of $C$. Thus, for any $x\in X$ we also get $\limsup_{t\rightarrow 0}u^s(t,x)\leq \f_0(x)$.
		Therefore $u^s\in \mathcal{S}_{\f_0,f,F}(X_T)$, so $u^s\leq U$ in $X_{T'}$. We thus obtain
		\begin{align*}
		\dfrac{\alpha_s}{s}\f(st,x)+(1-\alpha_s)g(t)\frac{\rho(x)+\chi(x)}{2}-C|s-1|(t+1)\leq U(t,x).
		\end{align*}
		We now take the supremum over all subsolutions $\varphi \in \mathcal{S}_{\f_0,f,F}(X_T)$ to get 
		\begin{align*}
		\dfrac{\alpha_s}{s}U(st,x)+ A|s-1|g(t)(\rho(x)+\chi(x))-C|s-1|(t+1)\leq U(t,x), \quad\forall (t,x)\in X_{T'}.
		\end{align*}
		Letting $s\rightarrow 1$, we infer, for all $(t,x)\in X_{T'}$ that
		\begin{align*}
		t|\partial_{t}U(t,x)|\leq C(T+1)+AM_0-Ag(T)(\rho(x)+\chi(x)).
		\end{align*}
		We can now define $L_U:=Ag(T)+C(T+1)+AM_0$. Finally, letting $T'\to T$ and applying Proposition~\ref{pro1} we finish the proof.
	\end{proof}
	%	
	%	\begin{remark}
	%		The arguments above are still working if we merely assume that $\rho$ satisfies
	%		\[
	%		(\theta+dd^c \rho)^{n} \geq 2^ne^{F\left (t,\cdot,\frac{\rho}{2}\right)} fdV.
	%		\]
	%		Then \eqref{eq: weaker assumption on rho} becomes 
	%		\[
	%		\left(\dfrac{1}{2}(\omega_t+g(t)\dc\rho)\right)^n\geq\left(\dfrac{g(t)}{2}(\theta+\dc\rho)\right)^n\geq e^{n \log g(t)+c_1 + F\left (t,\cdot, \frac{\rho}{2}\right ) }fdV.
	%		\]
	%	\end{remark}

	\subsection{Convergence at initial time}
	We define the {\em upper semi-continuous (u.s.c) regularization} $U^*$ of $U$ by the formula
	\[U^*(t,x)=\limsup\limits_{X_T\ni(s,y)\to(t,x)}U(s,y), \quad (t,x)\in X_T. \]	
	We then prove that the upper envelope has the right initial values:

	\begin{theorem}\label{thmlim}
		The upper semi-continuous regularisation of the upper envelope $U:=U_{\f_0,f,F,X_T}$ satisfies, for all $x\in \Omega$,
		\begin{align*}
		\lim_{\Omega_T\ni(t,y)\rightarrow(0,x)}U^*(t,y)=\f_0(x).
		\end{align*}
	\end{theorem}
	\begin{proof}
		Thanks to Lemma~\ref{lemexist}, it suffices to show that for all $x\in \Omega$,
		\begin{align*}
		\limsup_{\Omega_T\ni(t,y)\rightarrow(0,x)}U^*(t,y)\leq \f_0(x).
		\end{align*}
		Theorem~\ref{thmlips} ensures that for $y\in\Omega$ fixed, the upper envelope $U(\cdot,y)$ is locally Lipschitz in $(0,T)$. Arguing exactly as in the proof of~\cite[Lemma 1.7]{guedj2018pluripotential} we can show that $U^*(t,\cdot)=(U_t)^*$ for all $t\in (0,T)$, where $(U_t)^*$ denotes the u.s.c regularization of $u_t$ ($t$ fixed) in the $x$-variable only. It thus remains  to prove that, for all $x\in \Omega$, 
		\begin{align*}
		\limsup_{t\rightarrow 0}U_t^*(x)\leq \f_0(x).
		\end{align*}
		Fixing $M>0$, we set $G:= \{\underline{u}_T >-M\}$, where $\underline{u}$ is defined as in~\eqref{eq: subsol underline u}. 				 		
		We  claim that there exists a constant $C>0$ (also depending on $M$) such that,  for all $t\in (0,T)$,
		\begin{align}\label{claim28}
		\int_G U_t^*fdV\leq\int_G \f_0fdV +Ct.
		\end{align}
		Fix $t_0\in (0,T)$. By Choquet's lemma, there exists a sequence $\{\f^j\}$ in $\mathcal{S}_{\f_0,f,F}(X_T)$ such that
		\begin{align*}
		U_{t_0}^*=\left(\lim\limits_{j\rightarrow+\infty}\f_{t_0}^j\right)^* \quad\text{in}\; X.
		\end{align*} 
		Since the set $\mathcal{S}_{\f_0,f,F}(X_T)$ is stable under finite maximum, we can moreover assume that the sequence $\{\f^j\}$ is increasing with $\f^j\leq M_0$ on $X$. 
		It follows from \eqref{eq: bound integral} that
		\begin{align*}
		\int_G \f^j_tfdV\leq \int_G\f_0 fdV+Ct, \,\forall \, t\in(0,T), 
		\end{align*}
		for a constant $C=C(M)>0$ independent of the sequence $\{\varphi^j\}$. For $t=t_0$, letting $j\rightarrow+\infty$, we obtain
		\begin{align*}
		\int_G U_{t_0}^*fdV\leq \int_G \f_0fdV+Ct_0, 
		\end{align*}
		thanks to a classical theorem of Lelong (see e.g.~\cite[Proposition 1.40]{guedj2017degenerate}). Note that the sequence $\{\f_t^j\}$ depends on $t_0$, but the constant $C$ does not. Therefore the claim \eqref{claim28} follows.
		
		Let now $u_0\in \PSH(X,\omega_{0})$ be any cluster point of $U_t^*$ as $t\rightarrow 0$. We can assume that $U_t^*$ converges to $u_0$ in $L^q(X,dV)$ for any $q>1$. Then $U_t^* f$ converges to  $u_0f$ in $L^1(X)$. Thus, the claim above ensures that
		\begin{align*}
		\int_G u_0fdV\leq \int_G \f_0fdV.
		\end{align*} 
		We infer that $u_0\leq \f_0$ almost everywhere on $G$ with respect to $fdV$, hence everywhere on $G$ by the assumption on $f$. Letting $M\to +\infty$, we thus conclude that $\limsup_{t\rightarrow 0}U_t^*= \f_0$ on $\Omega$.
	\end{proof}
	
	\subsection{The envelope is a subsolution}
	We now consider the set of subsolutions which are locally uniformly Lipschitz. 
	\begin{definition}
		Let $\kappa$ be a fixed positive constant. We let  $\mathcal{S}^{\kappa}_{\f_0,f,F}(X_T)$  denote the set of all functions $\f\in\mathcal{S}_{\f_0,f,F}(X_T)$ such that, for all $t\in (0,T)$, $x\in \Omega$,
		\begin{align*}
		t\partial_t\f(t,x)\leq \kappa-\kappa(\rho(x)+\chi(x)).
		\end{align*}
		Set
		\begin{align*}
		U^{\kappa}:=U^{\kappa}_{\f_0,f,F, X_T}:=\sup\{\f: \f \in \mathcal{S}^{\kappa}_{\f_0,f,F}(X_T)\}.
		\end{align*}
	\end{definition}
	\begin{proposition}\label{pro2}
		For all $0<S<T$ we have $U^\kappa_{\f_0,f,F,X_S}=U^\kappa_{\f_0,f,F,X_T}$ in $X_S$.
	\end{proposition}
	\begin{proof}
		The proof is the same as that of Proposition~\ref{pro1}.
	\end{proof}
	\begin{theorem}\label{thmlipk}
		We have, for all $\kappa>0$ and $(t,x)\in X_T$,
		\begin{align*}
		t \partial_t U^{\kappa} (t,x)\leq L_U-L_U(\rho(x)+\chi(x)),
		\end{align*}
		where $L_U$ is the constant defined in Theorem~\ref{thmlips}. 
	\end{theorem}
	\begin{proof}
		The proof is the same as that of Theorem~\ref{thmlips}. In fact,  if $\varphi \in \mathcal{S}_{\varphi_0,f,F}^{\kappa}(X_T)$ then the function $u^s$ in the proof of Theorem~\ref{thmlips} satisfies 
		\[
		t\partial_t u^s \leq \alpha_s (\kappa-\kappa (\rho+\chi)) \leq \kappa-\kappa (\rho+\chi)
		\]
		because $0<\alpha_s\leq 1$. It follows that $u^s\in \mathcal{S}_{\varphi_0,f,F}^{\kappa}(X_T)$ and we argue as in the proof of Theorem~\ref{thmlips} to conclude.
	\end{proof}
	\begin{theorem}\label{thmsub}
		The upper envelope $U$ is a pluripotential subsolution to \eqref{cmaf}  in $X_T$. 
	\end{theorem}
	\begin{proof}
		
		We will first show that $U^{\kappa}=(U^{\kappa})^*$ is a subsolution to \eqref{cmaf}. Indeed, Choquet's lemma implies that there exists a sequence $\{\f^j\}$ in $\mathcal{S}^{\kappa}_{\f_0,f,F}(X_T)$ such that
		\begin{align*}
		(U^{\kappa})^*=\left(\sup_{j\in \mathbb{N}}\f^j\right)^* \quad\text{in} \;  X_T.
		\end{align*}
		Since $\mathcal{S}^{\kappa}$ is stable under finite maximum, we can assume that the sequence $\{\f^j\}$ is non-decreasing. We now claim that
		\begin{align*}
		dt\wedge(\omega_{t}+\dc\f^j_t)^n\rightarrow dt\wedge(\omega_{t}+\dc(U^{\kappa})^*_t)^n
		\end{align*}
		in the sense of measures in $(0,T)\times\Omega$.
		
		Let $K$ be a relatively compact open subset of $\Omega$ and $J$ be a compact interval of $(0,T)$. Then there exists a constant $C=C(J,K)>0$ such that for all $j\in\mathbb{N}$, $\f^j(t,x)-Ct$ is decreasing in $t\in J$, for any $x\in K$. Moreover, the sequence of functions $\f^j$  increases towards $u$, so for any $x\in K$, $u(t,x)-Ct$ is decreasing in $t$. Thus for each $x\in K$, there exists a countable subset $E_x\subset J$   such that  $u(\cdot,x)$ is continuous on $J\backslash E_x$. Now set
		\begin{align*}
		E:=\{(t,x)\in J\times K: t\in E_x \}.
		\end{align*}
		Note that $E$ has zero $(2n+1)$-dimensional Lebesgue measure by using Fubini's theorem. Let $N$ be the set of $t\in J$ such that $E_t=\{x\in K: (t,x)\in E \}$ has positive Lebesgue measure. We must have that $N$ has zero Lebesgue measure. 
		Thus for any $t\in J':=J\backslash N$, the set $E_t$ has zero Lebesgue measure, and 
		$\lim_{s\to t}u(s,x)=u(t,x)$ for all $x\in K\backslash E_t$. Fixing $(t,x)\in J'\times K$, we want to show that
		\begin{align}\label{ine_218}
		\limsup_{(s,y)\to (t,x)}u(s,y)\leq (U^\kappa_t)^*(x),
		\end{align}
		where the upper semicontinuous regularization on the RHS is in the $x$-variable only. Since the problem is local we may assume that the functions $u_s$ are psh and negative in a neighborhood $B(x,2r)\subset K$. Fix $\delta\in (0,r)$. For $y$ so close to $x$  that $B(x,r)\subset B(y,r+\delta)$ we have
		\begin{align*}
		\f^j(s,y)&\leq\frac{1}{\Vol(B(y,r+\delta))}\int_{B(y,r+\delta)}\f^j(s,z)dV(z) \\
		& \leq\frac{1}{\Vol(B(y,r+\delta))}\int_{B(y,r+\delta)}u(s,z)dV(z).
		\end{align*}
		Letting $j\to+\infty$ we get
		\begin{align*}
		u(s,y)&\leq\frac{1}{\Vol(B(y,r+\delta))}\int_{B(y,r+\delta)}u(s,z)dV(z)\\
		&\leq \frac{\Vol(B(x,r))}{\Vol(B(y,r+\delta))}\frac{1}{\Vol(B(x,r))}\int_{B(x,r)}u(s,z)dV(z).
		\end{align*}
		Since $\lim_{s\to t} u(s,z)=u(t,z)$ for almost every $z\in B(x,r)\subset K$, Fatou's lemma yields
		\begin{align*}
		\limsup_{(s,y)\to (t,x)}u(s,y)\leq \frac{\Vol(B(x,r))}{\Vol(B(x,r+\delta))} \frac{1}{\Vol(B(x,r))}\int_{B(x,r)}u(t,z)dV(z). 
		\end{align*}
		Now, we first let $\delta\to 0$ and then $r\to 0$ to obtain the desired inequality~\eqref{ine_218} by definition of $U^\kappa$. The reverse inequality is clear, hence we get the equality. Therefore, for each $t\in J'$ we have that  $\f_t^j$ increase almost everywhere towards $(U^\kappa_t)^*=(U^\kappa)^*_t$ on $K$, so Bedford-Taylor's convergence theorem yields
		\begin{align*}
		(\omega_t+\dc \f_t^j)^n\rightarrow(\omega_t+\dc (U^\kappa)^*_t)^n 
		\end{align*}
		in the weak sense of measures in $K$. Thus the claim follows directly from  Fubini's theorem.
		
		On the other hand, for each $x\in K$ fixed, the sequence $\{ \partial_{t}\f^j(t,x)+F(t,x,\f^j(t,x))\}$ converges to $\partial_{t}(U^{\kappa})^*(t,x)+F(t,x,(U^{\kappa})^*(t,x))$ in the sense of distributions in $J$, with the later being bounded in $J\times K$. Applying Proposition \ref{p26} we obtain
		\begin{align*}
		\lim_{j\rightarrow +\infty}e^{\partial_{t}\f^j+F(t,\cdot,\f^j)}fdt\wedge dV\geq e^{\partial_{t}(U^{\kappa})^*+F(t,\cdot,(U^{\kappa})^*)}fdt\wedge dV
		\end{align*}
		in the weak sense of measures in $J\times K$. It thus follows that $(U^{\kappa})^*$ is a pluripotential subsolution to~\eqref{cmaf} in $\Omega_T$, and hence $(U^{\kappa})^{*}\in \mathcal{S}^\kappa_{\f_0,f,F}(X_T)$. We thus deduce that $U^\kappa=(U^\kappa)^*$.  
		
		We have shown that, for some $\kappa_0>0$,  $U^{\kappa}=U^{\kappa_0}$ for all $\kappa>\kappa_0$ (by Theorem~\ref{thmlipk}).  It thus remains to prove that $U=U^{\kappa_0}$ in $X_T$. We first assume that 
		\[
		\f_0=P_{\omega_0}h:=\sup\{\psi \in \PSH(X,\omega_0): \psi\leq h \}
		\]
		for some continuous function $h$. Fix $0<S<T$, $s>0$ sufficiently small, and $\f\in \mathcal{S}_{\f_0,f,F}(X_T)$. For $(t,x)\in [0,S]\times X$, we define
		\begin{align*}
		u^s(t,x):=\alpha_s\f(t+s,x)+(1-\alpha_s)g(t+s)\dfrac{\rho(x)+\chi(x)}{2}-Cs(t+1)-\eta(s),
		\end{align*}  
		where  $\alpha_s=1-(\delta_0g(0))^{-1}s$, $\eta(s):=\sup_{X}(\alpha_s\f_s-h)$ and 
		$$C=(\delta_0g(0))^{-1}C_1(T+1)+2\delta_0^{-1}M_F+n|\log (g(T))| + |c_1|,$$ 
		with $C_1>0$ defined in \eqref{const_c1}.
		We compute
		\begin{align*}
		\omega_{t}+\dc u^s(t,\cdot)=&\alpha_s(\omega_{t+s}+\dc\f_{t+s})+\dfrac{1-\alpha_s}{2}(\omega_{t+s}+g(t+s)\dc\rho)\\
		&+\dfrac{1-\alpha_s}{2}(\omega_{t+s}+g(t+s)\dc\chi)+\omega_{t}-\omega_{t+s}.
		\end{align*}
		It follows from the assumption \eqref{estimate_ome} that
		$$\omega_{t}-\omega_{t+s}\geq -s\Theta.$$
		Since $\theta+\dc\chi\geq 2\delta_0\Theta$ we  thus obtain \begin{align*}
		\dfrac{1-\alpha_s}{2}(\omega_{t+s}+\dc\chi)+\omega_{t}-\omega_{t+s}&\geq \frac{(\delta_0g(0))^{-1}}{2}sg(t+s)(\theta+\dc\chi)-s\Theta
		\geq 0.
		\end{align*}
		We thus get
		\begin{align*}
		(\omega_{t}+\dc u^s(t,\cdot))^n&\geq(\alpha_s(\omega_{t+s}+\dc \f_{t+s})+(1-\alpha_s)g(t+s)(\theta+\dc\rho)/2)^n\\
		&\geq e^{\alpha_s(\partial_{t}\f_{t+s}+F(t+s,\cdot,\f_{t+s}))+(1-\alpha_s)(n\log g(t+s)+c_1)}fdV
		\end{align*}
		where we apply Lemma~\ref{logconcave} in the last line.
		Since the function $(t,r)\mapsto F(t,\cdot,r)$ is uniformly Lipschitz, it follows that
		\begin{align*}
		\alpha_sF(t+s,\cdot,\f_{t+s})&=F(t+s,\cdot,\f_{t+s})-(1-\alpha_s)F(t+s,\cdot,\f_{t+s})\\
		&\geq  F(t,\cdot,\f_{t+s})-\kappa_Fs-(\delta_0g(0))^{-1}sM_F.
		\end{align*}
		Since $\f\geq \underline{u}$ on $X_T$ we have
		\begin{align*}
		(1-\alpha_s)\f_{t+s}\geq (1-\alpha_s)g(t+s)\frac{\rho+\chi}{2}-(\delta_0g(0))^{-1}sC_1(T+1).
		\end{align*}
		Consequently, it follows from the choice of $C$ that $\f_{t+s}\geq u^s(t,\cdot)$ for all $t\in [0,S]$. Therefore,  
		\begin{align*}
		\alpha_sF(t+s,\cdot,\f_{t+s})
		&\geq F(t,\cdot,u^s(t,\cdot))-s(\kappa_F+(\delta_0g(0))^{-1}M_F),
		\end{align*}
		since the function $r\mapsto F(\cdot,\cdot,r)$ is increasing.	
		Observe now that $$\alpha_s\partial_t\f_{t+s}= \alpha_s\partial_t u^s+Cs-\alpha_sg'(t+s)\frac{\rho+\chi}{2}.$$
		It thus follows  from the choice of $C$ and the estimates above that
		\begin{align*}
		(\omega_{t}+\dc u_t^s)^n\geq e^{\partial_{t} u_t^s+F(t,\cdot,u^s(t,\cdot))}fdV,
		\end{align*}
		which means that $u^s$ is a pluripotential subsolution to~\eqref{cmaf}. By definition of $u^s$ we have  $u^s(0,\cdot)\leq h$ on $X$ since $\sup_X\rho=\sup_X\chi=0$. Since $u^s(0,\cdot)$ is $\omega_0$-psh, we infer  $u^s(0,\cdot)\leq \f_0=P_{\omega_0}h$ on $X$. 
		It follows that $u^s\in\mathcal{S}_{\f_0,f,F}(X_S)$, and hence $u^s\in\mathcal{S}^{\kappa}_{\f_0,f,F}(X_S)$ for some $\kappa>0$ large enough. Therefore, $u^s\leq U^{\kappa}=U^{\kappa_0}$ in $X_S$ by Proposition~\ref{pro2}. On the other hand it follows from Hartogs' Lemma that $\lim_{s\to 0}\eta(s)\leq 0$. Letting $s\rightarrow 0$ we get $\f\leq U^{\kappa_0}$  in $X_S$. Finally, letting $S\rightarrow T$ to obtain $\f\leq U^{\kappa_0}$, so $U\leq U^{\kappa_0}$ on $ X_T$ (see Proposition \ref{pro1}). Therefore $U=U^{\kappa_0}$ is the maximal subsolution to~\eqref{cmaf} with initial data $\f_0$. 
		
		We now remove the extra assumption on $\f_0$. Let $\{h_j\}$ be a sequence of continuous functions decreasing to $\f_0$. Then
		$\f_0^j=P_{\omega_0}(h_j)$ is a decreasing sequence of $\omega_0$-psh functions converging to $\f_0$. We thus obtain that the upper envelope $U^j:=U_{\f_0^j,f,F,X_T}$ is also a subsolution to \eqref{cmaf} by the previous arguments. We also provide a uniform Lipschitz constant for $U^j$. Since $\f_0^j$ decreases to  $\f_0$, we have that $ U^j$ decreases to some $V\in\mathcal{P}(X_T)$ which is a subsolution to \eqref{cmaf} and $U\leq V$. On the other hand we see that $V\big|_{\{0\}\times X}\leq \f_0$. Hence $V=U$. 
	\end{proof}

	\subsection{The envelope is locally uniformly semi-concave in time}
	
	\begin{theorem}\label{thmscc}
		There exists a uniform constant $C_U>0$ such that 
		\begin{align}\label{loc}
		t^2\partial_t^2 U(t,x)\leq C_U-C_U(\rho(x)+\chi(x)),
		\end{align}
		in the sense of distributions in $X_T$.
	\end{theorem}
	
	\begin{proof} 
		Fix $0<T'<T$ and $\e_0>0$ small enough such that $(1+\e_0)T'<T$, $s\in[1-\e_0,1+\e_0]$.
		Set, for any $(t,x)\in X_{T'}$,
		\begin{align*}
		u^s(t,x):=\alpha_s\dfrac{s^{-1}U(st,x)+sU(s^{-1}t,x)}{2}+(1-\alpha_s)g(t)\dfrac{\rho(x)+\chi(x)}{2}-C|s-1|^2(t+1),
		\end{align*}
		where $\alpha_s=1-A(s-1)^2$ for $A>0$ a uniform constant to be chosen later, and
		\begin{align*}
		C:=(A+1)C_0+\kappa_FT+An|\log (g(T)) + c_1|.
		\end{align*}
		We are going to prove that $u^s$ is a subsolution to \eqref{cmaf}. We compute
		\begin{align*}
		\omega_t+\dc u^s(t,\cdot)=&\dfrac{\alpha_s}{2}\left(\dfrac{1}{s}(\omega_{st}+\dc U_{st})+s(\omega_{s^{-1}t}+\dc U_{s^{-1}t})\right)+(1-\alpha_s)\dfrac{\omega_t+g(t)\dc\rho}{2}\\
		&+\alpha_s\omega_t-\dfrac{\alpha_s}{2}(s^{-1}\omega_{st}+s\omega_{s^{-1}t})+(1-\alpha_s)\dfrac{\omega_t+g(t)\dc\chi}{2}.
		\end{align*}	
		Consider, for $s\in[1-\varepsilon_0,1+\varepsilon_0]$,
		\begin{align*}
		h(s):=\omega_t-\frac{1}{2}(\omega_{st}+\omega_{s^{-1}t}).
		\end{align*} 
		We have $h(1)=h'(1)=0$, and $|h''(s)|\leq (T+2)^2\Theta$ on $[1-\varepsilon_0,1+\varepsilon_0]$. Hence
		\begin{align*}
		\alpha_s\omega_t-\dfrac{\alpha_s}{2}(s^{-1}\omega_{st}+s\omega_{s^{-1}t})\geq -(T+2)^2\Theta.
		\end{align*}
		Recall that $\chi$ is a $\theta$-psh function on $X$ such that $\theta+\dc\chi\geq2\delta_0\Theta$. If we take $A>0$ such that $A\geq (T+2)^2(\delta_0g(0))^{-1}$ then
		\begin{align*}
		\alpha_s\omega_t-\dfrac{\alpha_s}{2}(s^{-1}\omega_{st}+s\omega_{s^{-1}t})+(1-\alpha_s)\dfrac{\omega_t+\dc\chi}{2}\geq 0,
		\end{align*}
		hence,
		\begin{align*}
		\omega_t+\dc u^s(t,\cdot)\geq& \dfrac{\alpha_s}{2}\left(\dfrac{1}{s}(\omega_{st}+\dc U_{st})+s(\omega_{s^{-1}t}+\dc U_{s^{-1}t})\right)
		\\&+(1-\alpha_s)\dfrac{\omega_t+g(t)\dc\rho}{2}.
		\end{align*}
		It follows from Theorem~\ref{thmsub} that $U$ is a pluripotential subsolution to~\eqref{cmaf}. We then have for almost every $t\in(0,T')$
		\begin{align*}
		(s^{-1}(\omega_{st}+\dc U_{st}))^n\geq e^{\partial_{\tau}U_{st}+F(st,\cdot,U_{st})-n\log s}fdV,
		\end{align*}
		and 
		\begin{align*}
		(s(\omega_{s^{-1}t}+\dc U_{st}))^n\geq e^{\partial_{\tau}U_{s^{-1}t}+F(s^{-1}t,\cdot,U_{s^{-1}t})+n\log s}fdV.
		\end{align*}
		Combining these together with Lemma~\ref{logconcave}, we obtain
		\begin{align*}
		(\omega_t+\dc u^s(t,\cdot))^n\geq e^{\alpha_s(a(s)+a(s^{-1}))+(1-\alpha_s)(n\log g(t)+c_1)} fdV,
		\end{align*}
		where 
		\begin{align*}
		a(s)=\dfrac{\alpha_s}{2}(\partial_{\tau}U_{st}+F(st,\cdot,U_{st})).
		\end{align*}
		Since $F$ is a convex function in $r$ we get
		\begin{align}\label{est_Fcc}
		\dfrac{1}{2}F(st,\cdot,U_{st})+\dfrac{1}{2}F(s^{-1}t,\cdot,U_{s^{-1}t})&\geq F\left(\dfrac{(s+s^{-1})t}{2},\cdot,\dfrac{U({st},\cdot)+U({s^{-1}t},\cdot)}{2}\right).
		\end{align}
		Now we  use the same arguments as in Theorem~\ref{thmlips} to show that for each $t\in[0,T)$,
		\begin{align}\label{est_U}
		\dfrac{U({st},\cdot)+U({s^{-1}t},\cdot)}{2}\geq u^s(t,\cdot).
		\end{align}
		It is equivalent to show that
		\begin{align}\label{est_cc1}
		\frac{1}{2}\left[ (1-s^{-1}\alpha_s)U_{st}+(1-s\alpha_s)U_{s^{-1}t}\right]\geq (1-\alpha_s)g(t)\frac{\rho+\chi}{2}-C(t+1)(s-1)^2.
		\end{align}
		The left-hand side can be rewritten as
		\begin{align*}
		\frac{1}{2}\left[(1-s^{-1}(1-A(s-1)^2))U_{st}+(1-s(1-A(s-1)^2))U_{s^{-1}t} \right]\\
		=\frac{1}{2}\left[ (s-1)(U_{st}-U_{s^{-1}t})+(A-1)s^{-1}(s-1)^2U_{st}+As(s-1)^2U_{s^{-1}t}\right].
		\end{align*}
		By Theorem~\ref{thmlips}, we have
		\begin{align*}
		(s-1)(U_{st}-U_{s^{-1}t})\geq 2(s-1) ^2(L_U(\rho+\chi)-L_U).
		\end{align*}
		The same arguments as in the proof of Theorem~\ref{thmlips} give 
		\begin{align*}
		(A-1)s^{-1}U_{st}\geq AU_{st}\geq(A-L_U/g(0))g(t)\frac{\rho+\chi}{2}-C_1(T+1),\\
		AsU_{s^{-1}t}\geq (A+1)U_{s^{-1}t}\geq  (A-L_U/g(0))g(t)\frac{\rho+\chi}{2}-C_1(T+1).
		\end{align*}
		where $A$ is large enough so that $(A-L_U/g(0))/(A+1)\geq (1+\delta_0g(0))^{-1}\gamma$ ($\gamma=1\e_0\kappa_gTg(0)^{-1}$).
		Combining these estimates, it follows from the choice of $C$ that \eqref{est_cc1} holds.
		Since $F$ is non-decreasing in $r$ and uniformly Lipschitz in $t$,  
		it follows from~\eqref{est_Fcc} and \eqref{est_U} that
		\begin{align*}
		\dfrac{1}{2}F(st,\cdot,U_{st})+\dfrac{1}{2}F(s^{-1}t,\cdot,U_{s^{-1}t})& \geq F\left(t,\cdot,\dfrac{U({st},\cdot)+U({s^{-1}t},\cdot)}{2}\right)-\kappa_Ft\left(\frac{s+s^{-1}}{2}-1\right)\\
		&\geq F(t,\cdot,u^s(t,\cdot))-\kappa_FT(s-1)^2,
		\end{align*}
		hence
		\begin{align*}
		\dfrac{\alpha_s}{2}\left(F(st,\cdot,U_{st})+F(s^{-1}t,\cdot,U_{s^{-1}t})\right)& \geq   F(t,\cdot,u^s(t,\cdot))-\kappa_FT(s-1)^2- AM_F(s-1)^2.
		\end{align*}
		Therefore, we obtain
		\begin{align*}
		a(s)+a(s^{-1})+(1-\alpha_s)(n\log g(t)+c_1)\geq \partial_{\tau}u^s(t,\cdot)+F(t,\cdot,u^s(t,\cdot)).
		\end{align*}
		On the other hand, the choice of $C$ ensures, for any $(t,x)\in X_{T'}$, that
		\begin{align*}
		u^s(0,x)&\leq\f_0(x) -C(s-1)^2+(1-\alpha_s)g(0)\frac{\rho(x)+\chi(x)}{2}-\left(1-\dfrac{s+s^{-1}}{2}\alpha_s\right)\f_0(x)\\
		&\leq\f_0(x) -C(s-1)^2+A(s-1)^2g(0)\frac{\rho(x)+\chi(x)}{2}-(A+1)(s-1)^2\f_0(x)\\
		&\leq \f_0(x)-C(s-1)^2+(A+1)(s-1)^2\left(\frac{A}{A+1}g(0)\frac{\rho(x)+\chi(x)}{2}-\f_0(x) \right)\\
		&\leq \f_0(x)-C(s-1)^2+(A+1)C_0(s-1)^2\leq \f_0(x).
		\end{align*}
		Therefore, we conclude that $u^s\in\mathcal{S}_{\f_0,f,F}(X_{T'})$, so we obtain for any $(t,x)\in X_{T'}$ that
		\begin{align*}
		\alpha_s\dfrac{s^{-1}U(st,x)+sU(s^{-1}t,x)}{2}-U(t,x)+&A/2(s-1)^2(\rho(x)+\chi(x))\\
		&\leq C(T+1)(s-1)^2,
		\end{align*}
		hence
		\begin{align*} 
		\dfrac{s^{-1}U(st,x)+sU(s^{-1}t,x)}{2}-U(t,x)+&A(s-1)^2(\rho(x)+\chi(x))\\
		&\leq (C(T+1)+2AM_0)(s-1)^2.
		\end{align*}
		From this, we obtain for all $(t,x)\in X_{T'}$,
		{ \begin{align*}
			\dfrac{U(st,x)+U(s^{-1}t,x)}{2}&-U(t,x)+A(s-1)^2(\rho(x)+\chi(x))\\
			&\leq (C(T+1)+(2A+1)M_0+2L_U-L_U(\rho(x)+\chi(x)))(s-1)^2.
			\end{align*} }
		Letting $s\rightarrow 1$ yields for all $(t,x)\in X_{T'}$
		\begin{align*}
		t^2\partial^2_{t}U(t,x)\leq (C(T+1)+(2A+1)M_0+5L_U)-(A+L_U)(\rho(x)+\chi(x)).
		\end{align*}
		We finally let $T'\rightarrow T$ and apply Proposition~\ref{pro1} to complete the proof.
	\end{proof}

	\section{Existence and Uniqueness}\label{exist}
	
	\subsection{Existence of solutions}
	We shall prove in this section that $U_{\f_0,f,F,X_T}$ is the unique pluripotential solution to the Cauchy problem (see Definition~\ref{defcp}).
	\begin{theorem}\label{existence}
		The upper envelope $U:=U_{\f_0,f,F,X_T}$ is a pluripotential solution to the Cauchy problem for the parabolic complex Monge-Amp\`ere equation \eqref{cmaf} in $X_T$. Moreover, $U$ is locally uniformly semi-concave in $(0,T)\times \Omega$.
		
		%	In particular, if $\f_0$ is continuous in $\Omega$ then $U$ is continuous in $[0,T)\times \Omega$. 
		%	{\orange Does $U_t$ have minimal singularities, i.e. $U_t-V_{\omega_t}$ is bounded on $X$?}
	\end{theorem}
	
	\begin{proof}
		We have shown in Theorem~\ref{thmscc} that $U$ is locally uniformly semi-concave in $t\in (0,T)$, and $U\in \mathcal{S}_{\f_0,f,F}(X_T)$ and it satisfies the initial condition. It remains to show that $U$ solves the parabolic equation~\eqref{cmaf}. We apply a local balayage process to modify the function $U$ on a given "small ball" $B\Subset \Omega$ by constructing a new $\omega_{t}$-psh function $U_B$ so that it satisfies the local Monge-Amp\`ere flow $(\omega_t+\dc{U_B})^n=e^{\partial_{t}U_B(t,\cdot)+{F}(t,\cdot,U_B(t,\cdot))}fdV$ on $B_T=(0,T)\times B$, ${U_B}\geq U$ on $B_T$ and ${U_B}=U$ on $ X_T\setminus B_T$. 
		
		Indeed, we choose complex coordinates $z=(z_1,\cdots,z_n)$ identifying $B$ with the complex unit ball $\mathbb{B}\subset \C^n$. 
		We can write $\omega_t=\dc g_t$ in a local holomorphic coordinate chart $B\subset X$, for some smooth local potential $g_t$. Set $\tilde{f}=f\circ z^{-1}\in L^p(\mathbb{B})$ and $d\tilde{V}$ is the restriction of the volume $dV$ to $B$.
		We consider the following complex Monge-Amp\`ere flow
		\begin{align}\label{cmaf1}
		dt\wedge(\dc u_t)^n=e^{\partial_{t}u(t,\cdot)+\tilde{F}(t,\cdot,u(t,\cdot))}\tilde{f}d\tilde{V}\wedge dt
		\end{align}		
		in $\mathbb{B}_T:=(0,T)\times\mathbb{B}$ with the Cauchy-Dirichlet boundary data $h$ being the restriction of $U$ defined on the parabolic boundary of $\mathbb{B}_T$ denoted by $\partial_P\mathbb{B}_T:=([0,T)\times\partial\mathbb{B})\cup(\{0\}\times\mathbb{B})$. Here $\tilde{F}(t,x,r)=F(t,x,r-g_t(x))-\partial_tg_t$  satisfies the same assumptions as $F$. We have shown that $h$ is locally uniformly Lipschitz (see Theorem~\ref{thmlips}) and locally uniformly semi-concave (see Theorem~\ref{thmscc}) i.e. for all $0<T'<T$, and for all $(t,z)\in (0,T')\times\partial\mathbb{B}$, there exist constants $L=L(T')$, and $C=C(T')$ such that
		\begin{align}\label{assumh}
		t|\partial_th(t,z)|\leq L,\qquad t^2|\partial^2_th(t,z)|\leq C.
		\end{align}
		Using mollifiers we can find a sequence $h^j$ of continuous functions on $[0,T)\times\partial\mathbb{B}$ such that $h^j$ decreases pointwise to $h$. The function $h^j$ is thus the Cauchy-Dirichlet boundary data satisfying the same assumption~\eqref{assumh} as $h$. 
		%Let $(h^j)$ be a sequence of continuous functions  which  decreases toward $U$ on $[0,T)\times \mathbb{B}\cup \{0\}\times\partial \mathbb{B}$ such that $h^j=U$ on $[0,T)\times\partial \mathbb{B}$ (see \cite[Lemma 2.11]{GLZ1}). % {\orange please be careful! Check the details!}

		Then it follows from \cite[Theorem 6.4]{guedj2018pluripotential} that there exists a sequence of functions $u^j$ solving~\eqref{cmaf1} with the boundary data $h^j$. Moreover, $u^j$ is locally uniformly semi-concave in $t\in [0,T)$. %{\orange be careful, give all the details here!} 
		Since $h^j$ decreases to $U\circ z^{-1}$ on $\partial_P\mathbb{B}_T$, so $U\circ z^{-1}\leq u^j$ and the sequence $u^j$ decreases to some function $v$. The function $v$ solves~\eqref{cmaf1} by using Proposition~\ref{115}, and $\limsup_{t\rightarrow 0}v(t,z)\leq U_0\circ z^{-1}$ in $\mathbb{B}$. But the comparison principle (see~\cite[Theorem 6.5]{guedj2018pluripotential}) ensures that $U\circ z^{-1}\leq v$ in $\mathbb{B}_T$. Hence $\lim_{t\rightarrow 0}v(t,z)=U_0\circ z^{-1}$.
		We then define 
		$$
		U_B(t,x)=\begin{cases} v(t,z(x))&\textrm{in}\, B_T\\
		U(t,x) &\textrm{in}\, X_T\setminus B_T.
		\end{cases}
		$$ as required. We infer that $U_B$ belongs to the set $\mathcal{S}_{\f_0,f,F}(X_T)$, the maximal property ensures that $U_B\leq U$, hence equality. Since $B$ is an arbitrary ball in $\Omega$, this shows that $U$ solves~\eqref{cmaf} on $\Omega_T$, hence on $X_T$.

		Moreover, by Theorem~\ref{thmlips} and Theorem~\ref{thmscc}, $U$ is locally uniformly uniformly Lipschitz and semi-concave in $t$.
	\end{proof}
	
	\subsection{The comparison principle}\label{sect: compa_princ} 
	We first establish a version of the comparison principle which requires relatively strong regularity assumptions:
	\begin{proposition}\label{comparison1}
		Let $\f$ (resp. $\psi$) be a subsolution (resp. supersolution) to \eqref{cmaf} with initial value $\f_0$ (resp. $\psi_0$). We assume that
		\begin{itemize}
			\item[a)] $\f$ is $\mathcal{C}^1$ in $t$ and continuous on $(0,T)\times\Omega$,
			\item[b)] $\psi$ is locally uniformly semi-concave in $t$	
			\item[c)]  $\f_t\rightarrow \f_0$ and $\psi_t\rightarrow\psi_0$ in $L^1(X)$, as $t\rightarrow 0$,
			\item[d)] for any $t\in[0,T)$, $\psi_t$ has minimal singularities,
			\item[e)] the function $(t,x)\mapsto \psi(t,x)$ is continuous on $[0,T)\times \Omega$.
		\end{itemize}
		Then 
		$$\f_0\leq \psi_0\Rightarrow\f\leq \psi \quad\text{in}\,\, X_T.$$
	\end{proposition}
	\begin{proof}
		Fix $0<T'<T$, in particular $T'<+\infty$. We shall prove that $\f\leq\psi$ on $[0,T']\times X$. The result thus follows by letting $T'\rightarrow T$. We fix $\lambda,\e>0$ sufficiently small. Set for $(t,x)\in [0,T']\times X$,
		\begin{align*}
		\f_{\lambda}(t,x):=(1-\lambda)\f(t,x)+\lambda g(t)\dfrac{\rho(x)+\chi(x)}{2},
		\end{align*}
		where $\rho,\chi$ are $\theta$-psh functions defined in~\eqref{rho}, \eqref{chi}. One can moreover impose $\chi <0$ to be smooth in the ample locus $\Omega=\text{Amp}\{\theta\}$, with analytic singularities, and such that $\chi(x)\rightarrow-\infty$ as $x\rightarrow\partial\Omega$. We will show that $\f_{\lambda}\leq \psi$ and we then let $\lambda\rightarrow 0$ to conclude the proof. Set
		$$w(t,x):=\f_{\lambda}(t,x)+\lambda g(0)\delta_0\chi(x)-\psi(t,x)-3\e t.$$ 
		Observe that by Lemma~\ref{12}, this function is upper semi-continuous on $[0,T']\times\Omega$. By the assumption d) we have $\f_\lambda(t,\cdot)\leq \psi(t,\cdot)+O(1)$ for each $t$.
		Since $\chi$ is continuous in $\Omega$ and tends to $-\infty$ on $\partial \Omega$, we have that $w$
		tends to $-\infty$ on $\partial\Omega$.  
		Hence $w$ attains its maximum at some point $(t_0,x_0)\in [0,T']\times \Omega$. 
		
		We want to show that $w(t_0,x_0)\leq 0$. Assume by contradiction that it is not the case i.e $w(t_0,x_0)>0$, with $t_0>0$. The set 
		$$K:=\{x\in \Omega: w(t_0,x)=w(t_0,x_0)\}$$
		is a compact subset of $\Omega$ since $w(t_0,x)$ tends to $-\infty$ as $x\rightarrow\partial\Omega$. The classical maximum principle ensures for all $x\in K$ that
		$$(1-\lambda)\partial_{t}\f(t_0,x)\geq \partial_{t}^{-}\psi(t_0,x)+3\e,$$
		since $g'(t)\geq 0$ for all $t$.
		The partial derivative $\partial_{t}\f(t,x)$ is continuous in $\Omega$ by assumption. 
		Since the function $t\mapsto \psi(t,x)$ is locally uniformly semi-concave, for any $t\in(0,T)$,
		the left derivative  $\partial_t^-\psi(t,\cdot)$ is upper semi-continuous in $\Omega$ (see Proposition \ref{t_der}). We can thus find  $\eta>0$ small enough  that, by introducing the open set containing $K$,    
		$$D:=\{x\in\Omega:w(t_0,x)>w(t_0,x_0)-\eta \}\Subset \Omega.$$
		We have for all $x\in D$,
		\begin{align}\label{41}
		(1-\lambda)\partial_{t}\f(t_0,x)>\partial_{t}^{-}\psi(t_0,x)+2\e.
		\end{align}  
		Set $u:=\f_\lambda(t_0,\cdot)+\lambda g(0)\delta_0\chi$ and $v=\psi(t_0,\cdot)$. We observe that $$\lambda\frac{\omega_t+\dc  g(t)\chi}{2}+\lambda g(0)\dc\chi\geq \lambda g(0)\delta_0\Theta+\lambda g(0)\delta_0\dc\chi\geq 0.$$
		Since $\f$ is a pluripotential subsolution to \eqref{cmaf}, we infer by using Lemma~\ref{logconcave} that
		\begin{align*}
		(\omega_{t_0}+\dc u)^n&\geq ((1-\lambda)(\omega_{t_0}+\dc\f_{t_0})+\lambda (g(t_0)\theta+g(t_0)\dc\rho)/2)^n\\
		&\geq e^{(1-\lambda)(\partial_t\f(t_0,\cdot)+F(t_0,\cdot,\f(t_0,\cdot)))+\lambda(n\log g(t) +c_1)}fdV\\
		&\geq e^{(1-\lambda)\partial_{t}\f(t_0,\cdot)+F(t_0,\cdot,\f(t_0,\cdot))-\lambda (M_F+|n\log g(t)+c_1|)}fdV	
		\end{align*}
		in the weak sense of measures in $D$. Choosing $\lambda$ so small that $$\lambda<\min_{[0,T]}\{(M_F+|n\log g(t)+c_1|)^{-1}\e\}.$$ It thus follows from~\eqref{41} and the increasing property of $F$ that
		\begin{align*}
		(\omega_{t_0}+\dc u)^n\geq e^{\partial_{t}^-\psi(t_0,\cdot)+F(t_0,\cdot,u)+\e}fdV
		\end{align*}
		in the weak sense of measures in $D$. On the other hand, $\psi$ is a pluripotential supersolution to~\eqref{cmaf}, thus
		\begin{align*}
		(\omega_{t_0}+\dc v)^n\leq e^{\partial_{t}^-\psi(t_0,\cdot)+F(t_0,\cdot,\psi(t_0,\cdot))}fdV
		\end{align*} in the weak sense of measures in $D$. The last two inequalities yield
		\begin{align*}
		(\omega_{t_0}+\dc u)^n\geq e^{F(t_0,\cdot,u)-F(t_0,\cdot,v)+\e}(\omega_{t_0}+\dc v)^n.
		\end{align*}
		Shrinking $D$ if necessary, we can assume that $u(x)>v(x)$ for any $x\in D$. We thus get
		\begin{align*}
		(\omega_{t_0}+\dc u)^n\geq e^{\e}(\omega_{t_0}+\dc v)^n
		\end{align*}
		in the sense of measures in $D$.
		
		Consider now $\tilde{u}:=u+\min_{\partial D}(v-u)$. We observe that $v\geq \tilde{u}$ on $\partial D$, hence the elliptic comparison principle (see Proposition~\ref{domination}) yields
		\begin{align*}
		\int_{\{v<\tilde{u}\}\cap D}e^{\e}(\omega_{t_0}+\dc v)^n&\leq \int_{\{v<\tilde{u}\}\cap D}(\omega_{t_0}+\dc u)^n
		\leq \int_{\{v<\tilde{u}\}\cap D}(\omega_{t_0}+\dc v^n).
		\end{align*}
		Thus $\tilde{u}\leq v$ almost everywhere in $D$ with respect to the measure $(\omega_{t_0}+\dc v)^n$. It thus follows from the domination principle (Proposition~\ref{domination}) that $\tilde{u}\leq v$ everywhere in $D$. In particular,
		\begin{align}\label{contradict}
		u(x_0)-v(x_0)+\min_{\partial D}(v-u)\leq \tilde{u}(x)-v(x)\leq 0.
		\end{align}
		On the other hand, since $K\cap\partial D=\varnothing$, we get $w(t_0,x)\leq w(t_0,x_0)$, for all $x\in\partial D$, hence
		\begin{align*}
		u(x)-v(x)< u(x_0)-v(x_0),
		\end{align*}
		contradicting~\eqref{contradict}. Therefore, we must have $t_0=0$, hence
		\begin{align*}
		(1-\lambda)\f+\lambda g(t)\frac{\rho+\chi}{2}+\lambda g(0)\chi-\psi-3\e t\leq \lambda\sup_X\left(g(0)\frac{\rho+\chi}{2}-\f_0\right),
		\end{align*} 
		in $[0,T]\times \Omega$. The right-hand side is finite thanks to Lemma~\ref{lem_initial}. Letting $\lambda\rightarrow 0$ we obtain $\f\leq \psi+3\e t$ in $[0,T']\times \Omega$, hence in $[0,T']\times X$. We thus conclude the proof by letting $\e\rightarrow 0$ and $T'\rightarrow T$.
	\end{proof}
	\begin{proposition}\label{domination} Fix a nonempty open subset of $D\Subset \Omega$.
		Let $u$, $v$ be $\theta$-psh functions, which are bounded in a neighborhood of $D$, such that
		$$\limsup_{D\ni x\rightarrow\partial D}(u-v)(x)\geq 0.$$
		Assume that $v$ has minimal singularities.	Then 
		\begin{align*}
		\int_{\{u<v\}\cap D}\textrm{MA}_\theta(v)\leq \int_{\{u<v\}\cap D}\textrm{MA}_\theta(u).
		\end{align*}
		Moreover, if $\text{MA}_{\theta}(u)(\{u<v\}\cap D)=0$ then $u\geq v$ in $D$.
	\end{proposition}
	\begin{proof} %We can assume that $u$ is bounded by replacing $u$ by $\max(u,V_\theta-C)$ {\orange It is no longer $\theta$-psh. Please be careful with copy-past?} for $C>0$ large enough. 
		Fix $\e>0$.	
		For each $C>0$ we set $u^C:=\max(u,V_\theta-C)$.
		Then the function $\max(u^C,v-\e)$ is $\theta$-psh with minimal singularities and coincides with $u^C$ in a neighborhood of $\partial D$. The boundary condition means that for any $\e>0$ the subset $\overline{\{u^C<v-\e\}}$ is compact in $D$.
		Let $\overline{\{u^C<v-\e\}}\Subset D'\Subset D$. 
		We claim that
		\begin{align}\label{43}
		\int_{D'}(\theta+\dc u^C)^n=\int_{D'}(\theta+\dc\max(u^C,v-\e))^n.
		\end{align}
		Indeed, set $w:=\max(u^C,v-\e)$, using local regularization of plurisubharmonic functions, we observe that $(\theta+\dc w)^n-(\theta+\dc u^C)^n=\dc S$ in the sense of currents on $D$, where $S:=(w-u^C)((\theta+\dc w)^{n-1}+\cdots+(\theta+\dc u^C)^{n-1})$ is a well-defined current with compact support in $D$. Pick any test function $\gamma$ which is identically $1$ in a neighborhood of the support of $S$. Then 
		\begin{align*}
		\int_{D'}\dc S=\int_{D'}\gamma \dc S=\int_{D'}S\wedge\dc\gamma=0,
		\end{align*}
		where we have known that $\dc\gamma=0$ on the support of $S$. This implies~\eqref{43}.
		
		On the other hand, we apply~\cite[Theorem 3.27]{guedj2017degenerate} to get
		\begin{align*}
		\int_{\{u^C<v-\e\}}(\theta+\dc\max(u^C,v-\e))^n=\int_{\{u^C<v-\e\}}(\theta+\dc v)^n.
		\end{align*}
		Combining this together with the equality~\eqref{43}, we obtain
		\begin{align*}
		\int_{\{u^C<v-\e\}\cap D'}\MA_{\theta}(v)&=\int_{D'}\MA_{\theta}(\max(u^C,v-\e))-\int_{\{u^C\geq v-\e\}}\MA_{\theta}(\max(u^C,v-\e))\\
		&\leq\int_{D'}\MA_{\theta}(u^C)-\int_{\{u^C> v-\e\}}\MA_{\theta}(\max(u^C,v-\e))\\
		&\leq\int_{D'}\MA_{\theta}(u^C)-\int_{\{u^C> v-\e\}}\MA_{\theta}(u^C),
		\end{align*} which implies
		\begin{align*}
		\int_{\{u^C<v-\e \}\cap D}\MA_{\theta}(v)\leq \int_{\{u^C\leq v-\e\}\cap D}\MA_{\theta}(u^C) 
		\end{align*} since $D'$ was taken arbitrarily.
		Letting $\e\rightarrow 0$ and then $C\to+\infty$ we obtain the required inequality.

		Arguing as in~\cite[Corollary 2.5]{boucksom2010monge}, we  complete the last statement. Let $\chi$ be a $\theta$-psh function defined in \eqref{chi}. Since $v$ has minimal singularities, we may also assume that $\chi\leq v$.
		For any $\e>0$ small enough, consider $v_{\e}:=(1-\e)v+\e\chi$, hence $\limsup_{D \ni x\rightarrow\partial D}(u(x)-v_{\e}(x))\geq 0$. Then
		\begin{align*}
		\e^n\int_{\{u<v_{\e}\}}\MA_{\theta}(\chi)\leq \int_{\{u<v_{\e}\}}\MA_{\theta}(v_{\e})\leq \int_{\{u<v_{\e}\}}\MA_{\theta}(u)=0,
		\end{align*}
		since $\chi\leq v$ implies that $\{u<v_{\e}\}\subset \{u<v\}$. On the other hand, $\MA_{\theta}(\chi)$ dominates Lebesgue measure. We deduce that $u\geq v_{\e}$ almost everywhere with respect to Lebesgue measure in the open set $D$, hence everywhere in $D$. The result follows by letting $\e\rightarrow 0$.
	\end{proof}
	We now slightly relax the hypothesis e) in Proposition~\ref{comparison1}.
	\begin{proposition} \label{comparision}
		Let $\f$ (resp. $\psi$) be a pluripotential subsolution (resp. supersolution) to~\eqref{cmaf} with initial value $\f_0$ (resp. $\psi_0$). We assume that
		\begin{itemize}
			\item[a)] $\f$ is $\mathcal{C}^1$ in $t$ and continuous on $(0,T)\times\Omega$,
			\item[b)] $\psi$ is locally uniformly semi-concave in $t$,	
			\item[c)]  $\f_t\rightarrow \f_0$ and $\psi_t\rightarrow\psi_0$, as $t\rightarrow 0$,
			\item[d)] for any $t\in(0,T)$, $\psi_t$ has minimal singularities,
			\item[e')] the function $(t,x)\mapsto \psi(t,x)$ is continuous on $(0,T)\times \Omega$.
		\end{itemize}
		Then 
		$$
		\f_0\leq \psi_0\Rightarrow\f\leq \psi \quad\text{in}\,\, X_T.
		$$
	\end{proposition}
	
	\begin{proof}
		We proceed as in the proof of Theorem \ref{thmlips}. We fix $s>0$ sufficiently small and consider
		$$
		v^s(t,x)=\psi(t+s,x)+Cs(t+1)-Cs\log\delta_0^{-1} s,
		$$
		and
		$$
		u^s(t,x):=\alpha_s\f(t,x)+(1-\alpha_s)g(t)\dfrac{\rho(x)+\chi(x)}{2}-Cs(t+1).
		$$
		Here $\alpha_s=1-As\in(0,1)$, $A>0$ is determined in Lemma \ref{l1}, the functions $\rho,\chi$ are defined in \eqref{rho}, \eqref{chi} and $C$ is a positive constant which will be chosen later. We want to show that for $C>0$ large enough, $u^s$ is a subsolution while $v^s$ is a supersolution to \eqref{cmaf} and $u^s(0,\cdot)\leq v^s(0,\cdot)$. We can then apply Proposition~\ref{comparison1} and let $s\rightarrow 0$ to complete the proof. We first observe that
		\begin{align*}
		\omega_{t+s}+\dc u^s&=\alpha_s(\omega_t+\dc\f_t)+\dfrac{1-\alpha_s}{2}(\omega_t+g(t)\dc\rho)\\
		&\quad +\dfrac{1-\alpha_s}{2}(\omega_t+g(t)\dc\chi)+\omega_{t+s}-\omega_t.
		\end{align*}
		By assumptions, we see that for $s>0$ 
		$\omega_{t+s}-\omega_t\geq -s\Theta.$
		Since $\theta+\dc\chi\geq 2\delta_0\Theta$ we  thus obtain 
		$$\dfrac{1-\alpha_s}{2}(\omega_t+g(t)\dc\chi)+\omega_{t+s}-\omega_t\geq Asg(t)\delta_0\Theta-s\Theta\geq  0.$$
		Hence,
		\begin{align*}
		(\omega_{t+s}+\dc u^s)^n&\geq(\alpha_s(\omega_t+\dc \f_t)+(1-\alpha_s)g(t)(\theta+\dc\rho)/2)^n\\
		&\geq e^{\alpha_s(\partial_{t}\f_t+F(t,\cdot,\f_t))+(1-\alpha_s)(n\log g(t)+c_1)}fdV
		\end{align*}
		applying Lemma~\ref{logconcave} in the last line. Since $\alpha_s=1+O(s)$ and $F$ is bounded from above and locally uniformly Lipschitz, by choosing $C>0$ large enough,  depending on $\delta_0,\kappa_F, M_F,c_1$ (see Theorem~\ref{thmlips}), we have
		\begin{align*}
		(\omega_{t+s}+\dc u^s)^n\geq e^{\partial_{t} u^s+F(t,\cdot,u^s)}fdV.
		\end{align*} 
		On the other hand, since $\psi$ is a supersolution, we have
		\begin{align*}
		(\omega_{t+s}+\dc v^s)^n&\leq e^{\partial_{t} v^s-Cs+F(t+s,\cdot,\psi(t+s,\cdot))}f dV\\
		&\leq e^{\partial_{t}v^s-Cs+F(t,\cdot,v^s(t,\cdot))+\kappa_Fs}fdV\\
		&\leq e^{\partial_t v^s+F(t,\cdot,v^s(t,\cdot))}fdV,
		\end{align*}
		where the second line follows from the Lipschitz condition and the increasing monotonicity of $F$,
		the last line follows from the choice of $C$. Up to increasing $C>0$ it follows from Lemma~\ref{l1} that for any $x\in X$, \begin{align*}
		u^s(0,x)&\leq(1-As)\f(0,x)+Ag(0) s(\rho+\chi)/2\\
		&\leq \psi_s(x)+Cs-Cs\log Ag(0) s=v^s(0,x).
		\end{align*} It then follows from Proposition~\ref{comparison1} that
		$u^s(t,x)\leq v^s(t,x)$ for all $(t,x)\in X_T$.
		Letting $s\rightarrow 0$, we finish the proof.
	\end{proof}
	
	\begin{lemma}\label{l1}
		With the same assumptions of $\psi$ as in Proposition~\ref{comparision}. Then there exist uniform constants $A>0$, $C>0$, and $t_0>0$ small enough such that for all $(t,x)\in (0,t_0)\times X$, 
		$$\psi(t,x)\geq (1-At)\psi_0(x)+C(t\log (Ag(0)t)-t)+Ag(0) t(\rho(x)+\chi(x))/2.$$
	\end{lemma}
	
	\begin{proof} 
		The proof is similar to that of~\cite[Lemma 3.14]{guedj2020pluripotential}. 
		We recall that the function $F$ satisfies the Lipschitz condition i.e., there exists a constant $\kappa_F>0$ such that, for all $t,t'\in [0,T/2]$, $x\in X$, $r\in \mathbb{R}$, 
		\begin{align*}
		|F(t,x,r)-F(t',x,r)|\leq\kappa_F|t-t'|.
		\end{align*}
		Set $t_0=\min(1,A,T/4)$ with $A>0$ under control.
		Fix $s>0$ sufficiently small and consider for $(t,x)\in (0,t_0)\times X$,
		\begin{align*}
		u^s(t,x)&:=(1-At)\psi_s(x)+At g(s)(\rho(x)+\chi(x))/2+n(t\log Ag(0)t-t)-Ct,\\
		v^s(t,x)&:=\psi(t+s,x)+2\kappa_F ts,
		\end{align*}
		where $\rho,\chi$ are  $\theta$-psh functions on $X$  defined in \eqref{rho},\eqref{chi}, and $C$ is a positive constant to be chosen later. We see that $u^s$ is of class $\mathcal{C}^1$ in $t$, and for any $t\in(0,t_0)$ fixed, $u^s(t,\cdot)$ is continuous in $\Omega$ since $\rho$ is continuous in $\Omega$ (see e.g.~\cite[Theorem 12.23]{guedj2017degenerate}). Now let $A\geq (\delta_0g(0))^{-1}$.
		One has, for any $t\in (0,t_0),$
		\begin{align*}
		\omega_{t+s}+\dc u^s(t,\cdot)&=(1-At)(\omega_{s}+\dc\psi_s)+At(\omega_{s}+g(s)\dc\rho)/2\\
		&+At(\omega_{s}+g(s)\dc\chi)/2+\omega_{t+s}-\omega_{s}.
		\end{align*}
		By hypothesis~\eqref{estimate_ome}, we have
		$\omega_{t+s}-\omega_{s}\geq -t\Theta$.
		Since $\theta+\dc\chi\geq 2\delta_0\Theta$, we thus obtain $$		At(\omega_{s}+g(s)\dc\chi)/2+\omega_{t+s}-\omega_{s}\geq Atg(s)\delta_0\Theta-t\Theta\geq 0,
		$$		 hence
		$$(\omega_{t+s}+\dc u^s)^n\geq (At(\omega_{s}+g(s)\dc\rho)/2)^n\geq  (Ag(0)t)^ne^{c_1}fdV,$$
		since $g(s)\geq g(0)$.
		We now choose $C=A\sup_X(\rho+\chi-2\f_0)/2+M_F-\min(c_1,0)$, hence 
		$$(\omega_{t+s}+\dc u^s)^n\geq e^{\partial_{t}u^s(t,\cdot)+F(t,\cdot,u^s(t,\cdot))}fdV.$$
		It follows from the definition that  $u^s(t,\cdot)$ converges in $L^1(X,dV)$ to $u^s(0,\cdot)=\psi_s$.
		On the other hand, since $\psi $ is a supersolution to \eqref{cmaf}, we have
		\begin{align*}
		(\omega_{t+s}+\dc v^s)^n&\leq e^{\partial_{t}\psi_{t+s}+F(t+s,\cdot,\psi(t+s,\cdot))}fdV\\
		&=e^{\partial_tv^s-2\kappa s+F(t+s,\cdot,\psi(t+s,\cdot))}fdV.
		\end{align*}
		Since the function $F$ is Lipschitz in $t$ and is increasing in $r$, for all $t,s\in(0,t_0), x\in X$,
		\begin{align*}
		F(t+s,x,\psi(t+s,x))&\leq F(t,x,\psi(t+s,x))+\kappa_Fs\\
		&\leq F(t,x,v^s(t,x))+\kappa_Fs,
		\end{align*}
		this yields
		\begin{equation*}
		(\omega_{t+s}+\dc v^s)^n\leq e^{\partial_{t}v^s(t,\cdot)+F(t,\cdot,v^s(t,\cdot))}fdV.\end{equation*}
		Since for each $s$ that $\psi_s$ is continuous on $\Omega$ hence $v^s$ is continuous on $[0,t_0)\times\Omega$ and it is clear that $v^s(t,\cdot)$ converges to $v^s(0,\cdot)=\psi_s$ in $L^1(X,dV)$ as $t\to 0$. We moreover see that for each $t$, $v^s(t,\cdot)$ has minimal singularities because $\psi_t$ has. 
		We can now apply Proposition~\ref{comparison1} %\textcolor{red}{Need to check that $u^s$ is more singular than $v^s$.} 
		and get $u^s\leq v^s$ on $(0,t_0)\times X$. Letting $s\rightarrow  0$ we have for all $(t,x)\in(0,t_0)\times X$,
		$$(1-At)\psi_0(x)+Ag(0)t(\rho(x)+\chi(x))/2+C(t\log (Ag(0)t)-t)\leq \psi(t,x),$$
		as desired.
	\end{proof}

	\subsection{Space regularity}
	In this section, we shall use the extra assumption  
	\begin{align}\label{eq: derivative omega t}
	\dot{\omega}_t\leq A\omega_t, \quad\forall\, t\in[0,T),
	\end{align}
	for some constant $A>1$. 
	
	\begin{theorem}
		\label{thm: min sing}
		Under the extra assumption~\eqref{eq: derivative omega t}, the envelope $U$ has minimal singularities and $U_t$ is moreover continuous in $\Omega$, for each $t\in (0,T)$.
	\end{theorem}
	
	\begin{proof} We first show that for each $t$, $U_t$ has minimal singularities.
		Observe that 
		$$ \eta_t:=e^{-At}\omega_t $$
		is decreasing in $t$. By~\cite[Theorem 6.1]{boucksom2010monge}, for each $t\in[0,T)$ there exists a unique $\eta_t$-psh function $\phi_t$ with full Monge-Amp\`ere mass such that 
		\begin{align*}
		(\eta_t+\dc\phi_t)^n=e^{ \phi_t+c_1}fdV,
		\end{align*}
		for $c_1>0$ so that $\sup_X\phi_0=0$. For $0<s\leq t$ we have 
		\begin{align*}
		(\eta_s+\dc\phi_t)^n\geq (\eta_t+\dc\phi_t)=e^{\phi_t+c_1}fdV. 
		\end{align*}
		It follows that $\phi_t$ is a subsolution  to $(\eta_s+\dc\phi_s)^n=e^{\phi_s+c_1}fdV $, so a classical comparison principle (see e.g.~\cite[Proposition 6.3]{boucksom2010monge}) ensures that $\phi_t\leq\phi_s$ for $s\leq t$. Therefore the function $t\mapsto\phi_t(x)$ is decreasing for all $x\in X$. Since $t\mapsto\eta_t$ is decreasing in $t$, we may assume that $\eta_t\geq \theta'$ for some big $(1,1)$-form $\theta'$. %By \cite[Theorem 4.1]{BEGZ}, there exists a unique $\theta'$-psh function $\rho'$ with minimal singularities such that 
		As explained in Section~\ref{subsol}, $\dot{\phi}_t=\partial_{t}\phi_t$ is well-defined almost everywhere and negative.		
		Set for any $(t,x)\in X_T$, $$u(t,x):=e^{At}\phi_t(x)-C_2(t+1)$$ where $C_2>0$ is a constant to be chosen later. Since $\sup_X\phi_0=0$ and $\phi_t$ is decreasing in $t$, we have $\phi_t\leq 0$ for all $t\in [0,T)$. 
		We infer for almost every $t\in[0,T)$, $$\dot{u}(t,\cdot)=e^{At}\dot{\phi}_t+Ae^{At}\phi_t-C_2\leq \phi_t-C_2,$$ hence
		\begin{align*}
		(\omega_t+\dc u_t)^n&=e^{nAt+\phi_t+c_1}fdV\geq e^{\dot{u}_t+F(t,\cdot,u(t,\cdot))}fdV,
		\end{align*} where for the last inequality we have chosen $C_2>0$ so big that $C_2>nAT+c_1+M_F$.
		Hence $u_t$ is a subsolution to \eqref{cmaf}. Moreover, we can choose $C_2>0$ so big that $u(0,\cdot)=\phi_0-C_2\leq\f_0 $ since $\f_0$ is a $\eta_0$-psh function with minimal singularities. Therefore, $u$ is a subsolution to the Cauchy problem, i.e. $u\in\mathcal{S}_{\f_0,f,F}(X_T)$. By \cite[Theorem 6.1]{boucksom2010monge} we have $\phi_t\geq V_{\eta_t}-C(t)$ for some time-dependent constant $C(t)$. Since $V_{\omega_t}=e^{At} V_{\eta_t}$, we thus infer $u_t\geq V_{\omega_t}-C'(t)$, hence $U_t\geq V_{\omega_t}-C'(t)$ for all $t\in[0,T)$. 
		
		It remains to show that for each $t\in (0,T)$, $U_t$ is continuous in $\Omega$. We have  that for each $t>0$, $\partial_t U+F(t,\cdot,U_t)\leq \kappa-\kappa(\rho+\chi)+M_F$ on $X$, for some constant $\kappa>0$. 
		The continuity of $U_t$ in $\Omega$ thus follows from~\cite[Theorem 3.2]{dang2021continuity}.
	\end{proof}		
	
	We now show that the solution constructed in Theorem~\ref{existence} is unique:
	
	\begin{theorem}\label{thm_unique} 
		Let $\Phi$ be a pluripotential solution to the Cauchy problem  for \eqref{cmaf} with initial data $\f_0$. 
		Assume that \eqref{eq: derivative omega t} holds and
		\begin{itemize}
			\item $\Phi$ is locally uniformly semi-concave in $(0,T)$;
			\item for each $t$, $\Phi_t$ has minimal singularities;
			\item $\Phi_t$ is continuous in $\Omega$. 
		\end{itemize}
		Then $\Phi=U$.
	\end{theorem}
	
	\begin{proof}
		Since $\Phi$ is locally uniformly Lipschitz in $t$ we infer that $\Phi$ is continuous on  $(0,T)\times \Omega$. We would like to apply Proposition~\ref{comparision}  but $U$ is not $\mathcal{C}^1$ in $t$. We are going to regularize it by taking convolution in $t$ as in \cite[Proposition 3.16]{guedj2020pluripotential}. Fix $0<T'<T$, $s>0$ near $1$. Set, for any $(t,x)\in X_{T'}$,
		\begin{align*}
		V^s(t,x):=\dfrac{\alpha_s}{s}U(st,x)+(1-\alpha_s)g(t)\dfrac{\rho(x)+\chi(x)}{2}-C|s-1|(t+1),
		\end{align*}
		where $\alpha_s$, $C$ are defined as in the proof of Theorem~\ref{thmlips} so that $V^s\in\mathcal{S}_{\f_0,f,F}(X_{T'})$. Let $\eta$ be a smooth function with compact support in $[-1,1]$ such that $\int\eta(t)dt=1$. Set, for $\e>0$ small, $\eta_\e(t)=\e^{-1}\eta(t/\e)$, and we define for any $(t,x)\in X_{T'}$
		\begin{align*}
		u^{\e}(t,x):=\int_{\R}V^s(t,x)\eta_{\e}(s-1)ds-B\e(t+1).
		\end{align*}
		Using the arguments as in~\cite[Proposition 3.16]{guedj2020pluripotential}, we will show that $u^\varepsilon$ is a pluripotential subsolution to~\eqref{cmaf} which is $\mathcal{C}^1$ in $t$. Indeed, 
		since $V^{s}$ is a pluripotential subsolution to~\eqref{cmaf}, Lemma~\ref{lem} yields, for any $t\in (0,T)$,
		\begin{align*}
		(\omega_{t}+\dc V^s(t,\cdot))^n\geq \exp(\partial_{t}V^s+F(t,\cdot,V^s(t,\cdot)))fdV.
		\end{align*}
		We know that the function $A\mapsto (\det A)^{1/n}$ is concave on the convex cone of non negative hermitian matrices. It follows from Jensen's inequality  that, for any $t\in (0,T)$,
		\begin{align*}
		(\det(\omega_{t}+&\dc u^{\e}))^{1/n} =\left(\det\left(\int_{\R}(\omega_{t}+\dc V^s)\eta_{\e}(s-1)ds\right)\right)^{1/n}\\
		&\geq \left(\int_{\R}\exp\left(\frac{1}{n}\left(\partial_{t}V^s(t,\cdot)+F(t,\cdot,V^s(t,\cdot))\right)\right)\eta_{\e}(s-1)ds\right) f^{1/n}\\
		&\geq\exp\left(\int_{\R}\frac{1}{n}\left(\partial_{t}V^s+F(t,\cdot,V^s(t,\cdot))\right)\eta_{\e}(s-1)ds\right)f^{1/n}\\
		&\geq\exp\left(\frac{1}{n}\left(\partial_{t}u^{\e}(t,\cdot)+B\e+F\left(t,\cdot,\int_{\R}V^s(t,\cdot)\eta_{\e}(s-1)ds\right)\right)\right)f^{1/n}\\
		&\geq\exp\left(\frac{1}{n}\left(\partial_tu^{\e}(t,\cdot)+F(t,\cdot,u^{\e})\right)\right) f^{1/n}.
		\end{align*}
		%	in the weak sense of distributions on $\Omega$. 
		The second line follows from the Main Theorem in~\cite{guedj2019weak}, 
		the third and fourth ones follow from the convexity of the exponential and $F$, 
		and the last one follows from  the monotonicity of $F$. Using Lemma~\ref{lem} again, we infer that $u^{\e}$ is a subsolution to \eqref{cmaf}. 	
		
		On the other hand, it follows from the proof of Theorem~\ref{thmlips} that $V^s(0,x)\leq \f_0$ on $X$. Hence $u^{\e}(0,\cdot)\leq\f_0$ on $X$ by taking $B>0$ large enough.	
		We can thus apply Proposition~\ref{comparision} to obtain $u^{\e}\leq \Phi$ on $[0,T']\times \Omega$, hence on $[0,T']\times X$. Letting $\e\rightarrow 0$ and $T'\rightarrow T$ we get $U\leq \Phi$ on $[0,T)\times X$ . Hence the equality holds.
	\end{proof}

	\section{Applications}\label{applies}
	
	We apply the tools we have developed in related geometrical settings. We first define and study the longtime behavior of the normalized K\"ahler-Ricci flow  on  manifolds of general type. We next prove the existence of a longtime solution of the on manifolds with nonnegative Kodaira dimension. 
	We then analyze the normalized K\"ahler-Ricci flow on varieties with semi-log canonical singularities and ample canonical bundle (stable varieties).
	
	\subsection{Manifolds of general type}
	We study in this section the (normalized) K\"ahler-Ricci flow on a manifold of general type. We try to run such flow in a weak sense  beyond the maximal existence time by using the results obtained in the previous section. We then study the long-time behavior of this flow. 
	
	Let $(X,\omega_0)$ be a compact K\"ahler manifold of general type i.e. the canonical divisor $K_X$ is big, and $\omega_{0}$ is a K\"ahler form. The normalized K\"ahler-Ricci flow is the evolution equation:
	\begin{align}\label{nkrf}
	\dfrac{\partial\vartheta_{t}}{\partial t}=-{\rm Ric}(\vartheta_{t})-\vartheta_{t},\quad \vartheta|_{t=0}=\omega_{0}.
	\end{align}
	Let $T$ be the maximal existence time of smooth flows which is defined by
	\begin{align*}
	T:=\sup\{t>0: e^{-t}\{\omega_{0}\}+(1-e^{-t})c_1(K_X) \,\, \textrm{ is K\"ahler}\}.
	\end{align*}
	Note that $T=+\infty$ if and only if the canonical divisor $K_X$ is nef. In this case, the normalized K\"ahler-Ricci flow exists in the classical (smooth) sense and converges to a singular K\"ahler-Einstein metric (cf.~\cite{tsuji1988existence,tian2006kahler}). 
	
	When $K_X$ is not nef, the flow has a finite time singularity at $T$ ($<+\infty$). The limiting class at $T$ of the flow
	\begin{align*}
	\alpha_T:=\lim\limits_{t\rightarrow T}\{\vartheta_t\}=e^{-T}\{\omega_{0} \}+(1-e^{-T})c_1(K_X)
	\end{align*}   
	is big and nef. In~\cite[Theorem 1.5]{collins2015kahler}, Collins and Tosatti showed that the flow $\vartheta_t$ exists on the maximal time interval $[0,T)$ and develops singularities precisely on the Zariski closed set $X\backslash \text{Amp}(\alpha_T)$ as $t\rightarrow T^-$. For $t>T$, the cohomology class $\{\vartheta_t\}$ is still big, but no longer nef, we can not continue the flow in the classical sense. 
	
	In~\cite[Section 10]{feldman2003rotationally}, Feldman, Ilmanen and Knopf have asked the question: can one define and construct weak solutions of the K\"ahler-Ricci flow beyond the singular time?  
	In~\cite[Theorem 4]{boucksom2012semipositivity}, Boucksom and Tsuji have constructed the normalized K\"ahler-Ricci flow on smooth projective varieties with pseudoeffective canonical class for all times. They used the discretization of the K\"ahler-Ricci flow and some algebro-geometric tools. In the end, they have conjectured the same result for the case of  general K\"ahler manifolds (see~\cite[Conjecture 1]{boucksom2012semipositivity}). T\^o \cite{to2021convergence} used the viscosity theory to show that the weak K\"ahler-Ricci flow exists for all time in the viscosity sense and converges to the unique singular K\"ahler-Einstein metric in the class $c_1(K_X)$ constructed in~\cite{eyssidieux2009singular,boucksom2010monge}. 
	
	In this section we show that the normalized K\"ahler-Ricci flow can be extended through a finite time singularity and understood in the weak sense of  parabolic pluripotential theory developed in Section~\ref{exist}. A compact K\"ahler manifold of general type turns out to be projective by a classical 	result of Moishezon.  Our result thus gives  an alternative approach to the existence of weak K\"ahler-Ricci flows previously obtained by Boucksom-Tsuji and T\^o.   The main point in our proof is  that the flow survives through a finite time singularity provided that the limiting class is big.  The argument is thus working also on a general K\"ahler manifold with pseudoeffective canonical class as will be shown later in Section \ref{sect: singularity}.  
	
	%A compact K\"ahler manifold of general type turns out to be projective by a classical 	result of Moishezon.  We shall apply the parabolic pluripotential theory established in  previous sections to answer the question of Feldman-Ilmanen-Knopf and prove a more general convergence result (see Theorem~\ref{thm_conv}) in the K\"ahler setting.
	\medskip
	
	Now let $\theta$ 
	be a smooth closed $(1,1)$-form representing $c_1(K_X)$. Set $\omega_{t}:=e^{-t}\omega_{0}+(1-e^{-t})\theta$. 	Since $dV_X$ is a smooth volume form on $X$, $-\mathrm{Ric}(dV)\in -c_1(K_X)$, and so there is a smooth function $f$ such that $\theta=-\mathrm{Ric}(dV_X)+\dc  f$. We then define $$\mu=e^fdV$$ which is a smooth positive volume with $\theta=\textrm{Ric}(\mu)$. Thus the normalized K\"ahler-Ricci flow~\eqref{nkrf} can be written as the complex Monge-Amp\`ere flow
	\begin{equation}\label{pcmae}
	(\omega_{t}+\dc\f_t)^n=e^{\dot{\f}_t+\f_t}d\mu.
	\end{equation}
	It has been shown in~\cite[Theorem 6.1]{boucksom2010monge} that there exists a unique $\theta$-psh function $\varphi_{KE}$ with minimal singularities such that
	\begin{align}
	(\theta+\dc\f_{KE})^n=e^{\f_{KE}}\mu.
	\end{align}
	The current $\omega_{KE}:=\theta+\dc\f_{KE}$ is called the {\em singular K\"ahler- Einstein metric}. It has bounded potentials and is smooth in $\Omega:={\rm Amp}(K_X)$, where it satisfies
	$${\rm Ric}(\omega_{KE})=-\omega_{KE}.$$
	We are going to show that the (normalized) pluripotential K\"ahler-Ricci flow~\eqref{pcmae} exists for all time and continuously deforms any initial K\"ahler form $\omega_{0}$ towards $\omega_{KE}$ on $\textrm{Amp}(K_X)$, as $t\rightarrow+\infty$. The result even holds for an initial datum $S_0$ which is a positive current with bounded potentials:

	\begin{theorem}\label{thm_conv}
		Let $\f_0$ be a bounded $\omega_0$-psh function. Then there exists a unique pluripotential solution $\f$ to~\eqref{pcmae} with initial data $\f_0$ for $t>0$. Furthermore, $\f_t$ converges exponentially fast towards $\f_{KE}$ on $\textrm{Amp}(K_X)$, as $t \rightarrow +\infty$.
	\end{theorem}

	\begin{proof}
		%	The problem is equivalent to solving and studying the parabolic complex Monge-Amp\`ere equation~\eqref{pcmae}with initial data $\f_0$, where $S_0=\omega_{0}+\dc\f_0$ and $\omega_t=e^{-t}\omega_{0}+(1-e^{-t})\theta$. 
		Recall that $\theta$ is a smooth representative of $c_1(K_X)$. Since $\omega_{0}$ is a K\"ahler form, there exists a small constant $c>0$ such that $\omega_0\geq c\theta$. Hence $\omega_{t}=e^{-t}\omega_{0}+(1-e^{-t})\theta\geq g(t)\theta$, where $g(t)=ce^{-t}+1-e^{-t}$ is a smooth (strictly) positive function with $g'(t)=e^{-t}(1-c)>0$ for $c>0$ small enough. The existence of the unique pluripotential solution follows from the results in Section~\ref{exist}. 
		
		It thus remains to study its long-term behavior. We first establish a lower bound for the solution $\f$ by constructing a subsolution to the Cauchy Problem. Set, for any $(t,x)\in (0,T)\times X$,
		\begin{equation*}
		u(t,x):=e^{-t}\f_0(x)+(1-e^{-t})\f_{KE}(x)+h(t),
		\end{equation*}
		for a $\mathcal{C}^1$ function $h$  to be chosen later 
		so that $u$ is a subsolution to the Cauchy Problem. We observe that for all $t>0$,
		$$
		\omega_{t}+\dc u_t=e^{-t}(\omega_{0}+\dc\f_0)+(1-e^{-t})(\theta+\dc \f_{KE})\geq 0
		$$
		in the weak sense of currents, so $u_t$ is $\omega_t$-psh and
		\begin{align*}
		(\omega_{t}+\dc u_t)^n&\geq (1-e^{-t})^n(\theta+\dc\f_{KE})^n=e^{n\log(1-e^{-t})}e^{\f_{KE}}fdV.
		\end{align*}
		On the other hand $\partial_t{u}_t+u_t=\f_{KE}+h'(t)+h(t)$ hence  $u$ is a subsolution if $$n\log(1-e^{-t})\geq h(t)+h'(t).$$ We thus choose $h$ to be the unique solution of the ODE: $$h(t)+h'(t)=n\log(1-e^{-t}), \,h(0)=0.$$ We compute $(e^th(t))'=e^t(h(t)+h'(t))=ne^t\log(1-e^{-t})$, hence
		\begin{align*}
		h(t)=ne^{-t}\left[\int e^t\log(1-e^{-t})dt\right]=ne^{-t}\left[(e^t-1)\log(e^t-1)-te^t+C\right].
		\end{align*} for some constant $C$.
		Since $h(0)=0$ the constant $C$ must be zero, hence
		$$h(t)=ne^{-t}\left[(e^t-1)\log(e^t-1)-te^t\right]=O(te^{-t})\quad \text{as}\,\, t\rightarrow\infty.$$
		It follows from the comparison principle that
		$u\leq \f$ hence
		\begin{align}\label{tt}
		\f_{KE}(x)+e^{-t}(\f_0(x)-\f_{KE}(x))+h(t)\leq \f(t,x).
		\end{align}

		For the upper bound, we argue as in the proof of~\cite[Theorem 4.4]{to2021convergence}. Since the cohomology class of $\theta$ is big, we can find a $\theta$-psh function $\chi_0$ with analytic singularities such that
		$$
		\theta+\dc\chi_0\geq \e\omega_{0}
		$$
		for some small constant $\e>0$.
		%(see e.g. \cite{Bou4}). 
		We can assume that $\e\leq 1$. Replacing $\chi_0$ by $\chi_0-\sup_X\chi_0$, we can always assume that 
		$\chi_0\leq 0$
		hence $\chi_0\leq V_{\theta}$. We then have
		\begin{align}\nonumber
		\omega_{t}+\dc \f_t&=e^{-t}(\omega_0-\e^{-1}\dc\chi _0)+(1-e^{-t})\theta+\dc(\f_t+\e^{-1}e^{-t} \chi_0)\\
		&\label{49}\leq [e^{-t}\e^{-1}+(1-e^{-t})] \theta+\dc(\f_t+\e^{-1}e^{-t} \chi_0).
		\end{align}
		Set $g(t)=1+(\e^{-1}-1)e^{-t}$ and $u(t,x)=\f_t(x)+e^{-t}(\e^{-1}\chi_0(x)-C)$ for a constant $C>0$ to be chosen later. It follows from \eqref{49} that
		\begin{align*}
		(g(t)\theta+\dc u_t)^n&\geq (\omega_t+\dc\f_t)^n
		\geq e^{\dot{\f_t}+\f_t}\mu
		=e^{\dot{u_t}+u_t}\mu,
		\end{align*}
		%in the sense of measures. 
		Let $\phi_0$ be a $\e^{-1}\theta$-psh function with minimal singularities. 
		We can  find a constant $C>0$ such that $\phi_0-\e^{-1}\chi_0\geq \f_0-C$. 
		Therefore $u$ is a subsolution of the following Cauchy problem
		\begin{equation}
		\begin{cases}\label{flow}
		(g(t)\theta+\dc\phi_t)^n=e^{\partial_{t}\phi_t+\phi_t}\mu \\
		\phi(0,\cdot)=\phi_0,
		\end{cases}
		\end{equation}
		where $\phi$ denotes the pluripotential solution to~\eqref{flow}. Thus the comparison principle (Proposition~\ref{comparision}) yields $u\leq \phi$ on $[0,\infty)\times {\rm Amp}(K_X)$, i.e.
		\begin{align*}
		\f(t,x)+e^{-t}(\e^{-1}\chi_0(x)-C)\leq \phi(t,x).
		\end{align*}
		On the other hand, the function $\phi_t$ converges to $\f_{KE}$ on ${\rm Amp}(K_X)$ as $t\rightarrow+\infty$ by Lemma~\ref{lemme} below. Combining this with \eqref{tt}, we infer that $\f_t$ converges to $\f_{KE}$ on ${\rm Amp}(K_X)$ as $t\rightarrow +\infty$.
		
	\end{proof}
	
	\begin{lemma}\label{lemme}
		The solution $\phi_t$ of~\eqref{flow} converges locally exponentially fast towards $\f_{KE}$ on ${\rm Amp}(K_X)$ as $t\rightarrow +\infty$.
	\end{lemma}
	
	\begin{proof}		
		Set 
		$$\tilde{\phi_t}:=g(t)^{-1}(\phi_t-a(t)),$$
		where $g(t)=1+(\e^{-1}-1)e^{-t}$, and $a$ is the unique solution of the ODE: $a(t)+a'(t)=n\log g(t)$, with $a(0)=0$. An easy computation shows that $a(t)=O(te^{-t})$. 
		Now the flow~\eqref{flow} becomes 
		\begin{align}
		\begin{cases}
		(\theta+\dc\tilde{\phi}_t)^n=e^{g(t)\partial_{t}\tilde{\phi}_t+\tilde{\phi}_t}\mu \\
		\tilde{\phi}(0,\cdot)=\e\phi_0.
		\end{cases}
		\end{align}
		We now normalize in time $\psi(t,\cdot)=\tilde{\phi}(s(t),\cdot)$, where $s(t)$ is the unique solution of the ODE $s'(t)=g(s(t))$ with $s(0)=0$. 
		Then the flow \eqref{flow} can be written as 
		\begin{align}\label{48}
		\begin{cases}
		(\theta+\dc\psi_t)^n=e^{\partial_{t}\psi_t+\psi_t}\mu \\
		\psi(0,\cdot)=\e\phi_0.
		\end{cases}
		\end{align}	
		We set for any $(t,x)\in (0,+\infty)\times X$,
		%\begin{align*}
		$u(t,x):=e^{-t}\psi_0(x)+(1-e^{-t})\f_{KE}(x)+h(t),$
		%\end{align*}  
		where $h$ is the unique solution to the ODE
		$h'(t)+h(t)=n\log(1-e^{-t})$, with $ h(0)=0. $
		As in the proof of Theorem \ref{thm_conv} we can check that $u$ is a subsolution to the Cauchy problem for the flow~\eqref{48}. 
		
		On the other hand, since $\f_{KE}$ is a $\theta$-psh function with minimal singularities we can choose a constant $C>0$ such that $\f_{KE}+C\geq\psi_0 $. Set for $(t,x)\in (0,+\infty)\times X$,
		\[v(t,x):=\f_{KE}(x)+Ce^{-t}. \]
		One can check that $v$ is supersolution to~\eqref{48}. Therefore, the comparison principle yields 
		\[ e^{-t}\psi_0(x)+(1-e^{-t})\f_{KE}(x) +O(te^{-t})\leq \psi(t,x)\leq \f_{KE}(x)+Ce^{-t},  \]
		which implies $\psi_t\to \f_{KE}$ on ${\rm Amp}(K_X)$ as $t\to +\infty$.
		So does the flow $\phi_t$ since $s(t)\to +\infty$,  $g(t)\rightarrow 1$ and $a(t)\rightarrow 0$ as $t\rightarrow +\infty$.
	\end{proof}
	
	\begin{remark}
		The uniqueness of the flow~\eqref{pcmae} follows directly from Theorem~\ref{thm_unique}.
	\end{remark}
	
	%	The following result is a simple consequence of Theorem \ref{thmlim}:
	%	\begin{theorem}		The solution $\f_t$ of the Monge-Amp\`ere equation \eqref{pcmae} converges pointwise towards $\f_0$ as $t\to 0^+$.	\end{theorem}
	
	\subsection{Extending the K\"ahler-Ricci flow through finite time singularities} \label{sect: singularity}
	In this subsection, we apply our results to prove the existence of the (pluripotential) K\"ahler-Ricci flow on manifolds through a finite time singularitie. In particular, the answer of Feldman-Ilmanen-Knopf's  question is affirmative also in this case.    
	%We have shown it in previous section when 
	
	Let $(X,\omega_0)$ be a compact K\"ahler manifold of dimension $n$ and consider the K\"ahler-Ricci flow with initial data $\omega_0$,
	\begin{equation}\label{eq: kr}
	\frac{\partial\theta}{\partial t}=-\textrm{Ric}(\theta),\quad \theta|_{t=0}=\omega_0.
	\end{equation}
	The maximal existence time $T$ of the flow is defined  by
	\begin{equation*}
	T:=\sup\{t>0: \{\omega_0\}+tc_1(K_X)\quad \text{is K\"ahler}\}. 
	\end{equation*} Suppose that $T<\infty$ ($K_X$ is not nef).  Then the limiting class $\{\theta_T\}:=\{\omega_0\}+Tc_1(K_X)$ is nef, but not K\"ahler. If we assume moreover that $\int_X\theta_T^n>0$, then the class $\{\theta_T\}$ is big by a fundamental theorem of Demailly and Paun~\cite[Theorem 2.12]{demailly2004numerical}. Since the set of big cohomology classes is open, there is a constant $\e>0$ so small that the class $\{\theta_{t}\}$ is big  for $t\in [0,T+\varepsilon)$. We can prove the existence of a pluripotential solution of the flow on $[0,T+\varepsilon)$. 
	\begin{theorem}
		Let $(X,\omega_0)$ be a compact K\"ahler manifold. Assume that the solution $\theta(t)$ of the K\"ahler-Ricci flow~\eqref{eq: kr} starting at $\omega_0$ exists  on the maximal time interval $[0,T)$ with $T<\infty$, and that the limiting class $\{\omega_0\}+Tc_1(K_X)$  is big. Then the pluripotential K\"ahler-Ricci flow starting with $\omega_0$ exists for $t\in [0,T+\e)$ for some small $\e>0$.
	\end{theorem}
	\begin{proof}
		Let $\eta$ be a smooth representative of the class $\{\theta_{T+\e}\}$, and set
		\begin{equation*}
		\chi=\frac{1}{T+\e}(\eta-\omega_0)\in c_1(K_X);\end{equation*}
		\begin{equation*}
		\omega_t=\omega_0+t\chi=\frac{1}{T+\e}((T+\e-t)\omega_0+t\eta)\in\{\omega_0\}+tc_1(K_X).   
		\end{equation*}
		Fix a volume form $dV$ on $X$ with $\dc\log V=\chi$. Then the K\"ahler-Ricci flow can be written as  the complex Monge-Amp\`ere flow 
		\begin{equation*}
		(\omega_t+\dc\f_t)^n=e^{\dot{\f}_t}dV,\quad \f(0)=0.
		\end{equation*}
		Since $\omega_0$ is a K\"ahler form, there exists a small constant $c\in(0,1)$  such that $\omega_0\geq c\eta$. Hence $\omega_t\geq g(t)\eta$ for $t\in[0,T']$, where $g(t)=(T+\e)^{-1}(c(T+\e)+t(1-c))$ is a positive increasing function. Theorem~\ref{existence} can be applied (with $F(t,x,r)\equiv 0$, $f=1$) and guarantees the existence of a pluripotential solution to the Monge-Amp\`ere flow %on $X_{T'}$, hence 
		on $X_{T+\varepsilon}$. %we obtain a pluripotential solution to the K\"ahler-Ricci flow on $[0,T+\varepsilon)$.
	\end{proof}
	Using the
	same argument as above the pluripotential K\"ahler-Ricci flow can be continued as long as the class $\{\omega_0\}+tc_1(K_X)$ is big. If $X$ has nonnegative Kodaira dimension, then $c_1(K_X)$ is pseudoeffective, and hence the class $\{\omega_0\}+tc_1(K_X)$ is big for any $t>0$. In particular, the flow is volume non-collapsing at a finite time singularity, as emphasized by Collins-Tosatti (see~\cite[Proposition 4.2]{collins2015kahler}).  We thus obtain a longtime pluripotential solution: 
	\begin{theorem}
		Let $(X,\omega_0)$ be a compact K\"ahler manifold with nonnegative Kodaira dimension. Then the pluripotential K\"ahler-Ricci flow starting with $\omega_0$ exists for  $t\in[0,\infty)$.
	\end{theorem}
\begin{proof}
		Fix $T<+\infty$. Let $\eta$ be a smooth representative of the class $\{\theta_{T}\}$, and set
		\begin{equation*}
		\chi=\frac{1}{T}(\eta-\omega_0)\in c_1(K_X);\end{equation*}
		\begin{equation*}
		\omega_t=\omega_0+t\chi=\frac{1}{T}((T-t)\omega_0+t\eta)\in\{\omega_0\}+tc_1(K_X).   
		\end{equation*}
		Fix a volume form $dV$ on $X$ with $\dc\log V=\chi$. Then the K\"ahler-Ricci flow can be written as  the complex Monge-Amp\`ere flow 
		\begin{equation*}
		(\omega_t+\dc\f_t)^n=e^{\dot{\f}_t}dV,\quad \f(0)=0.
		\end{equation*} 	Since $\omega_0$ is a K\"ahler form, there exists a small constant $c\in(0,1)$  such that $\omega_0\geq c\eta$. Hence $\omega_t\geq g(t)\eta$ for $t\in[0,T]$, where $g(t)=T^{-1}(cT+t(1-c))$ is a positive increasing function. Again, by Theorem~\ref{existence} there exists  a pluripotential solution $U=U_{\f_0, f,F,X_T}$ with $\f_0=0,f=1,F=0$.   
		
		We next claim that $U_t$ has minimal singularities for each $t\in (0,T)$. The proof of the claim is very similar to that of Theorem~\ref{thm: min sing}, but for completeness we provide the details below.  We first observe that $t\dot{\omega}_t\leq \omega_t$ for all $t>0$, yielding that $$\eta_t:=t^{-1}\omega_t$$ is decreasing in $t$. %In particular $\eta_t\geq \eta_{T}$ which is big.
		 By~\cite[Theorem 6.1]{boucksom2010monge}, for each $t\in(0,T)$ there exists a unique $\eta_t$-psh function $\phi_t$ with full Monge-Amp\`ere mass such that 	\begin{align*}(\eta_t+\dc\phi_t)^n=e^{ \phi_t}dV.	\end{align*} For $0<s\leq t$ we have 	\begin{align*}	(\eta_s+\dc\phi_t)^n\geq (\eta_t+\dc\phi_t)=e^{\phi_t} dV. 	\end{align*}
		It follows that $\phi_t$ is a subsolution  to $(\eta_s+\dc\phi_s)^n=e^{\phi_s}fdV $, so the comparison principle (see e.g.~\cite[Proposition 6.3]{boucksom2010monge}) ensures that $\phi_t\leq\phi_s$ for $s\leq t$. Therefore the function $t\mapsto\phi_t(x)$ is decreasing for all $x\in X$. %Since $t\mapsto\eta_t$ is decreasing in $t$, we may assume that $\eta_t\geq \theta'$ for some big $(1,1)$-form $\theta'$. 
	 As explained in Section~\ref{subsol}, $\dot{\phi}_t=\partial_{t}\phi_t$ is well-defined almost everywhere on $X_T$. Set $u(t,x):= t\phi_t(x)+n(t\log t-t)$. %We can pick $c_1$ so that  $\sup_X\phi_0=0$, hence in particular $\phi_t\leq 0$ for all $t\in [0,T)$.
		We infer, for almost every $(t,x)\in(0,T)\times X$, $$\dot{u}(t,x)={t}\dot{\phi}_t(x)+\phi_t(x)+n\log t \leq \phi_t+n\log t,$$ hence
		\begin{align*}
		(\omega_t+\dc u_t)^n&=e^{n\log t+\phi_t}dV\geq e^{\dot{u}_t}dV.
		\end{align*} Moreover, since $u(0,\cdot)=0=\f_0$, we have that  $u$ is a subsolution to the Cauchy problem, i.e. $u\in\mathcal{S}_{\f_0,f,F}(X_T)$ with $\f_0=0$, $F\equiv 0$. 
		  By~\cite[Theorem 6.1]{boucksom2010monge} we have $\phi_t\geq V_{\eta_t}-C(t)$ for some time-dependent constant $C(t)$. Since $V_{\omega_t}=tV_{\eta_t}$, we thus infer $u_t\geq V_{\omega_t}-C'(t)$, hence $U_t\geq V_{\omega_t}-C'(t)$ for all $t\in(0,T)$.  This completes the proof of the claim. 
		  
		 If $T'>T$, then by the above arguments there exist a pluripotential solution $U'=U_{\f_0,f,F,X_T'}$ of the flow. Both $U$ and $U'$ satisfy the assumptions in Theorem \ref{thm_unique}, hence $U=U'$ on $X_T$. We can thus glue all these solutions to get a longtime solution of the flow, finishing the proof. 
		
		%there is a constant $A>0$ so large that $\chi=\dot{\omega}_t\leq A\omega_t$ for all $t\in [0,T]$.  
\end{proof}
	
	\subsection{Stable varieties}\label{sect: semi-log canonical} 
	
	\subsubsection*{Log canonical pairs.}
	A pair $(X,D)$ is by definition a complex normal compact projective variety carrying a Weil $\mathbb{Q}$-Cartier $D$ (not necessary effective). 
	We will say that the pair $(X,D)$ is a {\em log canonical (lc)} pair if $K_X+D$ is $\mathbb{Q}$-Cartier, and if for some (or equivalently any) log resolution $\pi:X'\rightarrow X$, we have  
	\begin{align*}
	K_{X'}=\pi^*(K_X+D)+\sum a_i E_i
	\end{align*}
	where $E_i$ are either exceptional divisors or components of the strict transform of $D$, and the coefficients $a_i$ satisfy the inequality $a_i\geq -1$.
	
	When $D\equiv 0$, we say that $X$ has {\em log canonical singularities}.
	\subsubsection*{Semi-log canonical singularities.}
	We give here a short overview of the notion of semi-log canonical singularities and stable varieties. 
	We refer to the survey~\cite[\S 5, 6]{kovac2013singularities} and the references therein for more details.
	
	In the sequel, $X$ will be a reduced and equidimensional scheme of finite type over $\mathbb{C}$ unless stated otherwise, and
	we set $n:=\dim_\mathbb{C}X$.  
	In order to study the normalized K\"ahler-Ricci flow, one needs  a canonical sheaf (or a canonical divisor). 
	Let us stress that the dualizing sheaf, even if it exists, is not necessarily a line bundle (or a divisor). 
	
	We say that the scheme (variety) $X$ is \emph{Cohen-Macaulay} if for every $x\in X$ the depth of $\mathcal{O}_{X,x}$, denoted by $\textrm{depth}(\mathcal{O}_{X,x})$,  is equal to its Krull dimension. If $X$ is Cohen-Macaulay, then $X$ admits a dualizing sheaf $\omega_X$. 
	
	We say that $X$ is \emph{Gorenstein} if $X$ is Cohen-Macaulay ($X$ admits a dualizing sheaf $\omega_X$) and $\omega_X$ is a line bundle. A scheme (variety) $X$ is called $G_1$ if it is Gorenstein in codimension 1, which means that there exists an open subset $U\subset X$ such that $\text{codim}_X(X\backslash U)\geq 2$ and $U$ is Gorenstein.
	
	We say that $X$ satisfies the \emph{$S_2$ condition} of Serre if for all $x\in X$, we have $\text{depth}(\mathcal{O}_{X,x})\geq \min\{\dim \mathcal{O}_{X,x},2\}$. This condition is equivalent to saying that for each closed subset $\imath:Z\hookrightarrow X$ 
	of codimension at least two, the natural map $\mathcal{O}_X\rightarrow\imath_*\mathcal{O}_{X\backslash Z}$ is an isomorphism.
	
	We now want to have an interpretation of $\omega_X$ in terms of Weil divisor. If $X$ satisfies the conditions $G_1$  and $S_2$, and $U$ is a Gorenstein open subset whose complement has codimension at least $2$, we may define the "canonical=dualizing" sheaf $\omega_U$ as the determinant of the cotangent bundle, i.e., the sheaf of top differential forms, $\omega_U=\det\Omega_U$. One can then define the canonical sheaf $\omega_X$ by $\omega_X=\jmath_*\omega_U$ where $\jmath:U\hookrightarrow X$ is the open embedding.
	
	As $U$ is non-singular, $\omega_U$ is a line bundle, hence corresponds to  a Cartier divisor. 
	Let $K_U:=\sum a_iK_i$ be a Weil divisor associated
	to this Cartier divisor such that for all $i$, $K_i$ does not contain any component of $X_{\rm sing}$ of codimension $1$. 
	Let $\bar{K}_i$ denote the closure of $K_i$ and  
	$$K_X:=\sum a_i\bar{K}_i.$$ 
	Since ${\rm codim}_X(U)\geq 2$, this is  the unique Weil divisor  for which $K_X|_U=K_U$. We see that the divisorial sheaf
	$$
	\mathcal{O}_X(K_X):=\{f\in K(X): K_X+{\rm div}(f)\geq 0 \}
	$$
	is reflexive, and coincides with $\omega_U=\omega_X|_U$, hence the $S_2$ condition implies that
	$$\omega_X\simeq\mathcal{O}_X(K_X).$$
	
	\begin{remark}
		The condition $G_1$ guarantees the existence of the canonical sheaf $\omega_X$, and the condition $S_2$ ensures its uniqueness.
		When $X$ is projective, we know that it admits a dualizing sheaf, as it is reflexive, it coincides with $\omega_X$ by the $S_2$ condition.
	\end{remark}
	
	We let $\omega_X^{[m]}$ denote the $m$-th reflexive power of the canonical sheaf $\omega_X$ (defined by $\omega_X^{[m]}:=(\omega_X^{\otimes m})^{**}$). The same arguments above yield $\omega_X^{[m]}\simeq \mathcal{O}_X(mK_X)$. Thus the Weil divisor $K_X$ is $\mathbb{Q}$-Cartier if and only if $\omega_X$ is a $\mathbb{Q}$-line bundle, i.e.,  $\omega_X^{[m]}$ is a line bundle for some $m>0$.
	From now on we work with the canonical divisor $K_X$ instead of its associated canonical sheaf $\omega_X$.
	
	We say that a closed point $x\in X$ is \emph{double crossing} if it is locally  analytically isomorphic to the singularity
	$$\{0\in (z_0z_1=0)\subset \mathbb{C}^{n+1} \}.$$
	A scheme $X$ is called \emph{demi-normal} if it satisfies the $S_2$ condition and has only double crossing singularities in codimension 1. We now give the definition of semi-log canonical models:
	\begin{definition}\label{def_slc}
		We say that $X$ has \emph{semi-log canonical (slc)} singularities if $K_X$ is $\mathbb{Q}$-Cartier and there exist two Zariski open sets $U,V$ such that
		\begin{itemize}
			\item $X=U\cup V$,
			\item $U$ is a normal variety with log canonical singularities,
			\item $V$ has only double crossing points. 
		\end{itemize}
	\end{definition}
	
	We mention that semi-log canonical models may not be normal varieties. Let $\mu:X^n\rightarrow X$ be a normalization of $X$. 
	We emphasize again   that $X$ is not irreducible in general, so its normalization is defined to be the disjoint 
	union of the normalization of its irreducible components. The \emph{conductor ideal sheaf}
	$$\mathscr{I}_{C_X}:=Ann_{\mathcal{O}_X}(\mu_*\mathcal{O}_{X^n}/\mathcal{O}_X)$$ is defined to be the largest ideal sheaf on $X$ that is also an ideal sheaf on $X^n$. If we consider the affine case where $A^n$ is the integral closure of some integral ring $A$, then one can see that the annihilator $Ann_A(A^n/A):=\{a\in A: aA^n\subset A \}$ is the largest ideal in $A$ that is also an ideal  in $B$.
	
	For the case of schemes (varieties), we let $\mathcal{I}_{C_{X^n}}$ denote the corresponding conductor ideal sheaf on $X^n$,
	and we define the conductor subscheme as  $C_X:=Spec_X(\mathcal{O}_X/\mathscr{I}_{C_X})$ on $X$ and $C_{X^n}:=Spec_{X^n}(\mathcal{O}_{X^n}/\mathscr{I}_{C_{X^n}})$ on $X^n$. If $X$ is seminormal (i.e. every finite morphism $X'\rightarrow X$, with $X'$ is reduced, that is a bijection on points is an isomorphism) and $S_2$, then one can show that these subschemes have pure codimension 1 hence they define Weil divisors which are moreover  reduced (cf.~\cite[4.5]{kovac2010canonical}). 
	
	If $X$ is demi-normal and $K_X$ is $\mathbb{Q}$-Cartier, then we have the following relation
	\begin{align}\label{cond}
	\mu^*K_X=K_{X^n}+C_{C^n}.
	\end{align}
	Under the previous seminormality and $S_2$ assumptions, the $G_1$ condition is equivalent to the demi-normality. 
	In other words, we may alternatively define slc models as follows:
	\begin{defprop}
		A scheme $X$ has semi-log canonical singularities if and only if
		\begin{itemize}
			\item $X$ is $G_1$ and $S_2$,
			\item $K_X$ is $\mathbb{Q}$-Cartier (of index $m$),
			\item The pair $(X^n,C_{X^n})$ is log-canonical.
		\end{itemize}
	\end{defprop}  
	Note that there are many schemes satisfying the $S_2$ condition and the seminormality but not demi-normality. 
	For instance, a reduced scheme consisting of the three axes in $\mathbb{A}^3$ does not have double crossings in codimension 1, but is both $S_2$ and
	seminormal. 
	
	\smallskip
	
	We can finally give the definition of stable variety:
	
	\begin{definition}
		A projective variety $X$ is called \emph{stable} if 
		\begin{itemize}
			\item $X$ has semi-log canonical singularities, 
			\item  $K_X$ is an ample $\mathbb{Q}$-Cartier divisor. 
		\end{itemize}
	\end{definition}
	
	From Definition~\ref{def_slc}, we can see that $X$ is a stable variety if $K_X$ is ample and $X=U\cup V$, where $U,V$ are Zariski open sets, 
	$U$ is  a normal variety with log canonical singularities, and $V$ has only double crossing singularities.

	\subsection{Convergence of NKRF on stable varieties}
	
	Let $X$ be a complex projective variety with semi-log canonical singularities such that $K_X$ is ample (stable variety). 	 
	We now consider the normalized K\"ahler-Ricci flow  starting at any K\"ahler metric $\omega_0$ on $X$,
	this is the evolution following equation:
	\begin{align}\label{nkrf2}
	\dfrac{\partial\theta_{t}}{\partial t}=-{\rm Ric}(\theta_{t})-\theta_{t},\quad \theta|_{t=0}=\omega_{0}.
	\end{align}
	
	After passing to a suitable resolution of singularities, we may as well assume that $X$ is smooth if we 
	study  the setting of {\em log pairs} $(X,D)$, where $D=\sum_{i=1}^N a_iD_i$ is the $\mathbb{Q}$-divisor on $X$ with simple normal crossing (snc), where the role of the canonical line bundle is played by the {\em log canonical line bundle} $K_X+D$ (which occurs as the pull-pack to the resolution of the original canonical line bundle). In this setting the original variety has semi-log canonical singularities precisely when the log pair $(X,D)$ is {\em log canonical (lc)} in the usual sense of Minimal Model Program (MMP), i.e. the coefficients of $D$ are at most equal to one (but negative coefficients allowed). Let us mention that even if the original canonical line bundle is ample, the corresponding log canonical line bundle is merely {\em semi-ample} and big on the resolution, since it is trivial along the exceptional divisors of the corresponding resolution. The initial data $\omega_0$ may now assume to be a smooth semi-positive (1,1) form with big cohomology class.
	
	Let $X$ be a compact K\"ahler manifold and $(X,D)$ be a log canonical pair such that  $K_X+D$ is semi-ample and big (i.e., $(K_X+D)^n>0$). We fix $\theta$ a smooth representative of the class $c_1(K_X+D)$. 
	It has been shown in~\cite[Theorem C]{berman2014kahler} that there exists a unique closed positive current
	$\omega_{KE}=\theta+\dc\psi_{KE}$ in $c_1(K_X+D)$ which is smooth on a Zariski open set $U$ of $X$ and satisfies
	\begin{align*}
	{\rm Ric}(\omega_{KE})=-\omega_{KE}+[D]
	\end{align*} in the  sense of currents on $X$. The current $\omega_{KE}$ is called the \textit{singular K\"ahler-Einstein metric}.
	\smallskip
	
	Our aim is to prove the existence and the convergence of the pluripotential solutions of the normalized K\"ahler-Ricci flow like the previous one, this is the content of the following theorem:
	\begin{theorem}\label{thm_conv2} Let $S_0$ be a positive closed current with bounded potentials.
		Then the normalized K\"ahler-Ricci flow starting with $S_0$ admits a unique pluripotential solution defined on $[0,+\infty)\times X$. 
		Furthermore, the pluripotential normalized K\"ahler-Ricci flow converges towards $\omega_{KE}$ on $\textrm{Amp}(K_X+D)$, as $t\rightarrow+\infty$.
	\end{theorem}
	
	Observe that Theorem~\ref{thm_conv2} implies Theorem~\ref{D} in the introduction:
	indeed, if $Y$ is a projective variety with semi-log canonical singularities such that $K_Y$ is ample (stable variety) and $\pi:(X,D)\rightarrow Y$ is a log resolution of the normalization (endowed with its conductor), then the exceptional locus of $\pi$ is contained in the complement of the ample locus of $K_X+D$.

	\begin{remark}
		A similar result has been obtained in \cite[Theorem 1.3]{chau2019kahler} with a very different approach.
		These authors generalize the a priori estimates of Song-Tian \cite{song2017kahler}
		to the case of	$\mathbb{Q}$-factorial projectives varieties with log canonical singularities. 
		They also show that if $X$  is stable, then the normalized K\"ahler-Ricci flow~\eqref{nkrf2} has a unique maximal weak solution on $[0,+\infty)$
		which is smooth in $(0,+\infty) \times X_{\rm reg}$ and  converges to the singular K\"ahler-Einstein metric $\omega_{KE}$ both 
		in the sense of currents and in the $\mathcal{C}^\infty_{\rm loc}(X_{\rm reg})$-topology as $t$ tends to infinity.
		
		Our approach  allows one to treat more general equations, avoiding any projectivity assumption on the variety
		nor any integrality on the initial cohomology class, 
		and applies to  big classes for which no smooth deformation is available.
		%It also provides more general uniqueness results, even in the cases covered by \cite{CGLS19}.
	\end{remark}
	
	\begin{proof}[Proof of Theorem~\ref{thm_conv2}]
		By definition, $D=\sum_{i=1}^Na_iD_i$ is a simple normal crossings $\mathbb{R}$-divisor with $a_i\in(-\infty,1]$ and defining section $s_i$.
		The normalized K\"ahler-Ricci flow~\eqref{nkrf2} can be written as the following complex Monge-Amp\`ere flow
		\begin{align}\label{fbg}
		(\omega_{t}+\dc\f_t)^n=e^{\dot{\f}_t+\f_t}d\mu, 
		\end{align}
		where $\omega_t:=e^{-t}\omega_0+(1-e^{-t})\theta$, 
		and $d\mu$ is a measure on $X$ which is of the form 
		\begin{align*}
		d\mu= \dfrac{dV_X}{\prod_{i=1}^N|s_i|^{2a_i}}=fdV_X
		\end{align*}
		where  $s_i$  are non-zero sections of $\mathcal{O}_X(D_i)$, $|\cdot|_i$ are smooth hermitian metrics on $\mathcal{O}_X(D_i)$, and $dV_X$ is a smooth volume form on $X$. We let $D_{\rm lc}:=\cup_{a_k=1}D_k$ denote the "non-klt" locus.
		
		\medskip 
		
		\label{step1}\noindent{\bf Step 1: constructing  a subsolution.} 
		We let  $\Omega$ denote the ample locus of the class $\{\theta\}$. Since the latter is big, there exists a $\theta$-psh function $\chi_0$ such that 
		\[\tilde{\theta}:=\theta+\dc\chi_0\geq \delta\omega_X \;\text{on}\; \Omega\; \text{for some}\, \delta>0\; \text{and}\; \chi_0\to -\infty \;\text{near}\; \partial\Omega. \]
		
		Up to multiplying by a positive constant
		% or smooth positive function for those $i$, 
		we can assume that $|s_i|^2\leq 1/e$ so that $-\log(|s_i|^2)\geq 1$ out of $D_i$. Note also that $dd^c(-\log(|s_i|^2))$ extends as a smooth real $(1,1)$-form on $X$ whose cohomology class is $2\pi c_1(D)$. We compute
		\begin{align*}
		-dd^c(\log(\lambda-\log(|s_i|^2)))=-\dfrac{dd^c(\lambda-\log(|s_i|^2))}{\lambda-\log(|s_i|^2)}+\dfrac{ds_i\wedge d^c s_i}{|s_i|^2(\lambda-\log(|s_i|^2))^2}.
		\end{align*}
		The second term is a semipositive $(1,1)$-form. Since $-\log(|s_i|^2)$ goes to $\infty$ near $D_j$, we infer that $\tilde{\theta}-\dc \log(\lambda-\log|s_i|^2)$ is positive on $\Omega\backslash D_j$ when $\lambda$ is big enough. Replacing $\tilde{\theta}$ by $\frac{1}{N}\tilde{\theta}$ and increasing $\lambda$ if necessary one has $\frac{1}{N} \tilde{\theta}-\dc \log(\lambda-\log|s_i|^2)>0$, hence $\sum(\frac{1}{N}\tilde{\theta}-\dc\log(\lambda-|s_i|^2))$ defines a K\"ahler form on $\Omega\backslash D$. Then for suitable positive constants $\lambda,A$ the function 
		\begin{align}\label{func_v}
		v:=-2\sum_{i=1}^N\log(\lambda-\log|s_i|^2)+\chi_0-A
		\end{align}
		is a subsolution of the complex Monge-Amp\`ere equation:
		\begin{align}\label{sol_bg}
		(\theta+\dc\psi_{KE})^n=e^{\psi_{KE}}\frac{dV_X}{\prod_i|s_i|^{2a_i}}.
		\end{align}
		By the arguments above, we  get the lower bound (see also~\cite[5.5.2]{berman2014kahler}): 
		\begin{align*}
		\psi_{KE}\geq\chi_0 -\sum_{a_k=1}\log(-\log|s_k|^2)-A
		\end{align*}
		for some uniform constant $A>0$. Here the hermitian metrics $|\cdot|_k$ are chosen conveniently.	 
		
		Since $\theta$ is semi-positive we see that for all $t$, $\omega_t\geq c\theta$ for some $c>0$ small enough. 
		For simplicity, we may assume that $c=1$. We can check that 
		\begin{equation}\label{eq: subsol}
		u(t,x):=\psi_{KE}(x)-C_0e^{-t}
		\end{equation}
		is a subsolution to~\eqref{fbg}, for a constant $C_0>0$ so large that $u(0,\cdot)\leq \f_0$.
		
		\vskip 0.3cm
		\label{step2}\noindent{\bf Step 2: the approximating flows.}
		We now establish the existence of the flow~\eqref{fbg} by an approximation argument using ideas from~\cite[Theorem 4.5]{di2015generalized}.
		Fix $T<+\infty$.  The difficulty is that the density $f=\Pi_i |s_i|^{-2a_i}$ is not in $L^p$, $p>1$,
		(not even in $L^1$) since some of the coefficients $a_j$ might be equal to $1$.
		For each $j\in\mathbb{N}$, Theorems~\ref{A} and \ref{B} provide a unique $\f_{t,j}\in\mathcal{P}(X_T,\omega)$ such that
		\begin{equation}\label{eq: approx. NKRF}
		(\omega_t +dd^c \varphi_{t,j})^n = e^{\dot{\varphi}_{t,j} + \varphi_{t,j}} \min(f,j) dV_X, \; \varphi_{0,j} = \varphi_0. 
		\end{equation} 
		Since $\omega_t$ is the pull-back of a smooth family of K\"ahler forms, we have 
		\[
		-A\omega_t \leq \dot{\omega}_t \leq A\omega_t,
		\]
		for a uniform constant $A>0$.  We can proceed as in the proof of Theorem~\ref{thmlips} and Theorem~\ref{thmscc} to establish the following uniform bounds: for each $T\in (0,+\infty)$ and any compact $K\subset \Omega\backslash D_{\rm lc}$, there is a constant $C(T,K)$ such that 
		
		\[
		t |\partial_t \varphi_{t,j}(x)| \leq C(T,K), \; \text{and} \;  t^2 \partial^2_t \varphi_{t,j}(x) \leq C(T,K),\quad \forall\;(t,x)\in (0,T)\times K. 
		\]
		Indeed, on $(0,T)$ the forms $\omega_t$ satisfy $\omega_t \geq g(t) \theta$, where $g(t) =c_0>0$ is a constant.  The function $F$ in our case is defined by  $r\mapsto F(t,x,r)\equiv r$ which satisfies the assumptions in the introduction. More precisely, we have the following:

		\smallskip
		
		\begin{proposition}\label{thm_lip2}
			Let $J=[a,b]$ be a compact interval of $(0,T)$. There exist uniform constants $C_0,C_1,C_2>0$ such that  for all $j\in\mathbb{N}$, $t\in J$,
			\begin{itemize}
				\item[(1)] $C_0\geq \f_{t,j}(x)\geq \psi_{KE}(x)-C_0e^{-t}$,
				\item[(2)] $|\partial_t\f_{t,j}|\leq C_1 +\sum_{a_k=1}\log(-\log|s_k|^2) -\chi_0$,
				\item[(3)] $\partial_t^2\f_{t,j}\leq C_2 +\sum_{a_k=1}\log(-\log|s_k|^2) -\chi_0.$ 
			\end{itemize}
		\end{proposition}

		\begin{proof} We first prove (1). For the lower bound, we can check that the function $u$ in~\eqref{eq: subsol} is also a subsolution to~\eqref{eq: approx. NKRF}. For the upper bound, we pick $C>0$ so big that $\omega_t^n\leq e^C fdV$ for all $t\in [0,T]$. The domination principle (see e.g.~\cite[Corollary 2.5]{boucksom2010monge}) yields $\f_{t,j}\leq C $ holds everywhere for all $t$, $j$.
			
			We next prove $(2)$. Fix $\e_0>0$ such that $(1+\varepsilon_0)b<T$. For all $t\in J$ and $s\in(1-\e_0,1+\e_0)$ there exists a constant $A_1>0$ such that
			\begin{align}\label{ineq}
			\omega_t\geq (1-A_1|s-1|)\omega_{ts}.
			\end{align}
			For $s$ small enough we set \begin{align}\label{alpha}
			\lambda_s:=\frac{|1-s|}{s},\quad\alpha_s:=s(1-\lambda_s)(1-A_1|s-1|)\in (0,1),
			\end{align} hence $\gamma_s:=\lambda_s/(1-\alpha_s)\geq \varepsilon_1>0$.  
			Shrinking $\e_1$ we may assume that $\gamma_s\omega_t\geq \e_1\theta$. 
			Let $v_1$ be a solution to the following equation
			\begin{align}\label{eq1_bh}
			(\e_1\theta+\dc v_1)^n=e^{v_1}f dV_X.
			\end{align}
			The same argument in the Step \hyperref[step1]{1} yields
			\begin{align*}
			v_1\geq \e_1\chi_0-\varepsilon_1\sum_{a_k=1}\log(-\log|s_k|^2)-A,
			\end{align*} for some uniform constant $A>0$.
			For any $(t,x)\in J\times X$ we set
			\begin{align*}
			u^s(t,x):=\frac{\alpha_s}{s}\f_{j}(ts,x)+(1-\alpha_s)v_1(x)-C|s-1|e^{-t},
			\end{align*}
			for $C>0$ to be chosen later. We have
			\begin{align*}
			(\omega_t+\dc u^s(t,\cdot))^n&=\left[(1-\lambda_s)\omega_t+\frac{\alpha_s}{s}\dc \f_{ts}+(1-\lambda_s)\omega_t+(1-\alpha_s)\dc v_1\right]^n\\
			&\geq [\alpha_s(\omega_{ts}+\dc \f_{ts,j})+(1-\alpha_s)(\gamma_s\omega_t+\dc v_1)]^n\\
			&\geq e^{\alpha_s(\partial_t\f_j(ts,\cdot)+\f_j(ts,\cdot))+(1-\alpha_s)v_1}\min (f,j)dV\\
			&=e^{\partial_tu^s(t,\cdot)+u^s(t,\cdot)}\min (f,j)dV
			\end{align*}
			where we use \eqref{ineq} in the second line and 
			Lemma~\ref{logconcave} in the third one. %We now choose a constant $C>0$ such that, for any $x\in X$,
			%\begin{align*}
			%   u^s(0,x)&\leq
			% \f_0(x)+(2(A_1+1)(\sup_X\f_0+\sup_Xv_1))|s|-C|s|\\
			%  &\leq \f_0(x).
			%	\end{align*}
			Therefore $u^s$ is a subsolution to \eqref{eq: approx. NKRF}.  Since $\f_0$ is bounded we can choose $C>$ so large that $u^s(0,\cdot)\leq \f_0$ on $X$. Hence  the comparison principle (Proposition~\ref{comparision}) ensures that for any $j$, $u^s\leq \f_j$ in $J\times X$, i.e.,
			\begin{align*}
			\frac{\alpha_s}{s}\f_j(ts,x)+(1-\alpha_s)v_1-C|s-1|e^{-t}\leq \f_j(t,x),\quad\forall\, (t,x)\in J\times X.
			\end{align*}
			Letting $s\rightarrow 1$ we infer for all $(t,x)\in J\times X$,
			\begin{align*}
			|\partial_t\f_j(t,x)|\leq C_1-C_1v_1(x),
			\end{align*}
			for a uniform constant $C_1>0$.	
			
			To prove $(3)$ we argue as above. Set for any $(t,x)\in J\times X$, 
			\begin{align*}
			v^s(t,x):={\alpha_s}\frac{s^{-1}\f_{j}(ts,x)+s\f_j(ts^{-1},x)}{2}+(1-\alpha_s)v_1(x)-C|s-1|e^{-t},
			\end{align*}
			for a constant $C>0$ so large that $v^s(0,\cdot)\leq \f_0$.
			We can check as above that $v^s$ is a subsolution to~\eqref{eq: approx. NKRF}. By the same arguments we can obtain the estimate $(3)$.
		\end{proof}

		We now finish the proof of Step 2.
		For $t\in (0,T)$ fixed,  $\f_{t,j}$ is decreasing as $j\to\infty$ by the comparison principle (Proposition~\ref{comparision}). 
		It follows from Proposition~\ref{thm_lip2} that
		$$\f_{t,j}(x)\geq \psi_{KE}(x)-C_0,\, \forall\; (t,x)\in [0,T)\times X,$$ 
		for a large constant $C_0>0$. It has been shown in~\cite[Theorem 4.2]{boucksom2010monge} that 
		$\psi_{KE}\in\mathcal{E}(X,\omega_t)$ for each $t$ since $0\leq \theta\leq \omega_t$, 
		hence $\f_{t,j}\in\mathcal{E}(X,\omega_t)$.
		We want to prove that $\lim_j\f_{t,j}=\f_t$ is a solution to the flow~\eqref{fbg}.

		%	Fix $M>0$ and set $G= \{f<M\}$. 

		Fix a compact sub-interval  $J\Subset(0,T)$, a compact subset $K\subset(\Omega\setminus D_{\rm lc})$. Proposition~\ref{thm_lip2} implies that there exists a constant $C=C_J>0$ such that the function $t\rightarrow \f_j(t,x)-Ct^2$ is concave in $J$, for all $x\in K$. 
		Moreover, the function $x\mapsto\f_j(t,x)$ is $\omega_t$-psh and uniformly bounded on $K$ for all $j$. We obtain the same properties for the limiting function $\f(t,x)$ by letting $j\rightarrow+\infty$. It follows from Proposition~\ref{t_der} that $\dot{\f}_j$, $\dot{\f}$ are well-defined and $\lim_j\dot{\f}_j(t,\cdot)=\dot{\f}(t,\cdot)$.
		Consider
		\[G:=\{x\in X: f(x)>M \}\cup \{x\in X: -v_1(x)>M \}, \] where $v_1$ is a solution to~\eqref{eq1_bh}. Since $-v_1$ is locally bounded outside a divisor, we can choose $M>0$ so large that  
		$G$ has small Monge-Amp\`ere capacity
		${\rm Cap}_\Theta(G)<\varepsilon$ for some K\"ahler form $\Theta$ and for any $\varepsilon>0$. Hence for all $t\in J$, $j\in\mathbb{N}$, we have that $\dot{\f}_{t,j}$ is uniformly bounded from above on $X\backslash G$. 
		Therefore Lebesgue dominated convergence theorem ensures that 
		\begin{align*}
		\lim_{j\rightarrow+\infty}\int_J\int_{X\backslash G}e^{\dot{\f}_{t,j}+\f_{t,j}}\min(f,j)dVdt=\int_J\int_{X\backslash G} e^{\dot{\f}_t+\f_{t}}fdVdt.
		\end{align*}
		Using the notations from~\cite[Section 2]{di2015generalized}, 
		% Since $\f_{t,j}\in\mathcal{E}(X,\Theta)$ for K\"ahler form $\Theta$,
		it follows from \cite[Theorem 2.9]{di2015generalized} that, for all $t\in (0,T)$, $j\in\mathbb{N}$,
		\begin{align*}
		\int_{G}(\omega_t+\dc\f_{t,j})^n\leq  {\rm Cap}_{\psi_{KE}-C_0,0}(G)\leq h(\varepsilon),
		\end{align*}
		for some continuous function $h: [0,+\infty) \rightarrow [0,+\infty)$ with $h(0)=0$.
		Hence
		\begin{align*}
		\int_J\int_{ X}e^{\dot{\f}_t+\f_t}fdVdt &\geq \int_J\int_{X\backslash G}e^{\dot{\f}_t+\f_t}fdVdt=\lim_{j\rightarrow +\infty}\int_J\int_{X\backslash G}(\omega_t+\dc\f_{t,j})^n\\
		&=\lim_{j\to\infty}\int_0^T\int_{X}(\omega_t+\dc\f_{t,j})^n-\lim_{j\rightarrow +\infty}\int_J\int_{G}(\omega_t+\dc\f_{t,j})^n\\
		&\geq \int_J\int_{X}\omega_t^n-Th(\varepsilon).
		\end{align*}
		%	where in the last line we used the fact that $\f_{t,j}\in\mathcal{E}(X,\omega_t)$ for all $t,j$.
		Letting $\e\rightarrow 0$ we obtain 
		$\int_J\int_{ X}e^{\dot{\f}_t+\f_t}fdVdt\geq \int_J\int_{X}\omega_t^n.$
		
		On the other hand, since $(\omega_t+\dc\f_{t,j})^n$ converges  to  $(\omega_t+\dc\f_{t})^n$, 
		Fatou's lemma yields
		\begin{align*}
		dt\wedge (\omega_t+\dc\f_t)^n\geq e^{\dot{\f}_t+\f_t}fdV dt
		\end{align*}
		in the sense of measures in $(0,T)\times X$, 
		whence equality. This implies that $\f$ is a solution to~\eqref{fbg}.
		\begin{proposition}
			For each $t$, the solution $\f_t$ of \eqref{fbg} is continuous on $\Omega\setminus D_{\rm lc}$.
		\end{proposition}
		
		\begin{proof}
			It follows from Proposition~\ref{thm_lip2} that $$e^{\dot{\f_t}+\f_t}f\leq \exp\left(C +\sum_{a_k=1}\log(-\log|s_k|^2) -\chi_0-\sum_{i}\log(|s_i|^2)\right).$$ The proof thus follows from~\cite[Theorem 3.2]{dang2021continuity}.
		\end{proof}
		
		\medskip
		
		\noindent{\bf Step 3: convergence at time zero.}
		Using similar arguments as in the proof of Theorem~\ref{thmlim}, we are going to check that
		the solution $\f_t$ of the equation~\eqref{fbg} converges pointwise towards $\f_0$ as $t\to 0^+$.
		
		\smallskip

		Arguing as at the beginning of the proof of Theorem~\ref{thm_conv}, we can check that \[\f(t,x)\geq u(t,x):= e^{-t}\f_0(x)+(1-e^{-t})\psi_{KE}(x)+h(t), \quad\forall\, (t,x)\in(0+\infty)\times X, \] where $\psi_{KE}$ is the solution of~\eqref{sol_bg} and $h(t)=ne^{-t}\left[(e^t-1)\log(e^t-1)-te^t\right].$
		It thus remains to show that for all $x \in X$,
		$\lim_{t\to 0}\f_t(x)\leq \f_0(x)$.
		
		Fix $T<+\infty$ and consider
		$$
		G:=\{x\in X: u(T,x)>-M \},
		$$ where $M>0$ is a constant such that $\mu(G)>\mu(X)/2$
		(recall that $\psi_{K_E}$ is smooth outside a divisor). 
		Observe that $\f(t,x)\geq u(t,x)\geq u(T,x)>-M$ for all $x\in G$, $t\in (0,T)$. 
		Following the proof of Theorem~\ref{thmlim}, we obtain as in~\eqref{claim28} that
		\begin{align}
		\int_G\f_t d\mu\leq \int_G \f_0 d\mu+Ct,
		\end{align}
		for a constant $C>0$ depending on $G$. 
		
		Let now $u_0\in \PSH(X,\omega_{0})$ be any cluster point of $\f_t$ as $t\rightarrow 0$. We can assume that $\f_t$ converges to $u_0$ in $L^q(X,dV)$ for any $q>1$. On the other hand, $d\mu=\prod_{i}|s_i|^{-2a_i}dV_X$ has density $f=\prod_{i}|s_i|^{-2a_i}\in L^{p}_{loc}(X\backslash D)$ for any $p>1$. Hence, $\f_t f$ converges to  $u_0f$ in $L^1(K)$ for any compact subset $K$ of $X\backslash D$. Thus, the claim above ensures that
		\begin{align*}
		\int_G u_0fdV\leq \int_G \f_0fdV.
		\end{align*} 
		We infer that $u_0\leq \f_0$ almost everywhere on $G$ with respect to $fdV$, hence everywhere on $G$.
		% by the assumption on $f$.
		Letting $M\to +\infty$, we conclude that $\limsup_{t\rightarrow 0}\f_t= \f_0$ on $X\backslash D$, hence on the whole $X$.
		
		\medskip
		
		\label{step4}\noindent{\bf Step 4: uniqueness of the flow.} By the previous steps, we have shown that there exists a solution $\f$ to~\eqref{fbg} with initial data $\f_0$. This function satisfies the following properties:
		\begin{itemize}
			\item $\f$ is locally uniformly semi-concave in $t$,
			\item  $(t,x)\to \f(t,x)$ is continuous on $(0,+\infty)\times U$, where $U:=\Omega\setminus D_{\rm lc}$,
			\item $\f_t\to \f_0$ pointwise as $t\to 0^+$.  
		\end{itemize}
		We are going to show that such a solution is unique. 
		Let $\Phi$ be a solution to~\eqref{fbg} with the same properties as above. 
		We shall prove that $\f\leq \Phi$ on $[0,+\infty)\times X$, whence equality. 
		The proof follows step by step from the uniqueness result obtained  in Section~\ref{sect: compa_princ}.
		
		\smallskip
		
		\noindent{\it Step 4.1.}\label{step41} Assume moreover that:
		\begin{enumerate}
			\item $\f$ is $\mathcal{C}^1$ in $t$,
			\item $\Phi$ is continuous on $[0,+\infty)\times U$.		
		\end{enumerate}
		Since $\theta$ is semi-positive we fix $c>0$ such that $\omega_t\geq c\theta$  for all $t$. For simplicity we again assume that $c=1$.  
		Let $\chi$ be a $\theta/2$-psh function with analytic singularities such that $\chi$ is smooth in $U$, $\chi=-\infty$ on $\partial U$, and $\sup_X\chi=0$. We will use this function in order to apply the classical maximum principle in $U$. The standard strategy is to replace $\f$ by $(1-\lambda)\f+\lambda\chi$. Nevertheless, the time derivative $\dot{\f_t}$ may blow up as $t\to 0$ so we need another auxiliary function. Let $\rho\in \PSH(X,\theta/2)$ be the unique solution to    
		\begin{align}
		(\theta/2+\dc\rho)^n=e^{{\rho}}d\mu,
		\end{align}
		normalized by $\sup_X\rho=0$, where $d\mu=\prod_{i}|f_i|^{-2a_i}dV$. It follows from~\cite[Corollary 3.5]{dang2021continuity} that $\rho$ is continuous in $U$.
		
		Fix $0<T<+\infty$. For $\e,\lambda>0$ small enough we set
		$$
		w(t,x):=(1-\lambda)\f(t,x)+\lambda(\rho(x)+\chi(x))-\Phi(t,x)-3\e t, \quad \forall\;(t,x)\in (0,T)\times X.
		$$ 
		By Lemma~\ref{12}, this function is upper semi-continuous on $[0,T]\times U$. 
		Since ${\rho+\chi}$ is a $\theta$-psh function which is continuous in $U$ and tends to $-\infty$ on $\partial U$, 
		%we have that $w$ tends to $-\infty$  on $\partial U$.   It thus follows that
		the function $w$ attains its maximum at some point $(t_0,x_0)\in [0,T]\times U$. 
		
		We want to show that $w(t_0,x_0)\leq 0$. Assume by contradiction that it is not the case i.e., $w(t_0,x_0)>0$ with $t_0>0$. The set 
		$$
		K:=\{x\in U: w(t_0,x)=w(t_0,x_0)\}
		$$
		is a compact subset of $U$ since $w(t_0,x)$ tends to $-\infty$ as $x\rightarrow\partial U$. 
		The classical maximum principle ensures that for all $x\in K$,
		$$
		(1-\lambda)\partial_{t}\f(t_0,x)\geq \partial_{t}^{-}\Phi(t_0,x)+3\e.   
		$$
		The partial derivative $\partial_{t}\f(t,x)$ is continuous on $U$ by assumption. 
		Since the function $t\mapsto \Phi(t,x)$ is locally uniformly semi-concave, for any $t\in(0,T)$,
		the left derivative  $\partial_t^-\Phi(t,\cdot)$ is upper semi-continuous in $\Omega$ (see Proposition~\ref{t_der}). We can thus find $\eta>0$ small enough that, by introducing the open set containing $K$,
		$$D:=\{x\in U:w(t_0,x)>w(t_0,x_0)-\eta \}\Subset U.$$
		We have for all $x\in D$
		\begin{align}
		(1-\lambda)\partial_{t}\f(t_0,x)>\partial_{t}^{-}\Phi(t_0,x)+2\e.
		\end{align}  
		Set $u:=(1-\lambda)\f(t_0,\cdot)+\lambda(\rho+\chi)$ and $v=\Phi(t_0,\cdot)$. Since $\f$ is a pluripotential solution to~\eqref{fbg}, using Lemma~\ref{logconcave} we infer that
		\begin{align*}
		(\omega_{t_0}+\dc u)^n&\geq [(1-\lambda)(\omega_{t_0}+\dc\f_{t_0})+\lambda (\theta/2+\dc\rho)]^n\\
		&\geq e^{(1-\lambda)(\partial_t\f(t_0,\cdot)+\f(t_0,\cdot))+\lambda\rho}d\mu
		\\&\geq e^{\partial_t^-\Phi(t_0,\cdot)+\varepsilon+(1-\lambda)\f(t_0,\cdot)+\lambda(\rho+\chi)}d\mu
		.
		\end{align*}
		%in the weak sense of measures in $D$. 
		On the other hand, $\Phi$ is a pluripotential solution to \eqref{fbg}, hence
		$$
		(\omega_{t_0}+\dc v)^n\leq e^{\partial_{t}^-\Phi(t_0,\cdot)+\Phi(t_0,\cdot)}d\mu
		$$ in the weak sense of measures in $D$. The last two inequalities yield
		\begin{align*}
		(\omega_{t_0}+\dc u)^n\geq e^{u-v+\e}(\omega_{t_0}+\dc v)^n.
		\end{align*}
		We then repeat the arguments as in the proof of Proposition~\ref{comparison1} to obtain a contradiction. Therefore, we must have $t_0=0$, hence
		\begin{align*}
		(1-\lambda)\f+\lambda (\rho+\chi)-\Phi-3\e t\leq \lambda\sup_X\left((\rho+\chi)-\f_0\right),
		\end{align*} 
		in $[0,T]\times U$. Letting $\lambda\rightarrow 0$ we obtain $\f\leq \Phi+3\e t$ in $[0,T]\times U$, hence in $[0,T]\times X$. We thus finish the proof by letting $\e\rightarrow 0$ and $T\rightarrow +\infty$.
		
		\medskip
		\noindent{\it Step 4.2.} \label{step42} We next remove the continuity assumption on $\Phi$ in Step~\hyperref[step41]{4.1}.
		
		Fixing $s>0$ small enough, we set
		\begin{align*}
		u^s(t,x):= e^{-s}\f_t(x)+(1-e^{-s})\psi_{KE}(x)+h(s)
		\end{align*}
		where $h$ is defined as in the proof of Theorem~\ref{thm_conv}. We observe that
		\[\omega_{t+s}=e^{-s}\omega_t+ (1-e^{-s})\theta, \quad\forall\; t\in[0,+\infty), \]
		hence
		\begin{align*}
		(\omega_{t+s}+\dc u^s)^n&=\left[e^{-s}(\omega_t+\dc\f_t)+(1-e^{-s})(\theta+\dc\psi_{KE}) \right]^n\\
		&\geq e^{e^{-s}(\partial_\tau \f_t+\f_t)+(1-e^{-s})\psi_{KE}}d\mu
		\end{align*}
		where the last inequality follows from Lemma~\ref{logconcave}. Since $h(s)\leq 0$ for $s>0$ we have 
		\[(\omega_{t+s}+\dc u^s)^n\geq e^{\partial_\tau{u^s}+u^s}d\mu. \] 
		On the other hand, 
		\[(\omega_{t+s}+\dc v^s)^n= e^{\partial_\tau{v^s}+v^s}d\mu, \] where $v^s(t,x):=\Phi(t+s,x)$ for $(t,x)\in (0,+\infty)\times X$. 
		By Lemma~\ref{lem413} we have $u^s(0,x)\leq v^s(0,x)$ for all $x\in X$. Since 
		$v^s$ is continuous on $[0,+\infty)\times U$, it follows from Step \hyperref[step41]{4.1} that $$u^s(t,x)\leq v^s(t,x),\quad (t,x)\in [0,+\infty)\times X.$$
		Letting $s\to 0$ we thus obtain $\f\leq \Phi$ on $[0,+\infty)\times X$. 
		
		\begin{lemma}\label{lem413}
			For all $(t,x)\in (0,T)\times X$,
			\begin{equation}\label{424}
			\Phi_t(x)\geq e^{-t}\f_0(x) +(1-e^{-t})\psi_{KE}(x)+h(t),
			\end{equation}
			where $h$ is the unique solution to the ODE: $h'(t)+h(t)=\log(1-e^{-t})$, $h(0)=0$.
		\end{lemma}
		\begin{proof}
			Fix $\e>0$, and consider
			\[w^\e(t,x)= e^{-t}\Phi_\e+(1-e^{-t})\psi_{K_E}+h(t). \]
			A direct computation shows that
			\begin{align*}
			(\omega_{t+\e}+\dc w^{\e})^n&=\left(e^{-t}(\omega_\e+\dc\Phi_\e)+(1-e^{-t})(\theta+\dc\psi_{KE}) \right)^n\\
			&\geq e^{\log(1-e^{-t})+\psi_{KE}}\mu.
			\end{align*}
			where we have used $\omega_\e+\dc\Phi_\e\geq 0$. Since $h'(t)+h(t)=n\log(1-e^{-t})$ we have
			\[(\omega_{t+\e}+\dc w^\e)^n\geq e^{\partial_{t}w^\e+w^\e}\mu. \]
			It is also clear from the definition that $w^\e(t,\cdot)$ converges in $L^1(X)$ to $w^\e(0,\cdot)=\Phi_\e$ as $t\to 0^+$.
			On the other hand, $w^\e$ is $\mathcal{C}^1$ in $t$ and $\Phi_{t+\e} $ is continuous on $[0,+\infty)\times U$. We can thus apply Step~\hyperref[step41]{4.1} to obtain $w^\e(t,x)\leq \Phi(t+\e,x)$. 
			The proof follows by letting $\e\to 0$.	 
		\end{proof}
		\noindent{\it Step 4.3.} We are now ready to treat the general case by removing the extra assumption on $\f$. 
		
		For $s>0$ near 1 we set, for any $(t,x)\in (0,T)\times X$
		\begin{align*}
		V^s(t,x):=\frac{\alpha_s}{s}\f(ts,x)+(1-\alpha_s)v_1(x)-C|s-1|e^{-t},
		\end{align*}
		where $\alpha_s$ is defined as in~\eqref{alpha}, and $v_1$ is a solution to~\eqref{eq1_bh}. For $C>0 $ large enough, the proof of Proposition~\ref{thm_lip2} ensures that $V^s$ is a subsolution to~\eqref{fbg} that satisfies $V^s(0,\cdot)\leq \f_0$ on $X$.
		Let $\{\eta_{\e}\}_{\e>0}$ be a family of smoothing kernels in $\mathbb{R}$ approximating the Dirac mass $\delta_0$. For $\e>0$ small enough we define
		\begin{align*}
		\f^{\e}(t,x):=\int_{\mathbb{R}}V^s(t,x)\eta_{\e}(s-1)ds
		\end{align*}
		We proceed as in the proof of Theorem~\ref{thm_unique} to show that $\f^\e-O(\e)$ is  again a subsolution and apply the previous step to conclude. %Indeed, since $\f^\e-O(\e)$ is $\mathcal{C}^1$ in $t$ and $\f^\e$ converges pointwise to $\f$ as $\e\to 0$, Step \hyperref[step42]{4.2} ensures  that $\f^\e\leq \Phi$ on $(0,+\infty)\times X$. The conclusion follows by letting $\e\to 0$.
		\medskip
		
		\noindent{\bf Step 5: the long-term behavior of the flow.}
		It remains to establish the convergence at $t=+\infty$. We have seen that $$u(t,x):=e^{-t}\f_0+(1-e^{-t})\psi_{KE}(x)+h(t)$$ is a subsolution to~\eqref{fbg}. The comparison principle (see Step~\hyperref[step4]{4}) yields for any $t>0$, $x\in X$,
		\begin{align*}
		\psi_{KE}(x)-C(t+1)e^{-t}\leq u(t,x)\leq \f(t,x)
		\end{align*}
		for some uniform constant $C>0$.
		
		For the upper bound, since $\tilde{\theta}=\theta+\dc\chi_0$ is a K\"ahler current we can fix a constant $A>0$ such that $\omega_0\leq (1+A)\tilde{\theta}$ on $\Omega$, thus $\omega_t\leq (1+Ae^{-t})\tilde{\theta}$ for all $t$. Set 
		\begin{align*}
		v(t,x):=(1+Ae^{-t})\psi_{KE}(x)+Be^{-t}
		\end{align*}
		where $B$ is chosen so that $v_0\geq \f_0$. 
		Thus the function $v$ is a supersolution to the Cauchy problem for the parabolic equation
		\begin{align*}
		((1+Ae^{-t})\tilde{\theta}+dd^cv_t)^n\leq e^{\dot{v}_t+v_t+nAe^{-t}}
		\end{align*}
		with initial data $\f_0$, while $w(t,x):=\f(t,x)-nAe^{-t}$ is a subsolution to this equation since 
		\begin{align*}
		((1+Ae^{-t}) \tilde{\theta}+\dc w_t)^n\geq (\omega_t+\dc\f_t)^n=e^{\dot{\f}_t+\f_t}fdV=e^{\dot{w}_t+w_t+nAe^{-t}}fdV.
		\end{align*}
		The comparison principle  thus yields 
		\begin{align*}
		\f(t,x)\leq (1+Ae^{-t}) \psi_{KE}(x)+C'e^{-t},
		\end{align*}
		as desired.
	\end{proof}

	\bibliographystyle{plain}
	\bibliography{bibfile}

\end{document}